\documentclass[]{scrartcl}

\usepackage[USenglish]{babel}
\usepackage{csquotes}

\usepackage[a4paper,top=27mm,bottom=20mm,inner=25mm,outer=20mm]{geometry}

\usepackage[%
  backend=bibtex,bibencoding=ascii,
  style=numeric-comp,
  giveninits=true, uniquename=init, 
  natbib=true,
  url=true,
  doi=true,
  isbn=false,
  backref=false,
  maxnames=99,
  ]{biblatex}
\usepackage{amsmath}
\allowdisplaybreaks
\numberwithin{equation}{section}
\usepackage{amssymb}
\usepackage{commath}
\usepackage{mathtools}
\usepackage{bbm}
\usepackage{nicefrac}
\usepackage{subdepth}
\usepackage{comment}

\usepackage{siunitx}
\usepackage{algorithm}
\usepackage{algpseudocode}
\usepackage[section]{placeins}

\usepackage{amsthm}

\usepackage[plainpages=false,pdfpagelabels,hidelinks,unicode]{hyperref}

\usepackage{cleveref}

\usepackage{thmtools}
\usepackage{etoolbox}
\makeatletter
\patchcmd{\thmt@setheadstyle}
 {\bgroup\thmt@space}
 {\thmt@space}
 {}{}
\patchcmd{\thmt@setheadstyle}
 {\egroup\fi}
 {\fi}
 {}{}
\makeatother
\declaretheoremstyle[
  bodyfont=\normalfont\itshape,
  headformat=\NAME\ \NUMBER\NOTE,
]{myplain}
\declaretheoremstyle[
  headformat=\NAME\ \NUMBER\NOTE,
]{mydefinition}
\newcommand{\envqed}{{\lower-0.3ex\hbox{$\triangleleft$}}}
\declaretheorem[style=myplain,numberwithin=section]{theorem}

\declaretheorem[style=mydefinition,numberlike=theorem,qed=\envqed]{definition}
\declaretheorem[style=mydefinition,numberlike=theorem,qed=\envqed]{remark}

\usepackage{color}
\usepackage{graphicx}
\usepackage[small]{caption}
\usepackage{subcaption}

\begingroup\expandafter\expandafter\expandafter\endgroup
\expandafter\ifx\csname pdfsuppresswarningpagegroup\endcsname\relax
\else
\fi

\usepackage{booktabs}
\usepackage{rotating}
\usepackage{multirow}

\usepackage{enumitem}

\usepackage{ifluatex}
\ifluatex
  \usepackage[no-math]{fontspec}
\else
  \usepackage[T1]{fontenc}
\fi
\usepackage{newpxtext,newpxmath}

\let\epsilon\varepsilon
\let\phi\varphi
\let\rho\varrho

\newcommand{\equalcontrib}{\textsuperscript{$\star$}}

\newcommand{\orcid}[1]{ORCID:~\href{https://orcid.org/#1}{#1}}
\usepackage{authblk}

\newenvironment{keywords}{\par\textbf{Key words.}}{\par}
\newenvironment{AMS}{\par\textbf{AMS subject classification.}}{\par}

\title{GPU-Accelerated Energy-Conserving Methods for the Two-Dimensional Hyperbolized Serre--Green--Naghdi Equations}

\author[1,2]{Collin~Wittenstein\equalcontrib\thanks{\orcid{0009-0006-8591-278X}}}

\affil[1]{Department of Electrical Engineering and Computer Science, Massachusetts Institute of Technology, Cambridge, MA 02139, USA}
\affil[2]{Institute of Mathematics, Johannes Gutenberg University Mainz, 55099 Mainz, Germany}

\author[2]{Vincent~Marks\equalcontrib\thanks{\orcid{0009-0007-4973-5074}}}

\author[3]{Mario Ricchiuto\thanks{\orcid{0000-0002-1679-7339}}}
\affil[3]{INRIA, U. Bordeaux, CNRS, Bordeaux INP, IMB, UMR 5251, 200 Av. de la Vieille Tour, 33400 Talence, France}

\author[2]{Hendrik~Ranocha\thanks{\orcid{0000-0002-3456-2277}}}

\date{May 26, 2026}

\makeatletter
\hypersetup{
  pdfauthor={Collin Wittenstein, Vincent Marks, Mario Ricchiuto, Hendrik Ranocha},
  pdftitle={GPU-Accelerated Energy-Conserving Methods for the Two-Dimensional Hyperbolized Serre--Green--Naghdi Equations},
  pdfsubject={Dispersive wave equations, GPU computing},
  pdfkeywords={SBP operators, Serre--Green--Naghdi, energy conservation, Julia, GPU}
}
\makeatother

\begin{document}
\maketitle

{\let\thefootnote\relax\footnotetext{$\star$These authors contributed equally to this work.}}

\begin{abstract}
  \noindent
  We develop energy-conserving numerical methods for a two-dimensional hyperbolic
approximation of the Serre--Green--Naghdi equations with variable bathymetry
and either periodic or reflecting boundary conditions. The hyperbolic
formulation avoids the costly inversion of an elliptic operator present in the
classical model. Our schemes combine split forms with summation-by-parts (SBP)
operators to construct semi-discretizations that conserve the total water mass
and the total energy. We provide analytical proofs of these conservation
properties and also verify them numerically. While the framework is general,
our implementation focuses on second-order finite-difference SBP operators. The
methods are implemented in Julia for CPU and GPU architectures (AMD and NVIDIA)
and achieve substantial speedups on modern accelerators. We validate the
approach through convergence studies based on solitary-wave and
manufactured-solution tests, and by comparisons to analytical, experimental,
and existing numerical results. All source code to reproduce our results is
available online.

\end{abstract}

\begin{keywords}
  summation-by-parts operators,
  dispersive wave equations,
  finite difference methods,
  structure-preserving methods,
  energy-conserving methods,
  entropy-stable methods
\end{keywords}

\begin{AMS}
  65M06, 
  65M12,  
  65M20 
\end{AMS}

\section{Introduction}

Dispersive free-surface waves are a fundamental feature of coastal and riverine
hydrodynamics. They arise in diverse settings, ranging from tsunami
propagation~\cite{nhess-13-1507-2013} to wave transformation over complex
bathymetry in estuaries and engineered
channels~\cite{Treske1994,Chassagne_Filippini_Ricchiuto_Bonneton_2019,JOUY2024170}.
To capture these effects in practical large-scale simulations, many operational
hazard-assessment codes rely on depth-averaged Boussinesq-type
models~\cite{Kazolea2024FullNI}. Within this class of equations, the
Serre--Green--Naghdi (SGN)
system~\cite{Serre1953,GN76} is of particular
interest, as it combines full nonlinearity with weakly dispersive
characteristics.

The aim of this work is to develop structure-preserving numerical schemes in
two space dimensions for the SGN equations with variable bathymetry,
accommodating both periodic and reflecting boundary conditions. For clarity, we
summarize the main physical variables used throughout this work in
Table~\ref{table:notation} and visualize them in
Figure~\ref{fig:Sketch_variables}.

\begin{figure}[htbp]
  \centering
  \includegraphics[width=0.6\textwidth]{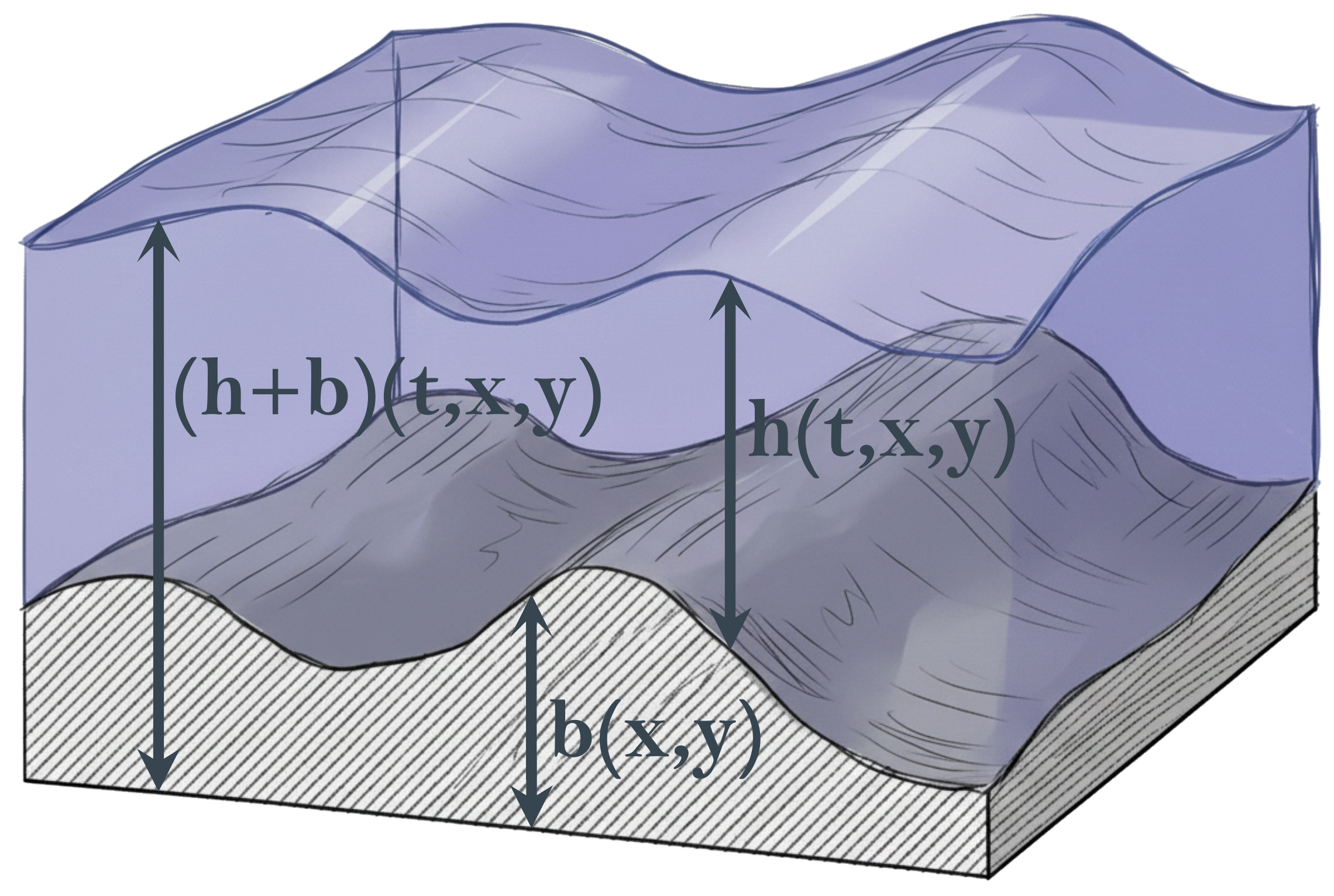}
  \caption{Sketch of the variables: total water height $h(t,x,y) + b(x,y)$, water height above the bathymetry $h(t,x,y)$, and bathymetry $b(x,y)$.}
  \label{fig:Sketch_variables}
\end{figure}

\begin{table}[!htbp]
  \centering
  \caption{Notation used in this work.}
  \label{table:notation}
  \begin{tabular}{ll}
    \toprule
    \textbf{Symbol} & \textbf{Description}                          \\
    \midrule
    $h + b$         & total water height (free-surface elevation)   \\
    $h$             & water height above the bathymetry             \\
    $b$             & bathymetry                                    \\
    $u, v$          & velocity components in $x$- and $y$-direction \\
    $g$             & gravitational acceleration                    \\
    \bottomrule
  \end{tabular}
\end{table}

Two main formulations of the SGN equations have been employed in the literature
for numerical treatment: the classical form (equation \eqref{eq:2D_classic_SGN} below)
involving the inversion of a nonlinear elliptic operator~\cite{Maria_article,gavrilyuk20242d}
and a hyperbolic approximation (equation
\eqref{eq:2D_hyperbolic_SGN} below)~\cite{Favrie_2017,Busto2021}, see also
\cite{guermond2022well,guermond2022hyperbolic}.
Such hyperbolizations can be implemented efficiently on modern hardware
such as GPUs and have been used in the context of other dispersive wave
models as well~\cite{escalante2019efficient}.
In this work, we consider the hyperbolization of~\cite{Busto2021}, which
includes a relaxation parameter~$\lambda$ and converges to the original
SGN system as $\lambda \to \infty$. It offers computational advantages on
relatively coarse meshes by avoiding the expensive solution of an elliptic
problem at each time step.

While the energy of the (hyperbolized) Serre--Green--Naghdi system is
conserved for smooth solutions, methods from the realm of hyperbolic
conservation laws suggest relaxing this requirement to energy stability
as a generalization of entropy stability~\cite{parisot2019entropy,noelle2022class}.
Some recent works have shown that numerical dissipation can lead to significant
underestimation of wave amplitudes on coarse meshes~\cite{JOUY2024170,ranocha2025structure}.
As long as only wave propagation is considered, without wave breaking or wetting and drying effects for which
energy dissipation is required (see, e.g., \cite{Kazolea_Ricchiuto_2018})
and positivity-preserving methods play an important role, discrete energy preservation
has very appealing features.
In this work, we focus
on energy-conserving numerical methods, similar to recent developments for
other dispersive wave models~\cite{antonopoulos2025bona,lampert2025structure,mitsotakis2021conservative}.

To construct numerical methods that conserve the energy at the semi-discrete
level, we utilize the framework of summation-by-parts (SBP) operators (see for
example~\cite{svard2014review,fernandez2014review}) and split
forms~\cite{FISHER2013353}. This approach enables a systematic derivation of
energy-conserving semi-discretizations by mimicking the properties employed at
the continuous level: SBP operators mimic integration by parts, while split
forms mimic the product and chain rules. Compared to existing work on
energy-conserving methods for the SGN equations~\cite{ranocha2025structure}, we
extend the methodology to the fully two-dimensional SGN system with variable
bathymetry and reflecting boundary conditions. For our implementation, we
mainly rely on the Julia packages
OrdinaryDiffEq.jl~\cite{DifferentialEquations.jl-2017} and
KernelAbstractions.jl~\cite{Churavy_KernelAbstractions_jl}, enabling
vendor-agnostic simulations on both CPU and GPU platforms.

The remainder of this article is organized as follows. In
Section~\ref{sect:split-sbp}, we review the techniques of split forms and SBP
operators that form the foundation of our spatial discretization.
Section~\ref{sect:sgn-review} provides an overview of the SGN equations in both
classical and hyperbolic approximation forms, along with their energy
conservation laws. In Section~\ref{sect:energy-conserving}, we develop
energy-conserving semi-discretizations for the two-dimensional hyperbolic SGN
system with variable bathymetry, treating both periodic and reflecting boundary
conditions.
In Section~\ref{sect:computing_solitary_waves}, we describe how we compute solitary wave solutions numerically.
We validate our implementation and present numerical experiments in
Section~\ref{sec:numerical_experiments}, including convergence studies, comparisons
with experimental data, and performance benchmarks across different hardware
platforms. Finally, we summarize our results and provide an outlook on future
work in Section~\ref{sect:conclusions}.

\section{Review of Split Forms and Summation-by-Parts Operators} \label{sect:split-sbp}

In this section, we review the techniques employed for the spatial
discretization of the SGN equations. We utilize split forms combined with
SBP operators to guarantee conservation of a discrete
energy. Following the method of lines approach, we first discretize in space
and subsequently perform the time integration.

\subsection{Split Forms}

In general, discrete derivative operators do not satisfy discrete analogs of
the product and chain rules~\cite{Ranocha_2018}. To address this limitation, we
employ split forms~\cite{FISHER2013353} that mimic those, thereby enabling us
to prove energy conservation solely through integration by parts.

For example, consider the classical one-dimensional shallow water equations with constant
bathymetry
\begin{equation} \label{eq:shallow_water}
  \begin{aligned}
     & h_t + (hu)_x = 0,                                   \\
     & (hu)_t + \left(hu^2 + \frac{1}{2}gh^2\right)_x = 0.
  \end{aligned}
\end{equation}
with $h$ the depth, $u$ the velocity, and $g$ the gravitational acceleration.
For smooth solutions, these equations satisfy the energy conservation law
\begin{equation}
  \biggl(\underbrace{\frac{1}{2}gh^2 + \frac{1}{2}hu^2}_{=E}\biggr)_t + \biggl(\underbrace{gh^2u + \frac{1}{2}hu^3}_{=F}\biggr)_x = 0.
\end{equation}
A split form of equation~\eqref{eq:shallow_water} is given by~\cite{ranocha2025structure,ranocha2017shallow}
\begin{equation}
  \begin{aligned}
     & h_t + h_x u + h u_x = 0,                                                                                         \\
     & hu_t + g(h^2)_x - ghh_x + \frac{1}{2}h(u^2)_x - \frac{1}{2}h_x u^2 + \frac{1}{2}(hu)_x u - \frac{1}{2}huu_x = 0.
  \end{aligned}
\end{equation}
Using integration by parts, we can verify that the energy is conserved for periodic or reflecting boundary conditions ($u|_{\partial\Omega} = 0$) as follows:
\begin{equation}
  \label{eq:energy_conservation_shallow_water}
  \begin{aligned}
    \frac{\dif}{\dif t} \int E \dif x
     & = \frac{\dif}{\dif t} \int \left(\frac{1}{2}gh^2 + \frac{1}{2}hu^2\right) \dif x                                                              \\
     & = \int \left(ghh_t + \frac{1}{2}h_t u^2 + huu_t\right) \dif x                                                                          \\
     & = \int \left[gh\left(-h_x u - hu_x\right) + \frac{1}{2}u^2\left(-h_x u - hu_x\right)\right] \dif x                                     \\
     & \quad - \int u\left[g(h^2)_x - ghh_x + \frac{1}{2}h(u^2)_x - \frac{1}{2}h_x u^2 + \frac{1}{2}(hu)_x u - \frac{1}{2}huu_x\right] \dif x \\
     & = \int \Bigl(-ghh_x u + ghh_x u - gh^2u_x - gu(h^2)_x - \frac{1}{2}u^3h_x + \frac{1}{2}h_x u^3                                       \\
     & \qquad - \frac{1}{2}u^2hu_x + \frac{1}{2}hu^2u_x - \frac{1}{2}hu(u^2)_x - \frac{1}{2}u^2(hu)_x\Bigr) \dif x
     = 0.
  \end{aligned}
\end{equation}

\subsection{Summation-by-Parts Operators}

SBP operators are discrete derivative operators that mimic
integration by parts and thereby enable the transfer of analytical results from
the continuous setting to discrete settings. An SBP operator consists of a
derivative operator and a mass matrix $M$ that defines a quadrature rule.
Originally proposed in the finite difference context~\cite{Kreiss1974,strand1994summation,carpenter1994time}, SBP
operators have a broad range of applications, including finite volume
\cite{Nordstrom2001}, continuous finite element~\cite{Hicken2016,hicken2020entropy,abgrall2020analysisI},
discontinuous Galerkin~\cite{Gassner_2013,carpenter2014entropy},
flux reconstruction~\cite{huynh2007flux,vincent2011newclass,ranocha2016summation},
active flux methods~\cite{eymann2011active,barsukow2025stability},
and meshfree schemes~\cite{hicken2025constructing,kwan2025robust}.

The definitions and notations described in this section are standard
in the SBP literature, e.g.,~\cite{svard2014review,fernandez2014review,Ranocha_2021}.
Here, we focus on one-dimensional SBP operators and extend them to multiple
dimensions via tensor products as usual for finite difference methods.

We discretize the spatial domain $[x_\mathrm{min}, x_\mathrm{max}]$ using a nodal collocation method.
Thus, we first introduce the grid $\boldsymbol{x} =
  (\boldsymbol{x}_1,\ldots,\boldsymbol{x}_\mathcal{N})^T$, where $x_\mathrm{min} = \boldsymbol{x}_1
  \leq \boldsymbol{x}_2 \leq \ldots \leq \boldsymbol{x}_\mathcal{N} = x_\mathrm{max}$.
The discrete representation $\boldsymbol{u}$ of a function
$u$ consists of its values at the grid points, i.e., $\boldsymbol{u}_i =
  u(\boldsymbol{x}_i)$.
Nonlinear operations such as $(\boldsymbol{u}^2)_i = \boldsymbol{u}_i^2$ are
evaluated componentwise.
Moreover, $\boldsymbol{1} = (1, \ldots, 1)^T$ denotes the vector of ones.
Finally, we use the notation
$\boldsymbol{e}_L=(1,0, \ldots, 0)^T$, $\boldsymbol{e}_R=(0, \ldots, 0,1)^T$.

\begin{definition}
  Given a grid $\boldsymbol{x}$, a $p$-th order accurate $i$-th derivative matrix $D^{(i)}$ is a matrix that satisfies
  \begin{equation}
    \forall k \in\{0, \ldots, p\}: \quad D^{(i)} \boldsymbol{x}^k=k(k-1) \ldots(k-i+1) \boldsymbol{x}^{k-i},
  \end{equation}
  with the convention $\boldsymbol{x}^0=\boldsymbol{1}$ and $0 \boldsymbol{x}^k=\boldsymbol{0}$. We say $D^{(i)}$ is consistent if $p \geq 0$.
\end{definition}

\begin{definition}\label{def:singledim_sbp}
  A first-derivative SBP operator consists of a grid $\boldsymbol{x}$, a consistent first-derivative matrix $D$, and a symmetric and positive-definite matrix $M$ such that
  \begin{equation}
    M D+D^T M=\boldsymbol{e}_R \boldsymbol{e}_R^T-\boldsymbol{e}_L \boldsymbol{e}_L^T.
  \end{equation}
  We refer to $M$ as a mass matrix or norm matrix. In the periodic case, we define $\boldsymbol{e}_{L/R}= \boldsymbol{0}$ such that
  \begin{equation}
    M D+D^T M=0.
  \end{equation}
\end{definition}

This definition yields a discrete version of integration by parts as
\begin{equation}
  \begin{aligned}
    \underbrace{\boldsymbol{u}^T M D \boldsymbol{v}+\boldsymbol{u}^T D^T M \boldsymbol{v}}_{\rotatebox{90}{$\approx$}}
     & =
    \underbrace{\boldsymbol{u}^T \boldsymbol{e}_R \boldsymbol{e}_R^T \boldsymbol{v}-\boldsymbol{u}^T \boldsymbol{e}_L \boldsymbol{e}_L^T \boldsymbol{v}}_{\rotatebox{90}{$\approx$}}
    \\
    \overbrace{\int_{x_{\min}}^{x_{\max}} u\,(\partial_x v)\,\dif x
      + \int_{x_{\min}}^{x_{\max}} (\partial_x u)\,v\,\dif x}
     & =
    \overbrace{u(x_{\max})\,v(x_{\max}) - u(x_{\min})\,v(x_{\min})}.
  \end{aligned}
\end{equation}

Since we are dealing with a two-dimensional problem, we also need to consider
SBP operators in multiple dimensions. Following~\cite{Ranocha2017ComparisonOS},
we provide the following definition:

\begin{definition}[Multidimensional SBP operator]\label{def:multidim_sbp}
  An SBP operator on a $d$-dimensional element $\Omega$ with order of accuracy $p \in \mathbb{N}$ consists of the following components:
  \begin{itemize}
    \item \textbf{Derivative operators} $D_j$, $j \in \{1, \ldots, d\}$, approximating the partial derivative in the $j$-th coordinate direction. These are required to be exact for polynomials of degree $\leq p$.

    \item \textbf{Mass matrix} $M$, approximating the $L_2$ scalar product on $\Omega$ via
          \begin{equation}
            \boldsymbol{u}^T M \boldsymbol{v} = \langle\boldsymbol{u}, \boldsymbol{v}\rangle_M \approx \langle u, v \rangle_{L_2(\Omega)} = \int_{\Omega} u v\,\dif x,
          \end{equation}
          where $u, v$ are functions on $\Omega$ and $\boldsymbol{u}, \boldsymbol{v}$ their approximations in the SBP basis (i.e., their projections onto the grid).

    \item \textbf{Restriction operator} $R$, performing interpolation of functions from the volume $\Omega$ to the boundary $\partial\Omega$.

    \item \textbf{Boundary mass matrix} $B$, approximating the $L_2$ scalar product on $\partial\Omega$ via
          \begin{equation}
            \boldsymbol{u}_B^T B \boldsymbol{v}_B = \langle\boldsymbol{u}_B, \boldsymbol{v}_B\rangle_B \approx \langle u_B, v_B \rangle_{L_2(\partial\Omega)} = \int_{\partial\Omega} u_B v_B\,\dif s,
          \end{equation}
          where $u_B, v_B$ are functions on $\partial\Omega$ and $\boldsymbol{u}_B, \boldsymbol{v}_B$ their approximations in the SBP basis.

    \item \textbf{Multiplication operators} $N_j$, $j \in \{x, y, z, \ldots\}$, performing multiplication of functions on the boundary $\partial\Omega$ with the $j$-th component $n_j$ of the outer unit normal vector. Thus, if $\boldsymbol{u}$ is the approximation of a function $u|_{\Omega}$ in the SBP basis, then $R\boldsymbol{u}$ is the approximation of $u|_{\partial\Omega}$ on the boundary, and $N_j R \boldsymbol{u}$ is the approximation of $n_j u|_{\partial\Omega}$, where $n_j$ is the $j$-th component of the outer unit normal at $\partial\Omega$.

    \item The restriction and boundary operators satisfy
          \begin{equation}
            \boldsymbol{u}^T R^T B N_j R \boldsymbol{v} \approx \int_{\partial\Omega} u v n_j\,\dif s,
          \end{equation}
          where $n_j$ is the $j$-th component of the outer unit normal vector $n$, and this approximation has to be exact for polynomials of degree $\leq p$.

    \item Finally, the \textbf{SBP property}
          \begin{equation}
            M D_j + D_j^T M = R^T B N_j R
          \end{equation}
          has to be fulfilled, mimicking the divergence theorem on a discrete level
          \begin{equation}
            \begin{aligned}
               & \int_{\Omega} u\,(\partial_j v)\,\dif x + \int_{\Omega} (\partial_j u)\,v\,\dif x                                                           \\
               & \quad \approx \boldsymbol{u}^T M D_j \boldsymbol{v} + \boldsymbol{u}^T D_j^T M \boldsymbol{v}
               = \boldsymbol{u}^T R^T B N_j R \boldsymbol{v}
              \approx \int_{\partial\Omega} u v n_j\,\dif s.
            \end{aligned}
          \end{equation}
  \end{itemize}
  In one space dimension, the indices of the derivative and multiplication operators $D_x$, $N_x$ are typically dropped. Furthermore, the boundary matrix is the $2 \times 2$ identity matrix $B = \operatorname{diag}(1, 1)$, and multiplication with the outer normal is given by $N = \operatorname{diag}(-1, 1)$, i.e.,
  $B N = \boldsymbol{e}_R \boldsymbol{e}_R^T - \boldsymbol{e}_L \boldsymbol{e}_L^T$.
\end{definition}

Further information about SBP operators in multidimensional settings can be
found in~\cite{Hicken2016,Ranocha_MastersThesis_2016}. The concrete operators employed in this work are specified in Remark~\ref{rem:concrete_sbp}.

\section{Review of the Serre--Green--Naghdi Equations} \label{sect:sgn-review}

In this section, we review the two-dimensional SGN
equations in both their classical form and their hyperbolic approximation
with variable bathymetry. For each formulation, we present the corresponding
energy conservation law.

\subsection{Classical Form with Variable Bathymetry}

The classical two-dimensional SGN equations with variable bathymetry~\cite{gavrilyuk20242d} read with similar notation to \eqref{eq:shallow_water}
\begin{equation}
  \label{eq:2D_classic_SGN}
  \begin{aligned}
    & h_t + \frac{\partial (hu)}{\partial x} + \frac{\partial (hv)}{\partial y} = 0, \\
    & (hu)_t + \frac{\partial}{\partial x}\left(hu^2
      + \frac{g h^2}{2} + \frac{h^2}{3}\left(\ddot{h}
      + \frac{3}{2} \ddot{b}\right)\right) + \frac{\partial (huv)}{\partial y}
      + \left(g h + h\left(\ddot{b} + \frac{1}{2} \ddot{h}\right)\right) \frac{\partial b}{\partial x} = 0, \\
    & (hv)_t + \frac{\partial (hvu)}{\partial x} + \frac{\partial}{\partial y}\left(hv^2
      + \frac{g h^2}{2} + \frac{h^2}{3}\left(\ddot{h}
      + \frac{3}{2} \ddot{b}\right)\right)
      + \left(g h + h\left(\ddot{b} + \frac{1}{2} \ddot{h}\right)\right) \frac{\partial b}{\partial y} = 0,
  \end{aligned}
\end{equation}
where now $(u,v)$ denote the $x$ and $y$ components of the velocity, and $b$ is the bathymetry level.
Here, the dot notation denotes the material derivative along the velocity
\begin{equation}
  \dot{f} = \frac{\partial f}{\partial t} + u \frac{\partial f}{\partial x}
          + v \frac{\partial f}{\partial y}
\end{equation}
for any scalar function $f(t,x,y)$, and two dots denote the
corresponding second material derivative.
For $b_t = 0$, the system satisfies the energy conservation law
\begin{equation}
  (h \mathcal{E})_t + \frac{\partial}{\partial x}\left(\left(h \mathcal{E}
+ \frac{g h^2}{2} + \frac{h^2}{3}\left(\ddot{h}
+ \frac{3}{2} \ddot{b}\right)\right) u\right)
+ \frac{\partial}{\partial y}\left(\left(h \mathcal{E}
+ \frac{g h^2}{2} + \frac{h^2}{3}\left(\ddot{h}
+ \frac{3}{2} \ddot{b}\right)\right) v\right) = 0
\end{equation}
with
\begin{equation}
  \mathcal{E} = \frac{u^2 + v^2}{2} + g\left(\frac{h}{2} + b\right)
    + \frac{1}{6}\left(\dot{h} + \frac{3}{2} \dot{b}\right)^2
    + \frac{1}{8} \dot{b}^2.
\end{equation}

\subsection{Hyperbolic Approximation with Variable Bathymetry}
The classical SGN equations require solving a time-independent elliptic system at each time step to obtain
the velocity at the new time level~\cite{gavrilyuk20242d}. This
introduces considerable overhead and is not well suited for certain implementations,
notably on GPUs. An alternative approach is
to use a hyperbolic approximation of the SGN equations~\cite{Favrie_2017}.
In this reference, the authors propose a new model obtained
by perturbing the Lagrangian of the original SGN equations, following the so-called
augmented Lagrangian method. The resulting Lagrangian involves a larger number of unknowns,
but the corresponding Euler--Lagrange equations constitute a hyperbolic PDE system.
The variational nature of this system also guarantees the existence of an energy conservation law
for an augmented version of the SGN energy.
Previous work has derived one- and two-dimensional hyperbolic approximations of the SGN
equations~\cite{Busto2021,Favrie_2017} and rigorously established their
properties~\cite{Duch_ne_2019}.

%

The two-dimensional hyperbolic SGN equations with variable bathymetry introduced in~\cite{Busto2021} can be written as
\begin{equation}
  \label{eq:2D_hyperbolic_SGN}
  \begin{aligned}
    \frac{\partial h}{\partial t} + \frac{\partial (hu)}{\partial x} + \frac{\partial (hv)}{\partial y}                                                                                                                 & = 0,                                      \\
    \frac{\partial (hu)}{\partial t} + \frac{\partial}{\partial x}\left(hu^2 + \frac{1}{2}gh^2 + hp\right) + \frac{\partial (huv)}{\partial y} + \left(gh + \frac{3}{2}\frac{h}{\eta}p\right)\frac{\partial b}{\partial x} & = 0,                                      \\
    \frac{\partial (hv)}{\partial t} + \frac{\partial (hvu)}{\partial x} + \frac{\partial}{\partial y}\left(hv^2 + \frac{1}{2}gh^2 + hp\right) + \left(gh + \frac{3}{2}\frac{h}{\eta}p\right)\frac{\partial b}{\partial y} & = 0,                                      \\
    \frac{\partial (h\eta)}{\partial t} + \frac{\partial (h\eta u)}{\partial x} + \frac{\partial (h\eta v)}{\partial y} + \frac{3}{2}h\left(u\frac{\partial b}{\partial x} + v\frac{\partial b}{\partial y}\right)        & = hw,                                     \\
    \frac{\partial (hw)}{\partial t} + \frac{\partial (hwu)}{\partial x} + \frac{\partial (hwv)}{\partial y}                                                                                                            & = \lambda\left(1 - \frac{\eta}{h}\right), \\
    \frac{\partial b}{\partial t}                                                                                                                                                                                     & = 0,
  \end{aligned}
\end{equation}
where
\begin{equation}
  p(h, \eta) = \frac{\lambda}{3}\frac{\eta}{h}\left(1 - \frac{\eta}{h}\right).
\end{equation}
This system satisfies the energy conservation law
\begin{equation}
  \label{eq:energy_conservation_2D_hyperbolic_SGN}
  \frac{\partial E}{\partial t}
  + \frac{\partial}{\partial x}\left(\left(E + P\right) u\right)
  + \frac{\partial}{\partial y}\left(\left(E + P\right) v\right) = 0
\end{equation}
with
\begin{equation}
  \begin{aligned}
    E          & = h\left(\frac{u^2 + v^2}{2} + e\right), \\
    e          & = \frac{1}{6}w^2 + \frac{g}{2}(h + 2b) + \frac{\lambda}{6}\left(\frac{\eta}{h} - 1\right)^2, \\
    P          & = \frac{gh^2}{2} + \frac{\lambda}{3}\eta\left(1 - \frac{\eta}{h}\right).
  \end{aligned}
\end{equation}
In the limit $\lambda \to \infty$, we (formally) obtain
\begin{equation}
  \eta = h, \quad w = -h \cdot (\partial_x u + \partial_y v) + \frac{3}{2} \cdot (u\partial_x b + v\partial_y b).
\end{equation}
\begin{remark}
Note that the hyperbolic approximation presented here is based on
the mild-slope approximation. In~\cite{Busto2021}, the augmented
Lagrangian used to derive the hyperbolic formulation is simplified
by assuming time-independent bathymetry ($b_t = 0$) and small gradients of $b$.
See also~\cite{ranocha2025structure} for a discussion and an alternative formulation.
\end{remark}

\section{Energy-Conserving Methods for the 2D Hyperbolization} \label{sect:energy-conserving}
In this section, we extend existing energy-conserving methods for the SGN
equations in one space dimension~\cite{ranocha2025structure} to the two-dimensional
hyperbolic SGN equations with variable bathymetry for periodic and reflecting
boundary conditions.

\subsection{Periodic Boundary Conditions}

For periodic boundary conditions, we propose the semi-discretization
\begin{equation} \label{eq:sd_periodic_2D_SGN}
   \begin{aligned}
     & \boldsymbol{h}_t + \boldsymbol{u}D_x \boldsymbol{h} + \boldsymbol{h}D_x\boldsymbol{u} + \boldsymbol{v}D_y \boldsymbol{h} + \boldsymbol{h}D_y\boldsymbol{v} = \boldsymbol{0},                                                                                                                                                        \\[0.5em]
     & \boldsymbol{h}\boldsymbol{u}_t + gD_x\left(\boldsymbol{h}(\boldsymbol{h}+\boldsymbol{b})\right) - g(\boldsymbol{h}+\boldsymbol{b})D_x \boldsymbol{h} + \frac{1}{2}\boldsymbol{h}D_x(\boldsymbol{u}^2) - \frac{1}{2}\boldsymbol{u}^2 D_x \boldsymbol{h} + \frac{1}{2}\boldsymbol{u}D_x(\boldsymbol{h}\boldsymbol{u}) - \frac{1}{2}\boldsymbol{h}\boldsymbol{u}D_x \boldsymbol{u}                                                                                                       \\
     & \quad + \frac{1}{2}D_y(\boldsymbol{h}\boldsymbol{u}\boldsymbol{v}) - \frac{1}{2} \boldsymbol{u}\boldsymbol{v}D_y\boldsymbol{h} + \frac{1}{2} \boldsymbol{h}\boldsymbol{v}D_y \boldsymbol{u} - \frac{1}{2} \boldsymbol{h}\boldsymbol{u}D_y \boldsymbol{v}                                                                                                            \\
     & \quad + \frac{\lambda}{6}\frac{\boldsymbol{\eta}^{2}}{\boldsymbol{h}^{2}}D_x \boldsymbol{h} + \frac{\lambda}{3}D_x \boldsymbol{\eta} - \frac{\lambda}{3}\frac{\boldsymbol{\eta}}{\boldsymbol{h}}D_x \boldsymbol{\eta} - \frac{\lambda}{6}D_x\left(\frac{\boldsymbol{\eta}^{2}}{\boldsymbol{h}}\right) + \frac{\lambda}{2}\left(1-\frac{\boldsymbol{\eta}}{\boldsymbol{h}}\right)D_x \boldsymbol{b} = \boldsymbol{0}, \\[0.5em]
     & \boldsymbol{h}\boldsymbol{v}_t + gD_y\left(\boldsymbol{h}(\boldsymbol{h}+\boldsymbol{b})\right) - g(\boldsymbol{h}+\boldsymbol{b})D_y\boldsymbol{h} + \frac{1}{2}\boldsymbol{h}D_y(\boldsymbol{v}^2) - \frac{1}{2}\boldsymbol{v}^2 D_y \boldsymbol{h} + \frac{1}{2}\boldsymbol{v}D_y(\boldsymbol{h}\boldsymbol{v}) - \frac{1}{2}\boldsymbol{h}\boldsymbol{v}D_y \boldsymbol{v}                                                                                                        \\
     & \quad + \frac{1}{2}D_x(\boldsymbol{h}\boldsymbol{u}\boldsymbol{v}) - \frac{1}{2} \boldsymbol{u}\boldsymbol{v}D_x \boldsymbol{h} + \frac{1}{2} \boldsymbol{h}\boldsymbol{u}D_x \boldsymbol{v} - \frac{1}{2} \boldsymbol{h}\boldsymbol{v}D_x \boldsymbol{u}                                                                                                                                                       \\
     & \quad + \frac{\lambda}{6}\frac{\boldsymbol{\eta}^{2}}{\boldsymbol{h}^{2}}D_y \boldsymbol{h} + \frac{\lambda}{3}D_y \boldsymbol{\eta} - \frac{\lambda}{3}\frac{\boldsymbol{\eta}}{\boldsymbol{h}}D_y \boldsymbol{\eta} - \frac{\lambda}{6}D_y\left(\frac{\boldsymbol{\eta}^{2}}{\boldsymbol{h}}\right) + \frac{\lambda}{2}\left(1-\frac{\boldsymbol{\eta}}{\boldsymbol{h}}\right)D_y \boldsymbol{b} = \boldsymbol{0}, \\[0.5em]
     & \boldsymbol{h}\boldsymbol{w}_t + \frac{1}{2}D_x(\boldsymbol{h}\boldsymbol{u}\boldsymbol{w}) + \frac{1}{2}\boldsymbol{h}\boldsymbol{u}D_x \boldsymbol{w} - \frac{1}{2}\boldsymbol{u}\boldsymbol{w}D_x \boldsymbol{h} - \frac{1}{2}\boldsymbol{h}\boldsymbol{w}D_x \boldsymbol{u}                                                                                                                                                   \\
     & \quad + \frac{1}{2}D_y(\boldsymbol{h}\boldsymbol{v}\boldsymbol{w}) + \frac{1}{2}\boldsymbol{h}\boldsymbol{v}D_y \boldsymbol{w} - \frac{1}{2}\boldsymbol{v}\boldsymbol{w}D_y \boldsymbol{h} - \frac{1}{2}\boldsymbol{h}\boldsymbol{w}D_y \boldsymbol{v} = \lambda\left(1-\frac{\boldsymbol{\eta}}{\boldsymbol{h}}\right),                                                                                                          \\[0.5em]
     & \boldsymbol{\eta}_t + \boldsymbol{u}D_x \boldsymbol{\eta} + \boldsymbol{v}D_y \boldsymbol{\eta} + \frac{3}{2}\boldsymbol{u}D_x \boldsymbol{b} + \frac{3}{2}\boldsymbol{v}D_y \boldsymbol{b} = \boldsymbol{w}.
  \end{aligned}
\end{equation}

\begin{theorem} \label{theorem:splitform_2D_SGN_periodic}
  Consider the semi-discretization \eqref{eq:sd_periodic_2D_SGN} of the two-dimensional hyperbolic approximation of the
  SGN equations \eqref{eq:2D_hyperbolic_SGN} with periodic boundary conditions.
  If $D_x,D_y$ are periodic first-derivative SBP operators
  with diagonal mass/norm matrix $M$,
  \begin{enumerate}
    \item the total water mass $\boldsymbol{1}^T M \boldsymbol{h}$ is conserved.
    \item the total energy $\boldsymbol{1}^T M \boldsymbol{E}$ is conserved.
  \end{enumerate}
\end{theorem}

\begin{proof}
Part 1:
We get
  \begin{equation*}
    \begin{aligned}
      \partial_t(\boldsymbol{1}^TM\boldsymbol{h}) & =\boldsymbol{1}^TM\partial_t\boldsymbol{h}                                                                                  \\
                        & = - \boldsymbol{1}^TM \boldsymbol{u}D_x \boldsymbol{h}
                            - \boldsymbol{1}^T M\boldsymbol{h} D_x \boldsymbol{u}
                            - \boldsymbol{1}^T M\boldsymbol{v}D_y \boldsymbol{h}
                            - \boldsymbol{1}^T M \boldsymbol{h} D_y \boldsymbol{v}                                        \\
                        & =-(\boldsymbol{u}^T MD_x \boldsymbol{h}  + \boldsymbol{h}^T MD_x\boldsymbol{u} + \boldsymbol{v}^TMD_y\boldsymbol{h} + \boldsymbol{h}^T MD_y \boldsymbol{v})
                        = 0
    \end{aligned}
  \end{equation*}
  using the periodic SBP property $M D_d + D_d^T M = 0$ for $d \in \{x,y\}$.

  Part 2:
To show energy conservation, we have to show that
  \begin{equation} \label{eq:splitcondition}
    \partial_t \boldsymbol{E} = \partial_{\boldsymbol{q}}\boldsymbol{E} \cdot \partial_t \boldsymbol{q} = \partial_{\boldsymbol{h}} \boldsymbol{E} \cdot \partial_t \boldsymbol{h} + \partial_{\boldsymbol{u}} \boldsymbol{E} \cdot \partial_t \boldsymbol{u} + \partial_{\boldsymbol{v}} \boldsymbol{E} \cdot \partial_t \boldsymbol{v} + \partial_{\boldsymbol{w}} \boldsymbol{E} \cdot \partial_t \boldsymbol{w} + \partial_{\boldsymbol{\eta}} \boldsymbol{E} \cdot \partial_t \boldsymbol{\eta}    \\
  \end{equation}
  satisfies $\boldsymbol{1}^T M \partial_t \boldsymbol{E} = 0$, with the energy
  \begin{equation*}
    \boldsymbol{E} = \boldsymbol{h}\left(\frac{\boldsymbol{u}^2+\boldsymbol{v}^2}{2} + \frac{1}{6}\boldsymbol{w}^2 + \frac{g}{2}(\boldsymbol{h}+2\boldsymbol{b}) + \frac{\lambda}{6}\left(\frac{\boldsymbol{\eta}}{\boldsymbol{h}}-1\right)^2\right)
  \end{equation*}
  and its derivatives
  \begin{equation}
    \begin{aligned}
      \partial_{\boldsymbol{u}} \boldsymbol{E}    & = \boldsymbol{h} \boldsymbol{u},                                                           \\
      \partial_{\boldsymbol{v}} \boldsymbol{E}    & = \boldsymbol{h} \boldsymbol{v},                                                           \\
      \partial_{\boldsymbol{w}} \boldsymbol{E}    & = \tfrac{1}{3} \boldsymbol{h} \boldsymbol{w},                                              \\
      \partial_{\boldsymbol{\eta}} \boldsymbol{E} & = -\tfrac{\lambda}{3}\!\left(1-\tfrac{\boldsymbol{\eta}}{\boldsymbol{h}}\right),           \\
      \partial_{\boldsymbol{h}} \boldsymbol{E}    & = \tfrac{1}{2}\boldsymbol{u}^2 + \tfrac{1}{2}\boldsymbol{v}^2 + \tfrac{1}{6} \boldsymbol{w}^2 + g\boldsymbol{h} + g\boldsymbol{b}
      + \tfrac{\lambda}{6}\!\left(1-\tfrac{\boldsymbol{\eta}^2}{\boldsymbol{h}^2}\right).
    \end{aligned}
  \end{equation}
  We compute
  \begin{equation*}
    \begin{aligned}
       & -\partial_t(\boldsymbol{1}^TM\boldsymbol{E})                                                                                                                                                                                                                                                                           \\
       & = -\left(\boldsymbol{1}^{T} M \partial_{\boldsymbol{h}} \boldsymbol{E} \partial_t \boldsymbol{h}+\boldsymbol{1}^T M \partial_{\boldsymbol{u}} \boldsymbol{E} \partial_t \boldsymbol{u}+\boldsymbol{1}^T M \partial_{\boldsymbol{v}} \boldsymbol{E} \partial_t \boldsymbol{v}+\boldsymbol{1}^T M \partial_{\boldsymbol{w}} \boldsymbol{E} \partial_t \boldsymbol{w}+\boldsymbol{1}^T M \partial_{\boldsymbol{\eta}} \boldsymbol{E} \partial_t \boldsymbol{\eta}\right)                                                                                                      \\
       & =\boldsymbol{1}^{T} M\left(g \boldsymbol{h}+g\boldsymbol{b}+\frac{1}{2} \boldsymbol{u}^2+\frac{1}{2} \boldsymbol{v}^2+\frac{1}{6} \boldsymbol{w}^2+\frac{\lambda}{6}\left(1-\frac{\boldsymbol{\eta}^2}{\boldsymbol{h}^2}\right)\right) \left(\boldsymbol{u} D_x \boldsymbol{h}+\boldsymbol{h} D_x \boldsymbol{u}+\boldsymbol{v} D_y \boldsymbol{h}+\boldsymbol{h} D_y \boldsymbol{v} \right)                                                                                                                                                                                      \\
       & +\boldsymbol{1}^{T} M\boldsymbol{u}\left( g D_x(\boldsymbol{h}(\boldsymbol{h}+\boldsymbol{b}))-g(\boldsymbol{h}+\boldsymbol{b}) D_x \boldsymbol{h}+\frac{1}{2} \boldsymbol{h} D_x \boldsymbol{u}^2 - \frac{1}{2} \boldsymbol{u}^2 D_x \boldsymbol{h}+\frac{1}{2} \boldsymbol{u} D_x(\boldsymbol{h} \boldsymbol{u})  \right.                                                                                                                                                                                    \\
       & -\frac{1}{2} \boldsymbol{h} \boldsymbol{u} D_x \boldsymbol{u} + \frac{1}{2} D_y(\boldsymbol{h} \boldsymbol{u} \boldsymbol{v})+\frac{1}{2}\left(-\boldsymbol{u} \boldsymbol{v} D_y \boldsymbol{h}+\boldsymbol{h} \boldsymbol{v} D_y \boldsymbol{u}-\boldsymbol{h} \boldsymbol{u} D_y \boldsymbol{v}\right) \\
       & \left. +\frac{\lambda}{6} \frac{\boldsymbol{\eta}^2}{\boldsymbol{h}^2} D_x \boldsymbol{h}+\frac{\lambda}{3} D_x \boldsymbol{\eta}-\frac{\lambda}{3} \frac{\boldsymbol{\eta}}{\boldsymbol{h}} D_x \boldsymbol{\eta} -\frac{\lambda}{6} D_x\left(\frac{\boldsymbol{\eta}^2}{\boldsymbol{h}}\right)+\frac{\lambda}{2}\left(1-\frac{\boldsymbol{\eta}}{\boldsymbol{h}}\right) D_x \boldsymbol{b} \right) \\
       & +\boldsymbol{1}^TM  \boldsymbol{v}\Bigl(g D_y(\boldsymbol{h}(\boldsymbol{h}+\boldsymbol{b}))-g(\boldsymbol{h}+\boldsymbol{b}) D_y \boldsymbol{h}+\frac{1}{2} \boldsymbol{h} D_y\left(\boldsymbol{v}^2\right)-\frac{1}{2} \boldsymbol{v}^2 D_y \boldsymbol{h}+\frac{1}{2} \boldsymbol{v} D_y(\boldsymbol{h} \boldsymbol{v})\\
       & -\frac{1}{2} \boldsymbol{h} \boldsymbol{v} D_y \boldsymbol{v}+\frac{1}{2} D_x(\boldsymbol{h} \boldsymbol{u} \boldsymbol{v}) +\frac{1}{2}\left(-\boldsymbol{u} \boldsymbol{v} D_x \boldsymbol{h}+\boldsymbol{h} \boldsymbol{u} D_x \boldsymbol{v}-\boldsymbol{h} \boldsymbol{v} D_x \boldsymbol{u}\right)+\frac{\lambda}{6} \frac{\boldsymbol{\eta}^2}{\boldsymbol{h}^2} D_y \boldsymbol{h}+\frac{\lambda}{3} D_y \boldsymbol{\eta} \\
       & -\frac{\lambda}{3} \frac{\boldsymbol{\eta}}{\boldsymbol{h}} D_y \boldsymbol{\eta}-\frac{\lambda}{6} D_y\left(\frac{\boldsymbol{\eta}^2}{\boldsymbol{h}}\right)+\frac{\lambda}{2}\left(1-\frac{\boldsymbol{\eta}}{\boldsymbol{h}}\right) D_y \boldsymbol{b}\Bigr) \\
       & +\boldsymbol{1}^T M\left(\frac{1}{3}  \boldsymbol{w}\right)\left(\frac{1}{2} D_x(\boldsymbol{h} \boldsymbol{u} \boldsymbol{w})+\frac{1}{2} \boldsymbol{h} \boldsymbol{u} D_x \boldsymbol{w}-\frac{1}{2} \boldsymbol{u} \boldsymbol{w} D_x \boldsymbol{h}-\frac{1}{2} \boldsymbol{h} \boldsymbol{w} D_x \boldsymbol{u}\right.                                                                                                          \\
       & \left. +\frac{1}{2} D_y(\boldsymbol{h} \boldsymbol{v} \boldsymbol{w})+\frac{1}{2} \boldsymbol{h} \boldsymbol{v} D_y \boldsymbol{w} -\frac{1}{2} \boldsymbol{v} \boldsymbol{w} D_y \boldsymbol{h}-\frac{1}{2} \boldsymbol{h} \boldsymbol{w} D_y \boldsymbol{v}-\lambda\left(1-\frac{\boldsymbol{\eta}}{\boldsymbol{h}}\right)\right)                                                                                                                                                                                               \\
       & +\boldsymbol{1}^{T} M\left(-\frac{\lambda}{3}\left(1-\frac{\boldsymbol{\eta}}{\boldsymbol{h}}\right)\right)\left(\boldsymbol{u} D_x \boldsymbol{\eta}+\boldsymbol{v} D_y \boldsymbol{\eta}+\frac{3}{2} \boldsymbol{u} D_x \boldsymbol{b}+\frac{3}{2} \boldsymbol{v} D_y \boldsymbol{b}-\boldsymbol{w}\right).                                                                                                                                                          \\
    \end{aligned}
  \end{equation*}
    To simplify the equation above, we decompose the expression for
  $\partial_t(\boldsymbol{1}^TM\boldsymbol{E})$ into a shallow water (SW) component and a nonhydrostatic (NHS)
  component. In the subsequent calculation, several terms cancel either due to
  appearing twice with opposite signs or because of the periodic SBP property.
  First, the SW part is
\begin{equation*}
    \begin{aligned}
       & g(\boldsymbol{h}\boldsymbol{u})^{T}MD_x \boldsymbol{h} + g\left(\boldsymbol{h}^2\right)^{T}MD_x \boldsymbol{u} + g(\boldsymbol{h}\boldsymbol{v})^{T}MD_y \boldsymbol{h} + g\left(\boldsymbol{h}^2\right)^{T}MD_y \boldsymbol{v}                                                                                            \\
       &  + \frac{1}{2}\left(\boldsymbol{u}^3\right)^{T}MD_x \boldsymbol{h} + \frac{1}{2}\left(\boldsymbol{u}^2 \boldsymbol{h}\right)^{T}MD_x \boldsymbol{u} + \frac{1}{2}\left(\boldsymbol{u}^2 \boldsymbol{v}\right)^{T}MD_y \boldsymbol{h} + \frac{1}{2}\left(\boldsymbol{u}^2 \boldsymbol{h}\right)^{T}MD_y \boldsymbol{v}                                  \\
       & + \frac{1}{2}\left(\boldsymbol{v}^2 \boldsymbol{u}\right)^{T}MD_x \boldsymbol{h} + \frac{1}{2}\left(\boldsymbol{v}^2 \boldsymbol{h}\right)^{T}MD_x \boldsymbol{u} + \frac{1}{2}\left(\boldsymbol{v}^3\right)^{T}MD_y \boldsymbol{h} \\
       &  + \frac{1}{2}\left(\boldsymbol{v}^2 \boldsymbol{h}\right)^{T}MD_y \boldsymbol{v}  + g(\boldsymbol{h}\boldsymbol{u})^{T}MD_x \boldsymbol{h}                                                         \\
       &  - g\left(\boldsymbol{u} (\boldsymbol{h}+\boldsymbol{b})\right)^{T}MD_x \boldsymbol{h} + \frac{1}{2}(\boldsymbol{h}\boldsymbol{u})^{T}MD_x\left(\boldsymbol{u}^2\right) - \frac{1}{2}\left(\boldsymbol{u}^3\right)^{T}MD_x \boldsymbol{h} + \frac{1}{2}\left(\boldsymbol{u}^2\right)^{T}MD_x(\boldsymbol{h}\boldsymbol{u})                                                                              \\
       &  - \frac{1}{2}\left(\boldsymbol{h}\boldsymbol{u}^2\right)^{T}MD_x \boldsymbol{u} + \frac{1}{2}(\boldsymbol{u})^{T}MD_y(\boldsymbol{h}\boldsymbol{u}\boldsymbol{v}) - \frac{1}{2}\left(\boldsymbol{u}^2 \boldsymbol{v}\right)^{T}MD_y \boldsymbol{h} + \frac{1}{2}(\boldsymbol{h}\boldsymbol{u}\boldsymbol{v})^{T}MD_y \boldsymbol{u}                                                     \\
       &  - \frac{1}{2}\left(\boldsymbol{h}\boldsymbol{u}^2\right)^{T}MD_y \boldsymbol{v} + g(\boldsymbol{v})^{T}MD_y\left(\boldsymbol{h}(\boldsymbol{h}+\boldsymbol{b})\right) - g\left(\boldsymbol{v} (\boldsymbol{h}+\boldsymbol{b})\right)^{T}MD_y \boldsymbol{h} + \frac{1}{2}(\boldsymbol{h}\boldsymbol{v})^{T}MD_y\left(\boldsymbol{v}^2\right)                                       \\
       &  - \frac{1}{2}\left(\boldsymbol{v}^3\right)^{T}MD_y \boldsymbol{h} + \frac{1}{2}\left(\boldsymbol{v}^2\right)^{T}MD_y(\boldsymbol{h}\boldsymbol{v}) - \frac{1}{2}\left(\boldsymbol{h}\boldsymbol{v}^2\right)^{T}MD_y \boldsymbol{v} + \frac{1}{2}(\boldsymbol{v})^{T}MD_x(\boldsymbol{h}\boldsymbol{u}\boldsymbol{v})                                           \\
       &  - \frac{1}{2}\left(\boldsymbol{v}^2 \boldsymbol{u}\right)^{T}MD_x \boldsymbol{h} + \frac{1}{2}\left(\boldsymbol{h}\boldsymbol{u}\boldsymbol{v}\right)^{T}MD_x \boldsymbol{v} - \frac{1}{2}\left(\boldsymbol{h}\boldsymbol{v}^2\right)^{T}MD_x \boldsymbol{u} + g(\boldsymbol{b}\boldsymbol{u})^TMD_x\boldsymbol{h}                                                                                                                                                                                               \\
       &  + g(\boldsymbol{h}\boldsymbol{b})^TMD_x\boldsymbol{u} + g(\boldsymbol{b}\boldsymbol{v})^TMD_y \boldsymbol{h} + g(\boldsymbol{h}\boldsymbol{b})^TMD_y\boldsymbol{v} \\
       & = g(\boldsymbol{h}\boldsymbol{u})^TMD_x\boldsymbol{h} + g(\boldsymbol{b}\boldsymbol{u})^TMD_x\boldsymbol{h} - g\left(\boldsymbol{u}(\boldsymbol{h}+\boldsymbol{b})\right)^TMD_x\boldsymbol{h}                                                                                                                                  \\
       &  + g\left(\boldsymbol{h}^2\right)^TMD_x\boldsymbol{u} + g(\boldsymbol{u})^TMD_x\boldsymbol{h}^2 + g\left(\boldsymbol{h}^2\right)^TMD_y\boldsymbol{v} + g(\boldsymbol{v})^TMD_y\boldsymbol{h}^2                                                                                                                                                                                                                \\
       &  + \frac{1}{2}\left(\boldsymbol{u}^2\right)^TMD_x(\boldsymbol{h}\boldsymbol{u}) + \frac{1}{2}(\boldsymbol{h}\boldsymbol{u})^TMD_x\boldsymbol{u}^2  + \frac{1}{2}(\boldsymbol{h}\boldsymbol{u}\boldsymbol{v})^TMD_x\boldsymbol{v} + \frac{1}{2}(\boldsymbol{v})^TMD_x(\boldsymbol{h}\boldsymbol{u}\boldsymbol{v})                                                                                                                                            \\
       &  + \frac{1}{2}(\boldsymbol{h}\boldsymbol{u}\boldsymbol{v})^TMD_y\boldsymbol{u} + \frac{1}{2}(\boldsymbol{u})^TMD_y\boldsymbol{h}\boldsymbol{u}\boldsymbol{v} + \frac{1}{2}\left(\boldsymbol{v}^2\right)^TMD_x\boldsymbol{v}\boldsymbol{h} + (\boldsymbol{v}\boldsymbol{h})^TMD_x\boldsymbol{v}^2                                                                                                                                                                                                     \\
       &  + g(\boldsymbol{h}\boldsymbol{b})^TMD_x\boldsymbol{u} + g(\boldsymbol{u})^TMD_x(\boldsymbol{h}\boldsymbol{b}) + g(\boldsymbol{h}\boldsymbol{b})^TMD_y\boldsymbol{v} + g(\boldsymbol{v})^TMD_y(\boldsymbol{h}\boldsymbol{b})                                                                                                                                                                                           \\
       & = 0.
    \end{aligned}
  \end{equation*}
  The NHS part is
  \begin{align*}
       & \frac{1}{6}\left(\boldsymbol{w}^2\boldsymbol{u}\right)^TMD_x \boldsymbol{h} + \frac{1}{6}\left(\boldsymbol{w}^2\boldsymbol{h}\right)^TMD_x\boldsymbol{u}+ \frac{1}{6}\left(\boldsymbol{w}^2\boldsymbol{v}\right)^TMD_y\boldsymbol{h}+ \frac{1}{6}\left(\boldsymbol{w}^2\boldsymbol{h}\right)^TMD_y\boldsymbol{v}                                                                                                                                   \\
       & +\frac{\lambda}{6}\left(\left(1-\frac{\boldsymbol{\eta}^2}{\boldsymbol{h}^2}\right)\boldsymbol{u}\right)^TMD_x \boldsymbol{h} + \frac{\lambda}{6}\left(\left(1-\frac{\boldsymbol{\eta}^2}{\boldsymbol{h}^2}\right)\boldsymbol{h}\right)^TMD_x\boldsymbol{u}+ \frac{\lambda}{6}\left(\left(1-\frac{\boldsymbol{\eta}^2}{\boldsymbol{h}^2}\right)\boldsymbol{v}\right)^TMD_y\boldsymbol{h}         \\
       & + \frac{\lambda}{6}\left(\left(1-\frac{\boldsymbol{\eta}^2}{\boldsymbol{h}^2}\right)\boldsymbol{h}\right)^TMD_y\boldsymbol{v} +\frac{1}{6}\left(\boldsymbol{w}\right)^T M D_x\left(\boldsymbol{h} \boldsymbol{u} \boldsymbol{w}\right)+\frac{1}{6}\left(\boldsymbol{w}\boldsymbol{h}\boldsymbol{u}\right)^T M D_x\left(\boldsymbol{w}\right)                                          \\
       & -\frac{1}{6}\left(\boldsymbol{w}^2\boldsymbol{u}\right)^T M D_x\left(\boldsymbol{h}\right) -\frac{1}{6}\left(\boldsymbol{w}^2\boldsymbol{h}\right)^T M D_x\left(\boldsymbol{u}\right) + \frac{1}{6}\left(\boldsymbol{w}\right)^T M D_y\left(\boldsymbol{h} \boldsymbol{v} \boldsymbol{w}\right) + \frac{1}{6}\left(\boldsymbol{w}\boldsymbol{h}\boldsymbol{v}\right)^T M D_y\left(\boldsymbol{w}\right)                                                                                    \\
       & -\frac{1}{6}\left(\boldsymbol{w}^2\boldsymbol{v}\right)^T M D_y\left(\boldsymbol{h}\right)-\frac{1}{6}\left(\boldsymbol{w}^2\boldsymbol{h}\right)^T M D_y\left(\boldsymbol{v}\right)-\frac{1}{3}\left(\boldsymbol{w}\lambda\left(1-\frac{\boldsymbol{\eta}}{\boldsymbol{h}}\right)\right)^TM                                                                                            \\
       & -\frac{\lambda}{3}\left(\left(1-\frac{\boldsymbol{\eta}}{\boldsymbol{h}}\right)\boldsymbol{u}\right)^TMD_x \boldsymbol{\eta} -\frac{\lambda}{3}\left(\left(1-\frac{\boldsymbol{\eta}}{\boldsymbol{h}}\right)\boldsymbol{v}\right)^TMD_y\boldsymbol{\eta}-\frac{\lambda}{2}\left(\left(1-\frac{\boldsymbol{\eta}}{\boldsymbol{h}}\right)\boldsymbol{u}\right)^TMD_x\boldsymbol{b}                                                                                                                                                                         \\
       & - \frac{\lambda}{2}\left(\left(1-\frac{\boldsymbol{\eta}}{\boldsymbol{h}}\right)\boldsymbol{v}\right)^TMD_y\boldsymbol{b} +\frac{\lambda}{3}\left(\left(1-\frac{\boldsymbol{\eta}}{\boldsymbol{h}}\right)\boldsymbol{w}\right)^T M  +\frac{\lambda}{6} \left(\boldsymbol{u}\frac{\boldsymbol{\eta}^2}{\boldsymbol{h}^2}\right)^TM D_x \boldsymbol{h}+\frac{\lambda}{3}\left(\boldsymbol{u}\right)^TM D_x \boldsymbol{\eta}\\
       & -\frac{\lambda}{3} \left(\boldsymbol{u} \frac{\boldsymbol{\eta}}{\boldsymbol{h}}\right)^TM D_x \boldsymbol{\eta} -\frac{\lambda}{6}\left(\boldsymbol{u}\right)^TM D_x\left(\frac{\boldsymbol{\eta}^2}{\boldsymbol{h}}\right)+\frac{\lambda}{2}\left(\boldsymbol{u}\left(1-\frac{\boldsymbol{\eta}}{\boldsymbol{h}}\right)\right)^TM D_x \boldsymbol{b} +\frac{\lambda}{6} \left(\boldsymbol{v}\frac{\boldsymbol{\eta}^2}{\boldsymbol{h}^2}\right)^TM D_y \boldsymbol{h} \\
       & +\frac{\lambda}{3}\left(\boldsymbol{v}\right)^T M D_y \boldsymbol{\eta} - \frac{\lambda}{3}\left(\boldsymbol{v} \frac{\boldsymbol{\eta}}{\boldsymbol{h}}\right)^TM D_y \boldsymbol{\eta}-\frac{\lambda}{6}\left(\boldsymbol{v}\right)^TM D_y\left(\frac{\boldsymbol{\eta}^2}{\boldsymbol{h}}\right)+\frac{\lambda}{2}\left(\boldsymbol{v}\left(1-\frac{\boldsymbol{\eta}}{\boldsymbol{h}}\right)\right)^TM D_y \boldsymbol{b} \\
       & = \frac{1}{6} \left(\boldsymbol{w}\right)^T MD_x\left(\boldsymbol{h}\boldsymbol{u}\boldsymbol{w}\right)+\frac{1}{6} \left(\boldsymbol{w}\boldsymbol{h}\boldsymbol{u}\right)^TMD_x\boldsymbol{w} +\frac{1}{6} \left(\boldsymbol{w}\right)^T MD_y\left(\boldsymbol{h}\boldsymbol{v}\boldsymbol{w}\right)+\frac{1}{6} \left(\boldsymbol{w}\boldsymbol{h}\boldsymbol{v}\right)^TMD_y\boldsymbol{w}                                                                                                                                 \\
       & +\frac{\lambda}{6} \left(\boldsymbol{h}\right)^TMD_x\boldsymbol{u} + \frac{\lambda}{6} \boldsymbol{u}^T MD_x \boldsymbol{h} -\frac{\lambda}{6} \left(\frac{\boldsymbol{\eta}^2}{\boldsymbol{h}} \right)^TMD_x\boldsymbol{u}-\frac{\lambda}{6}\boldsymbol{u}^TMD_x \frac{\boldsymbol{\eta}^2}{\boldsymbol{h}}                                                                                                                                     \\
       & +\frac{\lambda}{6} \left(\boldsymbol{h}\right)^TMD_y\boldsymbol{v} + \frac{\lambda}{6} \boldsymbol{v}^T MD_y \boldsymbol{h} -\frac{\lambda}{6} \left(\frac{\boldsymbol{\eta}^2}{\boldsymbol{h}} \right)^TMD_y\boldsymbol{v}-\frac{\lambda}{6}\boldsymbol{v}^TMD_y \frac{\boldsymbol{\eta}^2}{\boldsymbol{h}}                                                                                                                                     \\
       & =0.
  \end{align*}
  Thus, the semi-discretization conserves the energy.
\end{proof}

\begin{remark}
  For $\lambda = 0$, the hyperbolic system reduces to the shallow-water equations.
  Consequently, Theorem~\ref{theorem:splitform_2D_SGN_periodic} provides, as a special case,
  an energy-conserving semi-discretization of the two-dimensional shallow water equations with variable bathymetry,
  given by
\begin{equation} \label{eq:splitform_2D_shallow_water}
    \begin{aligned}
            & \boldsymbol{h}_t + \boldsymbol{u} D_x \boldsymbol{h} + \boldsymbol{h} D_x \boldsymbol{u} + \boldsymbol{v} D_y \boldsymbol{h} + \boldsymbol{h} D_y \boldsymbol{v} = 0, \\
     & \boldsymbol{h} \boldsymbol{u}_t + g D_x(\boldsymbol{h}(\boldsymbol{h}+\boldsymbol{b})) - g(\boldsymbol{h}+\boldsymbol{b}) D_x \boldsymbol{h} + \frac{1}{2} \boldsymbol{h} D_x\left(\boldsymbol{u}^2\right) - \frac{1}{2} \boldsymbol{u}^2 D_x \boldsymbol{h} + \frac{1}{2} \boldsymbol{u} D_x(\boldsymbol{h} \boldsymbol{u}) - \frac{1}{2} \boldsymbol{h} \boldsymbol{u} D_x \boldsymbol{u} \\
            & + \frac{1}{2} D_y(\boldsymbol{h} \boldsymbol{u} \boldsymbol{v}) + \frac{1}{2}\left(-\boldsymbol{u} \boldsymbol{v} D_y \boldsymbol{h} + \boldsymbol{h} \boldsymbol{v} D_y \boldsymbol{u} - \boldsymbol{h} \boldsymbol{u} D_y \boldsymbol{v}\right) = 0, \\
     & \boldsymbol{h} \boldsymbol{v}_t + g D_y(\boldsymbol{h}(\boldsymbol{h}+\boldsymbol{b})) - g(\boldsymbol{h}+\boldsymbol{b}) D_y \boldsymbol{h} + \frac{1}{2} \boldsymbol{h} D_y\left(\boldsymbol{v}^2\right) - \frac{1}{2} \boldsymbol{v}^2 D_y \boldsymbol{h} + \frac{1}{2} \boldsymbol{v} D_y(\boldsymbol{h} \boldsymbol{v}) - \frac{1}{2} \boldsymbol{h} \boldsymbol{v} D_y \boldsymbol{v} \\
            & + \frac{1}{2} D_x(\boldsymbol{h} \boldsymbol{u} \boldsymbol{v}) + \frac{1}{2}\left(-\boldsymbol{u} \boldsymbol{v} D_x \boldsymbol{h} + \boldsymbol{h} \boldsymbol{u} D_x \boldsymbol{v} - \boldsymbol{h} \boldsymbol{v} D_x \boldsymbol{u}\right) = 0.
    \end{aligned}
\end{equation}
See also~\cite{wintermeyer2017entropy} for an alternative split form
in conservative variables.
\end{remark}

\subsection{Weakly Enforced Reflecting Boundary Conditions}

For reflecting (solid wall) boundary conditions, the normal velocity at
the boundary is zero. We impose this condition weakly using
simultaneous approximation terms (SATs)~\cite{carpenter1994time}.
Thus, we propose the semi-discretization
\begin{equation} \label{eq:semi_2D}
  \begin{aligned}
     & \boldsymbol{h}_t + \boldsymbol{u}D_x \boldsymbol{h} + \boldsymbol{h}D_x\boldsymbol{u} + \boldsymbol{v}D_y \boldsymbol{h} + \boldsymbol{h}D_y\boldsymbol{v} - M^{-1}R^TBN_xR(\boldsymbol{h}\boldsymbol{u}) - M^{-1}R^TBN_yR(\boldsymbol{h}\boldsymbol{v}) = \boldsymbol{0},                                                                                                                                                        \\[0.5em]
     & \boldsymbol{h}\boldsymbol{u}_t + gD_x\left(\boldsymbol{h}(\boldsymbol{h}+\boldsymbol{b})\right) - g(\boldsymbol{h}+\boldsymbol{b})D_x \boldsymbol{h} + \frac{1}{2}\boldsymbol{h}D_x(\boldsymbol{u}^2) - \frac{1}{2}\boldsymbol{u}^2 D_x \boldsymbol{h} + \frac{1}{2}\boldsymbol{u}D_x(\boldsymbol{h}\boldsymbol{u}) - \frac{1}{2}\boldsymbol{h}\boldsymbol{u}D_x \boldsymbol{u}                                                                                                       \\
     & \quad + \frac{1}{2}D_y(\boldsymbol{h}\boldsymbol{u}\boldsymbol{v}) - \frac{1}{2} \boldsymbol{u}\boldsymbol{v}D_y\boldsymbol{h} + \frac{1}{2} \boldsymbol{h}\boldsymbol{v}D_y \boldsymbol{u} - \frac{1}{2} \boldsymbol{h}\boldsymbol{u}D_y \boldsymbol{v}                                                                                                                                                          \\
     & \quad + \frac{\lambda}{6}\frac{\boldsymbol{\eta}^{2}}{\boldsymbol{h}^{2}}D_x \boldsymbol{h} + \frac{\lambda}{3}D_x \boldsymbol{\eta} - \frac{\lambda}{3}\frac{\boldsymbol{\eta}}{\boldsymbol{h}}D_x \boldsymbol{\eta} - \frac{\lambda}{6}D_x\left(\frac{\boldsymbol{\eta}^{2}}{\boldsymbol{h}}\right) + \frac{\lambda}{2}\left(1-\frac{\boldsymbol{\eta}}{\boldsymbol{h}}\right)D_x \boldsymbol{b} = \boldsymbol{0}, \\[0.5em]
     & \boldsymbol{h}\boldsymbol{v}_t + gD_y\left(\boldsymbol{h}(\boldsymbol{h}+\boldsymbol{b})\right) - g(\boldsymbol{h}+\boldsymbol{b})D_y\boldsymbol{h} + \frac{1}{2}\boldsymbol{h}D_y(\boldsymbol{v}^2) - \frac{1}{2}\boldsymbol{v}^2 D_y \boldsymbol{h} + \frac{1}{2}\boldsymbol{v}D_y(\boldsymbol{h}\boldsymbol{v}) - \frac{1}{2}\boldsymbol{h}\boldsymbol{v}D_y \boldsymbol{v}                                                                                                        \\
     & \quad + \frac{1}{2}D_x(\boldsymbol{h}\boldsymbol{u}\boldsymbol{v}) - \frac{1}{2} \boldsymbol{u}\boldsymbol{v}D_x \boldsymbol{h} + \frac{1}{2} \boldsymbol{h}\boldsymbol{u}D_x \boldsymbol{v} - \frac{1}{2} \boldsymbol{h}\boldsymbol{v}D_x \boldsymbol{u}                                                                                                                                                          \\
     & \quad + \frac{\lambda}{6}\frac{\boldsymbol{\eta}^{2}}{\boldsymbol{h}^{2}}D_y \boldsymbol{h} + \frac{\lambda}{3}D_y \boldsymbol{\eta} - \frac{\lambda}{3}\frac{\boldsymbol{\eta}}{\boldsymbol{h}}D_y \boldsymbol{\eta} - \frac{\lambda}{6}D_y\left(\frac{\boldsymbol{\eta}^{2}}{\boldsymbol{h}}\right) + \frac{\lambda}{2}\left(1-\frac{\boldsymbol{\eta}}{\boldsymbol{h}}\right)D_y \boldsymbol{b} = \boldsymbol{0}, \\[0.5em]
     & \boldsymbol{h}\boldsymbol{w}_t + \frac{1}{2}D_x(\boldsymbol{h}\boldsymbol{u}\boldsymbol{w}) + \frac{1}{2}\boldsymbol{h}\boldsymbol{u}D_x \boldsymbol{w} - \frac{1}{2}\boldsymbol{u}\boldsymbol{w}D_x \boldsymbol{h} - \frac{1}{2}\boldsymbol{h}\boldsymbol{w}D_x \boldsymbol{u}                                                                                                                                                   \\
     & \quad + \frac{1}{2}D_y(\boldsymbol{h}\boldsymbol{v}\boldsymbol{w}) + \frac{1}{2}\boldsymbol{h}\boldsymbol{v}D_y \boldsymbol{w} - \frac{1}{2}\boldsymbol{v}\boldsymbol{w}D_y \boldsymbol{h} - \frac{1}{2}\boldsymbol{h}\boldsymbol{w}D_y \boldsymbol{v} = \lambda\left(1-\frac{\boldsymbol{\eta}}{\boldsymbol{h}}\right),                                                                                                          \\[0.5em]
     & \boldsymbol{\eta}_t + \boldsymbol{u}D_x \boldsymbol{\eta} + \boldsymbol{v}D_y \boldsymbol{\eta} + \frac{3}{2}\boldsymbol{u}D_x \boldsymbol{b} + \frac{3}{2}\boldsymbol{v}D_y \boldsymbol{b} = \boldsymbol{w}.
  \end{aligned}
\end{equation}
Since we use a non-conservative form of the equations for the velocities
and $\eta$, we do not need to include SATs there. Thus, the only equation
that is modified (in addition to choosing non-periodic SBP operators)
compared to the periodic case \eqref{eq:sd_periodic_2D_SGN} is the
first one.

\begin{theorem}\label{theorem:conservation_reflecting}
  Consider the semi-discretization \eqref{eq:semi_2D} of the two-dimensional hyperbolic approximation of the
  SGN equations \eqref{eq:2D_hyperbolic_SGN} with reflecting boundary conditions.
  If $D_x,D_y$ are non-periodic first-derivative SBP operators
  with diagonal mass/norm matrix $M$,
  \begin{enumerate}
    \item the total water mass $\boldsymbol{1}^T M \boldsymbol{h}$ is conserved.
    \item the total energy $\boldsymbol{1}^T M \boldsymbol{E}$ is conserved.
  \end{enumerate}
\end{theorem}

\begin{proof}
  Part 1:
We get
  \begin{equation*}
  \begin{aligned}
      \partial_t(\boldsymbol{1}^TM\boldsymbol{h}) & =\boldsymbol{1}^TM\partial_t\boldsymbol{h}                                                                                  \\
                        & = -\boldsymbol{1}^TM \boldsymbol{u}D_x \boldsymbol{h} -\boldsymbol{1}^T M\boldsymbol{v}D_y \boldsymbol{h} - \boldsymbol{1}^T M\boldsymbol{h} D_x \boldsymbol{u} - \boldsymbol{1}^T M \boldsymbol{h} D_y \boldsymbol{v}  \\
                        & \quad + \boldsymbol{1}^T R^T BN_xR(\boldsymbol{h}\boldsymbol{u}) + \boldsymbol{1}^T R^TBN_yR(\boldsymbol{h}\boldsymbol{v}) \\
                        & = -\boldsymbol{u}^T MD_x \boldsymbol{h}  - \boldsymbol{h}^T MD_x\boldsymbol{u}-\boldsymbol{v}^TMD_y\boldsymbol{h} - \boldsymbol{h}^T MD_y \boldsymbol{v}             \\
                        & \quad +\boldsymbol{1}^T R^TBN_xR(\boldsymbol{h}\boldsymbol{u}) + \boldsymbol{1}^T R^TBN_yR(\boldsymbol{h}\boldsymbol{v})  \\
                        & = - \boldsymbol{u}^T(R^TBN_xR - D_x^TM)\boldsymbol{h} - \boldsymbol{h}^T MD_x\boldsymbol{u}  +\boldsymbol{h}^T(R^TBN_xR)\boldsymbol{u}                                          \\
                        & \quad - \boldsymbol{v}^T(R^TBN_yR-D_y^TM)\boldsymbol{h} - \boldsymbol{h}^TMD_y\boldsymbol{v} + \boldsymbol{h}^T(R^TBN_yR)\boldsymbol{v}                                         \\
                        & = -\boldsymbol{u}^T(R^TBN_xR)\boldsymbol{h} + \boldsymbol{h}^T(R^TBN_xR)\boldsymbol{u} - \boldsymbol{v}^T(R^TBN_yR)\boldsymbol{h} + \boldsymbol{h}^T(R^TBN_yR)\boldsymbol{v}                              \\
                        & \quad +  \boldsymbol{u}^T D_x^T M \boldsymbol{h} - \boldsymbol{h}^T M D_x \boldsymbol{u} + \boldsymbol{v}^TD_y^T M\boldsymbol{h}-\boldsymbol{h}^TMD_y \boldsymbol{v}                                      \\
                        & = 0
    \end{aligned}
 \end{equation*}
using the SBP property.

 Part 2: We compute
  \begin{equation*}
    \begin{aligned}
       & -\partial_t(\boldsymbol{1}^TM\boldsymbol{E})                                                                                                                                                                                                                                                                           \\
       & = -\left(\boldsymbol{1}^{T} M \partial_{\boldsymbol{h}} \boldsymbol{E} \partial_t \boldsymbol{h}+\boldsymbol{1}^T M \partial_{\boldsymbol{u}} \boldsymbol{E} \partial_t \boldsymbol{u}+\boldsymbol{1}^T M \partial_{\boldsymbol{v}} \boldsymbol{E} \partial_t \boldsymbol{v}+\boldsymbol{1}^T M \partial_{\boldsymbol{w}} \boldsymbol{E} \partial_t \boldsymbol{w}+\boldsymbol{1}^T M \partial_{\boldsymbol{\eta}} \boldsymbol{E} \partial_t \boldsymbol{\eta}\right)                                                                                                      \\
       & =\boldsymbol{1}^{T} M\left(g \boldsymbol{h}+g\boldsymbol{b}+\frac{1}{2} \boldsymbol{u}^2+\frac{1}{2} \boldsymbol{v}^2+\frac{1}{6} \boldsymbol{w}^2+\frac{\lambda}{6}\left(1-\frac{\boldsymbol{\eta}^2}{\boldsymbol{h}^2}\right)\right)                                                                                                                                                                            \\
       & \cdot \left(\boldsymbol{u} D_x \boldsymbol{h}+\boldsymbol{h} D_x \boldsymbol{u}+\boldsymbol{v} D_y \boldsymbol{h}+\boldsymbol{h} D_y \boldsymbol{v}+ M^{-1} R^{T} B N_x R(\boldsymbol{h} \boldsymbol{u})+M^{-1} R^{T} B N_y R(\boldsymbol{h} \boldsymbol{v})\right)                                                                                                                                                                                      \\
       & +\boldsymbol{1}^{T} M\boldsymbol{u}\Bigl(g D_x(\boldsymbol{h}(\boldsymbol{h}+\boldsymbol{b}))-g(\boldsymbol{h}+\boldsymbol{b}) D_x \boldsymbol{h}+\frac{1}{2} \boldsymbol{h} D_x \boldsymbol{u}^2 - \frac{1}{2} \boldsymbol{u}^2 D_x \boldsymbol{h}+\frac{1}{2} \boldsymbol{u} D_x(\boldsymbol{h} \boldsymbol{u}) -\frac{1}{2} \boldsymbol{h} \boldsymbol{u} D_x \boldsymbol{u}                                                                                                                       \\
       & +\frac{1}{2} D_y(\boldsymbol{h} \boldsymbol{u} \boldsymbol{v}) +\frac{1}{2}\left(-\boldsymbol{u} \boldsymbol{v} D_y \boldsymbol{h}+\boldsymbol{h} \boldsymbol{v} D_y \boldsymbol{u}-\boldsymbol{h} \boldsymbol{u} D_y \boldsymbol{v}\right)+\frac{\lambda}{6} \frac{\boldsymbol{\eta}^2}{\boldsymbol{h}^2} D_x \boldsymbol{h}+\frac{\lambda}{3} D_x \boldsymbol{\eta}-\frac{\lambda}{3} \frac{\boldsymbol{\eta}}{\boldsymbol{h}} D_x \boldsymbol{\eta}\\
       & -\frac{\lambda}{6} D_x\left(\frac{\boldsymbol{\eta}^2}{\boldsymbol{h}}\right)+\frac{\lambda}{2}\left(1-\frac{\boldsymbol{\eta}}{\boldsymbol{h}}\right) D_x \boldsymbol{b}\Bigr) \\
       & +\boldsymbol{1}^TM  \boldsymbol{v}\Bigl(g D_y(\boldsymbol{h}(\boldsymbol{h}+\boldsymbol{b}))-g(\boldsymbol{h}+\boldsymbol{b}) D_y \boldsymbol{h}+\frac{1}{2} \boldsymbol{h} D_y\left(\boldsymbol{v}^2\right)-\frac{1}{2} \boldsymbol{v}^2 D_y \boldsymbol{h}+\frac{1}{2} \boldsymbol{v} D_y(\boldsymbol{h} \boldsymbol{v})                                                                                                               \\
       & -\frac{1}{2} \boldsymbol{h} \boldsymbol{v} D_y \boldsymbol{v}+\frac{1}{2} D_x(\boldsymbol{h} \boldsymbol{u} \boldsymbol{v}) +\frac{1}{2}\left(-\boldsymbol{u} \boldsymbol{v} D_x \boldsymbol{h}+\boldsymbol{h} \boldsymbol{u} D_x \boldsymbol{v}-\boldsymbol{h} \boldsymbol{v} D_x \boldsymbol{u}\right)+\frac{\lambda}{6} \frac{\boldsymbol{\eta}^2}{\boldsymbol{h}^2} D_y \boldsymbol{h}+\frac{\lambda}{3} D_y \boldsymbol{\eta}\\
       & -\frac{\lambda}{3} \frac{\boldsymbol{\eta}}{\boldsymbol{h}} D_y \boldsymbol{\eta}-\frac{\lambda}{6} D_y\left(\frac{\boldsymbol{\eta}^2}{\boldsymbol{h}}\right)+\frac{\lambda}{2}\left(1-\frac{\boldsymbol{\eta}}{\boldsymbol{h}}\right) D_y \boldsymbol{b}\Bigr) \\
       & +\boldsymbol{1}^T M\left(\frac{1}{3}  \boldsymbol{w}\right)\Bigl(\frac{1}{2} D_x(\boldsymbol{h} \boldsymbol{u} \boldsymbol{w})+\frac{1}{2} \boldsymbol{h} \boldsymbol{u} D_x \boldsymbol{w}-\frac{1}{2} \boldsymbol{u} \boldsymbol{w} D_x \boldsymbol{h}-\frac{1}{2} \boldsymbol{h} \boldsymbol{w} D_x \boldsymbol{u}+\frac{1}{2} D_y(\boldsymbol{h} \boldsymbol{v} \boldsymbol{w})                                                                                                       \\
       & +\frac{1}{2} \boldsymbol{h} \boldsymbol{v} D_y \boldsymbol{w}  -\frac{1}{2} \boldsymbol{v} \boldsymbol{w} D_y \boldsymbol{h}-\frac{1}{2} \boldsymbol{h} \boldsymbol{w} D_y \boldsymbol{v}-\lambda\left(1-\frac{\boldsymbol{\eta}}{\boldsymbol{h}}\right)\Bigr)                                                                                                                                                                                               \\
       & +\boldsymbol{1}^{T} M\left(-\frac{\lambda}{3}\left(1-\frac{\boldsymbol{\eta}}{\boldsymbol{h}}\right)\right)\left(\boldsymbol{u} D_x \boldsymbol{\eta}+\boldsymbol{v} D_y \boldsymbol{\eta}+\frac{3}{2} \boldsymbol{u} D_x \boldsymbol{b}+\frac{3}{2} \boldsymbol{v} D_y \boldsymbol{b}-\boldsymbol{w}\right)                                                                                                                                                          \\
    \end{aligned}
  \end{equation*}
As in Theorem~\ref{theorem:splitform_2D_SGN_periodic}, we split this expression into SW and NHS parts. In the subsequent calculation, several terms cancel either because they appear twice with opposite signs or because of the SBP property.
  The SW part is
  \begin{align*}
       & g(\boldsymbol{h}\boldsymbol{u})^{T}MD_x \boldsymbol{h} + g\left(\boldsymbol{h}^2\right)^{T}MD_x \boldsymbol{u} + g(\boldsymbol{h}\boldsymbol{v})^{T}MD_y \boldsymbol{h} + g\left(\boldsymbol{h}^2\right)^{T}MD_y \boldsymbol{v} - g\left(\boldsymbol{h}^2\right)^{T}\left(R^{T}BN_xR\right)\boldsymbol{u}                                                                                            \\
       &  - g\left(\boldsymbol{h}^2\right)^{T}\left(R^{T}BN_yR\right)\boldsymbol{v} + \frac{1}{2}\left(\boldsymbol{u}^3\right)^{T}MD_x \boldsymbol{h} + \frac{1}{2}\left(\boldsymbol{u}^2 \boldsymbol{h}\right)^{T}MD_x \boldsymbol{u} + \frac{1}{2}\left(\boldsymbol{u}^2 \boldsymbol{v}\right)^{T}MD_y \boldsymbol{h}                                  \\
       &  + \frac{1}{2}\left(\boldsymbol{u}^2 \boldsymbol{h}\right)^{T}MD_y \boldsymbol{v} - \frac{1}{2}\left(\boldsymbol{u}^2 \boldsymbol{h}\right)^{T}\left(R^{T}BN_xR\right)\boldsymbol{u} - \frac{1}{2}\left(\boldsymbol{u}^2 \boldsymbol{h}\right)^{T}\left(R^{T}BN_yR\right)\boldsymbol{v} + \frac{1}{2}\left(\boldsymbol{v}^2 \boldsymbol{u}\right)^{T}MD_x \boldsymbol{h} \\
       &  + \frac{1}{2}\left(\boldsymbol{v}^2 \boldsymbol{h}\right)^{T}MD_x \boldsymbol{u} + \frac{1}{2}\left(\boldsymbol{v}^3\right)^{T}MD_y \boldsymbol{h} + \frac{1}{2}\left(\boldsymbol{v}^2 \boldsymbol{h}\right)^{T}MD_y \boldsymbol{v} - \frac{1}{2}\left(\boldsymbol{v}^2\boldsymbol{h}\right)^{T}\left(R^{T}BN_xR\right)\boldsymbol{u}                                                         \\
       &  - \frac{1}{2}\left(\boldsymbol{v}^2 \boldsymbol{h}\right)^{T}\left(R^{T}BN_yR\right)\boldsymbol{v} + g(\boldsymbol{h}\boldsymbol{u})^{T}MD_x \boldsymbol{h} - g\left(\boldsymbol{u} (\boldsymbol{h}+\boldsymbol{b})\right)^{T}MD_x \boldsymbol{h}  \\
       &  + \frac{1}{2}(\boldsymbol{h}\boldsymbol{u})^{T}MD_x\left(\boldsymbol{u}^2\right) - \frac{1}{2}\left(\boldsymbol{u}^3\right)^{T}MD_x \boldsymbol{h} + \frac{1}{2}\left(\boldsymbol{u}^2\right)^{T}MD_x(\boldsymbol{h}\boldsymbol{u})                                                                              \\
       &  - \frac{1}{2}\left(\boldsymbol{h}\boldsymbol{u}^2\right)^{T}MD_x \boldsymbol{u} + \frac{1}{2}(\boldsymbol{u})^{T}MD_y(\boldsymbol{h}\boldsymbol{u}\boldsymbol{v}) - \frac{1}{2}\left(\boldsymbol{u}^2 \boldsymbol{v}\right)^{T}MD_y \boldsymbol{h} + \frac{1}{2}(\boldsymbol{h}\boldsymbol{u}\boldsymbol{v})^{T}MD_y \boldsymbol{u}                                                    \\
       &  - \frac{1}{2}\left(\boldsymbol{h}\boldsymbol{u}^2\right)^{T}MD_y \boldsymbol{v}  + g(\boldsymbol{v})^{T}MD_y\left(\boldsymbol{h}(\boldsymbol{h}+\boldsymbol{b})\right) - g\left(\boldsymbol{v} (\boldsymbol{h}+\boldsymbol{b})\right)^{T}MD_y \boldsymbol{h} + \frac{1}{2}(\boldsymbol{h}\boldsymbol{v})^{T}MD_y\left(\boldsymbol{v}^2\right) \\
       & - \frac{1}{2}\left(\boldsymbol{v}^3\right)^{T}MD_y \boldsymbol{h} + \frac{1}{2}\left(\boldsymbol{v}^2\right)^{T}MD_y(\boldsymbol{h}\boldsymbol{v})    - \frac{1}{2}\left(\boldsymbol{h}\boldsymbol{v}^2\right)^{T}MD_y \boldsymbol{v}                                   \\
       &   + \frac{1}{2}(\boldsymbol{v})^{T}MD_x(\boldsymbol{h}\boldsymbol{u}\boldsymbol{v}) - \frac{1}{2}\left(\boldsymbol{v}^2 \boldsymbol{u}\right)^{T}MD_x \boldsymbol{h} + \frac{1}{2}\left(\boldsymbol{h}\boldsymbol{u}\boldsymbol{v}\right)^{T}MD_x \boldsymbol{v} - \frac{1}{2}\left(\boldsymbol{h}\boldsymbol{v}^2\right)^{T}MD_x \boldsymbol{u}                                          \\
       &  + g(\boldsymbol{b}\boldsymbol{u})^TMD_x\boldsymbol{h} + g(\boldsymbol{h}\boldsymbol{b})^TMD_x\boldsymbol{u} + g(\boldsymbol{b}\boldsymbol{v})^TMD_y \boldsymbol{h} + g(\boldsymbol{h}\boldsymbol{b})^TMD_y\boldsymbol{v}                                                                                                                                                              \\
       &  - g(\boldsymbol{b}\boldsymbol{h})^T\left(R^TBN_xR\right)\boldsymbol{u} - g(\boldsymbol{b}\boldsymbol{h})^T\left(R^TBN_yR\right)\boldsymbol{v}                                                                                                                                                                                                                         \\
       & = g(\boldsymbol{h}\boldsymbol{u})^TMD_x\boldsymbol{h} + g(\boldsymbol{b}\boldsymbol{u})^TMD_x\boldsymbol{h} - g\left(\boldsymbol{u}(\boldsymbol{h}+\boldsymbol{b})\right)^TMD_x\boldsymbol{h}                                                                                                                                  \\
       &  + g\left(\boldsymbol{h}^2\right)^TMD_x\boldsymbol{u} + g(\boldsymbol{u})^TMD_x\boldsymbol{h}^2 - g(\boldsymbol{h})^T\left(R^TBN_xR\right)(\boldsymbol{h}\boldsymbol{u})                                                                                                                                                                              \\
       &  + g\left(\boldsymbol{h}^2\right)^TMD_y\boldsymbol{v} + g(\boldsymbol{v})^TMD_y\boldsymbol{h}^2 - g(\boldsymbol{h})^T\left(R^TBN_yR\right)(\boldsymbol{h}\boldsymbol{v})                                                                                                                                                                              \\
       &  + \frac{1}{2}\left(\boldsymbol{u}^2\right)^TMD_x(\boldsymbol{h}\boldsymbol{u}) + \frac{1}{2}(\boldsymbol{h}\boldsymbol{u})^TMD_x\boldsymbol{u}^2 - \frac{1}{2}\left(\boldsymbol{u}^2\right)^T\left(R^TBN_xR\right)(\boldsymbol{h}\boldsymbol{u}) + \frac{1}{2}(\boldsymbol{h}\boldsymbol{u}\boldsymbol{v})^TMD_x\boldsymbol{v}                                                                                 \\
       &  + \frac{1}{2}(\boldsymbol{v})^TMD_x(\boldsymbol{h}\boldsymbol{u}\boldsymbol{v})  - \frac{1}{2}\left(\boldsymbol{u}^2\right)^T\left(R^TBN_yR\right)(\boldsymbol{h}\boldsymbol{v}) + \frac{1}{2}(\boldsymbol{h}\boldsymbol{u}\boldsymbol{v})^TMD_y\boldsymbol{u} + \frac{1}{2}(\boldsymbol{u})^TMD_y\boldsymbol{h}\boldsymbol{u}\boldsymbol{v}                                                                                      \\
       &  - \frac{1}{2}\left(\boldsymbol{v}^2\right)^T\left(R^TBN_xR\right)(\boldsymbol{h}\boldsymbol{u}) + \frac{1}{2}\left(\boldsymbol{v}^2\right)^TMD_y\boldsymbol{v}\boldsymbol{h} +  \frac{1}{2} (\boldsymbol{v}\boldsymbol{h})^TMD_y\boldsymbol{v}^2                                                                                                                                             \\
       &  - \frac{1}{2}\left(\boldsymbol{v}^2\right)^T\left(R^TBN_yR\right)(\boldsymbol{v}\boldsymbol{h}) + g(\boldsymbol{h}\boldsymbol{b})^TMD_x\boldsymbol{u} + g(\boldsymbol{u})^TMD_x(\boldsymbol{h}\boldsymbol{b}) - g(\boldsymbol{b})^T\left(R^TBN_xR\right)(\boldsymbol{h}\boldsymbol{u})                                                                                                                                                        \\
       &  + g(\boldsymbol{h}\boldsymbol{b})^TMD_y\boldsymbol{v} + g(\boldsymbol{v})^TMD_y(\boldsymbol{h}\boldsymbol{b})  - g(\boldsymbol{b})^T\left(R^TBN_yR\right)(\boldsymbol{h}\boldsymbol{v})                                                                                                                                                                                                                         \\
       & = 0.
  \end{align*}
  The NHS part is
 \begin{align*}
       & \frac{1}{6}\left(\boldsymbol{w}^2\boldsymbol{u}\right)^TMD_x \boldsymbol{h} + \frac{1}{6}\left(\boldsymbol{w}^2\boldsymbol{h}\right)^TMD_x\boldsymbol{u}+ \frac{1}{6}\left(\boldsymbol{w}^2\boldsymbol{v}\right)^TMD_y\boldsymbol{h}+ \frac{1}{6}\left(\boldsymbol{w}^2\boldsymbol{h}\right)^TMD_y\boldsymbol{v}                                                                                                                                   \\
       & - \frac{1}{6}\left(\boldsymbol{w}^2\right)^T R^TBN_xR\left(\boldsymbol{h}\boldsymbol{u}\right) - \frac{1}{6}\left(\boldsymbol{w}^2\right)^TR^TBN_yR\left(\boldsymbol{h}\boldsymbol{v}\right)                                                                                                                                                                                                 \\
       & +\frac{\lambda}{6}\left(\left(1-\frac{\boldsymbol{\eta}^2}{\boldsymbol{h}^2}\right)\boldsymbol{u}\right)^TMD_x \boldsymbol{h} + \frac{\lambda}{6}\left(\left(1-\frac{\boldsymbol{\eta}^2}{\boldsymbol{h}^2}\right)\boldsymbol{h}\right)^TMD_x\boldsymbol{u}+ \frac{\lambda}{6}\left(\left(1-\frac{\boldsymbol{\eta}^2}{\boldsymbol{h}^2}\right)\boldsymbol{v}\right)^TMD_y\boldsymbol{h}        \\
       & + \frac{\lambda}{6}\left(\left(1-\frac{\boldsymbol{\eta}^2}{\boldsymbol{h}^2}\right)\boldsymbol{h}\right)^TMD_y\boldsymbol{v}  -\frac{\lambda}{6}\left(1-\frac{\boldsymbol{\eta}^2}{\boldsymbol{h}^2}\right)^T R^TBN_xR\left(\boldsymbol{h}\boldsymbol{u}\right) - \frac{\lambda}{6}\left(1-\frac{\boldsymbol{\eta}^2}{\boldsymbol{h}^2}\right)^TR^TBN_yR\left(\boldsymbol{h}\boldsymbol{v}\right)                                                                                                                                                    \\
       & +\frac{1}{6}\left(\boldsymbol{w}\right)^T M D_x\left(\boldsymbol{h} \boldsymbol{u} \boldsymbol{w}\right)+\frac{1}{6}\left(\boldsymbol{w}\boldsymbol{h}\boldsymbol{u}\right)^T M D_x\left(\boldsymbol{w}\right)-\frac{1}{6}\left(\boldsymbol{w}^2\boldsymbol{u}\right)^T M D_x\left(\boldsymbol{h}\right) -\frac{1}{6}\left(\boldsymbol{w}^2\boldsymbol{h}\right)^T M D_x\left(\boldsymbol{u}\right)                                             \\
       & + \frac{1}{6}\left(\boldsymbol{w}\right)^T M D_y\left(\boldsymbol{h} \boldsymbol{v} \boldsymbol{w}\right) + \frac{1}{6}\left(\boldsymbol{w}\boldsymbol{h}\boldsymbol{v}\right)^T M D_y\left(\boldsymbol{w}\right)-\frac{1}{6}\left(\boldsymbol{w}^2\boldsymbol{v}\right)^T M D_y\left(\boldsymbol{h}\right)-\frac{1}{6}\left(\boldsymbol{w}^2\boldsymbol{h}\right)^T M D_y\left(\boldsymbol{v}\right)                                                                                     \\
       & -\frac{1}{3}\left(\boldsymbol{w}\lambda\left(1-\frac{\boldsymbol{\eta}}{\boldsymbol{h}}\right)\right)^TM-\frac{\lambda}{3}\left(\left(1-\frac{\boldsymbol{\eta}}{\boldsymbol{h}}\right)\boldsymbol{u}\right)^TMD_x \boldsymbol{\eta} -\frac{\lambda}{3}\left(\left(1-\frac{\boldsymbol{\eta}}{\boldsymbol{h}}\right)\boldsymbol{v}\right)^TMD_y\boldsymbol{\eta}                                                                                          \\
       & -\frac{\lambda}{2}\left(\left(1-\frac{\boldsymbol{\eta}}{\boldsymbol{h}}\right)\boldsymbol{u}\right)^TMD_x\boldsymbol{b} - \frac{\lambda}{2}\left(\left(1-\frac{\boldsymbol{\eta}}{\boldsymbol{h}}\right)\boldsymbol{v}\right)^TMD_y\boldsymbol{b} +\frac{\lambda}{3}\left(\left(1-\frac{\boldsymbol{\eta}}{\boldsymbol{h}}\right)\boldsymbol{w}\right)^T M                                                                                                                                                                        \\
       & +\frac{\lambda}{6} \left(\boldsymbol{u}\frac{\boldsymbol{\eta}^2}{\boldsymbol{h}^2}\right)^TM D_x \boldsymbol{h}+\frac{\lambda}{3}\left(\boldsymbol{u}\right)^TM D_x \boldsymbol{\eta}-\frac{\lambda}{3} \left(\boldsymbol{u} \frac{\boldsymbol{\eta}}{\boldsymbol{h}}\right)^TM D_x \boldsymbol{\eta}-\frac{\lambda}{6}\left(\boldsymbol{u}\right)^TM D_x\left(\frac{\boldsymbol{\eta}^2}{\boldsymbol{h}}\right) \\
       & +\frac{\lambda}{2}\left(\boldsymbol{u}\left(1-\frac{\boldsymbol{\eta}}{\boldsymbol{h}}\right)\right)^TM D_x \boldsymbol{b} +\frac{\lambda}{6} \left(\boldsymbol{v}\frac{\boldsymbol{\eta}^2}{\boldsymbol{h}^2}\right)^TM D_y \boldsymbol{h}+\frac{\lambda}{3}\left(\boldsymbol{v}\right)^T M D_y \boldsymbol{\eta}-\frac{\lambda}{3}\left(\boldsymbol{v} \frac{\boldsymbol{\eta}}{\boldsymbol{h}}\right)^TM D_y \boldsymbol{\eta}\\
       & -\frac{\lambda}{6}\left(\boldsymbol{v}\right)^TM D_y\left(\frac{\boldsymbol{\eta}^2}{\boldsymbol{h}}\right)+\frac{\lambda}{2}\left(\boldsymbol{v}\left(1-\frac{\boldsymbol{\eta}}{\boldsymbol{h}}\right)\right)^TM D_y \boldsymbol{b} \\
       & = \frac{1}{6} \left(\boldsymbol{w}\right)^T MD_x\left(\boldsymbol{h}\boldsymbol{u}\boldsymbol{w}\right)+\frac{1}{6} \left(\boldsymbol{w}\boldsymbol{h}\boldsymbol{u}\right)^TMD_x\boldsymbol{w} -\frac{1}{6} \left(\boldsymbol{w}^2\right)^TR^TBN_xR\left(\boldsymbol{h}\boldsymbol{u}\right)                                                                                                                                                                 \\
       & +\frac{1}{6} \left(\boldsymbol{w}\right)^T MD_y\left(\boldsymbol{h}\boldsymbol{v}\boldsymbol{w}\right)+\frac{1}{6} \left(\boldsymbol{w}\boldsymbol{h}\boldsymbol{v}\right)^TMD_y\boldsymbol{w} -\frac{1}{6} \left(\boldsymbol{w}^2\right)^TR^TBN_yR\left(\boldsymbol{h}\boldsymbol{v}\right)                                                                                                                                                                  \\
       & +\frac{\lambda}{6} \left(\boldsymbol{h}\right)^TMD_x\boldsymbol{u} + \frac{\lambda}{6} \boldsymbol{u}^T MD_x \boldsymbol{h}- \frac{\lambda}{6} \left(\boldsymbol{1}\right)^T R^TBN_xR\left(\boldsymbol{h}\boldsymbol{u}\right)                                                                                                                                                                            \\
       & -\frac{\lambda}{6} \left(\frac{\boldsymbol{\eta}^2}{\boldsymbol{h}} \right)^TMD_x\boldsymbol{u}-\frac{\lambda}{6}\boldsymbol{u}^TMD_x \frac{\boldsymbol{\eta}^2}{\boldsymbol{h}}  +  \frac{\lambda}{6}\left(\frac{\boldsymbol{\eta}^2}{\boldsymbol{h}^2}\right)^TR^TBN_xR\left(\boldsymbol{h}\boldsymbol{u}\right)                                                                                                                               \\
       & +\frac{\lambda}{6} \left(\boldsymbol{h}\right)^TMD_y\boldsymbol{v} + \frac{\lambda}{6} \boldsymbol{v}^T MD_y \boldsymbol{h}- \frac{\lambda}{6} \left(\boldsymbol{1}\right)^T R^TBN_yR\left(\boldsymbol{h}\boldsymbol{v}\right)                                                                                                                                                                            \\
       & -\frac{\lambda}{6} \left(\frac{\boldsymbol{\eta}^2}{\boldsymbol{h}} \right)^TMD_y\boldsymbol{v}-\frac{\lambda}{6}\boldsymbol{v}^TMD_y \frac{\boldsymbol{\eta}^2}{\boldsymbol{h}}  +  \frac{\lambda}{6}\left(\frac{\boldsymbol{\eta}^2}{\boldsymbol{h}^2}\right)^TR^TBN_yR\left(\boldsymbol{h}\boldsymbol{v}\right)                                                                                                                               \\
       & =0.
  \end{align*}
  Thus, the energy is conserved.
\end{proof}

\begin{remark}[Concrete SBP operators and SATs]
  \label{rem:concrete_sbp}
  We provide explicit examples of one-dimensional second-order accurate SBP operators and their associated mass matrices for both periodic and reflecting boundary conditions. These operators are constructed on uniform grids with spacing $\Delta x$ and are employed in all numerical experiments in Section~\ref{sec:numerical_experiments}.

  For \textbf{periodic boundary conditions}, the derivative operator and mass matrix are given by~\cite{FORNBERG}
  \begin{align}
    D_x&=\frac{1}{2 \Delta x}\begin{pmatrix}
      0 & 1 & & & -1 \\
      -1 & 0 & 1 & & \\
      & \ddots & \ddots & \ddots & \\
      & & -1 & 0 & 1 \\
      1 & & & -1 & 0
    \end{pmatrix} \in \mathbb{R}^{\mathcal{N}_x \times \mathcal{N}_x},
    \quad
    M=\Delta x \operatorname{diag}(1, 1, \ldots, 1,  1).
  \end{align}
  One can verify that $M D_x + D_x^T M = 0$, as required for periodic SBP operators.

  For \textbf{reflecting boundary conditions}, the derivative operator employs the standard second-order accurate central finite difference stencil in the interior and first-order accurate one-sided finite differences at the boundaries~\cite{MATTSSON2004503}, yielding
  \begin{align}
    D_x&=\frac{1}{2\Delta x}\begin{pmatrix}
      -2 & 2 & & & \\
      -1 & 0 & 1 & & \\
      & \ddots & \ddots & \ddots & \\
      & & -1 & 0 & 1 \\
      & & & -2 & 2
    \end{pmatrix} \in \mathbb{R}^{\mathcal{N}_x \times \mathcal{N}_x},
    \quad
    M=\Delta x \operatorname{diag}(\tfrac{1}{2},1, \ldots, 1,\tfrac{1}{2}).
  \end{align}
  For these operators, the SBP property $M D_x + D_x^T M = \boldsymbol{e}_R \boldsymbol{e}_R^T - \boldsymbol{e}_L \boldsymbol{e}_L^T$ holds, where $\boldsymbol{e}_L = (1,0,\ldots,0)^T$ and $\boldsymbol{e}_R = (0,\ldots,0,1)^T$. Note that the reduced boundary weights of $1/2$ in the mass matrix arise from the trapezoidal quadrature rule.

  The \textbf{SATs} in the semi-discretization~\eqref{eq:semi_2D} act only at boundary nodes. The term $M^{-1}R^TBN_xR(\boldsymbol{h}\boldsymbol{u})$ evaluates to
  \begin{equation}
    \bigl(M^{-1}R^TBN_xR(\boldsymbol{h}\boldsymbol{u})\bigr)_{i,j} =
    \begin{cases}
      -\dfrac{2}{\Delta x}(\boldsymbol{h}\boldsymbol{u})_{1,j} & \text{if } i = 1, \\[0.8em]
      +\dfrac{2}{\Delta x}(\boldsymbol{h}\boldsymbol{u})_{\mathcal{N}_x,j} & \text{if } i = \mathcal{N}_x, \\[0.5em]
      0 & \text{otherwise},
    \end{cases}
  \end{equation}
  where the signs arise from the outer unit normal $-1$ at the left boundary and $+1$ at the right boundary, and the factor $2/\Delta x$ comes from the boundary weight of the inverted mass matrix. Analogously, the term $M^{-1}R^TBN_yR(\boldsymbol{h}\boldsymbol{v})$ evaluates to
  \begin{equation}
    \bigl(M^{-1}R^TBN_yR(\boldsymbol{h}\boldsymbol{v})\bigr)_{i,j} =
    \begin{cases}
      -\dfrac{2}{\Delta y}(\boldsymbol{h}\boldsymbol{v})_{i,1} & \text{if } j = 1, \\[0.8em]
      +\dfrac{2}{\Delta y}(\boldsymbol{h}\boldsymbol{v})_{i,\mathcal{N}_y} & \text{if } j = \mathcal{N}_y, \\[0.5em]
      0 & \text{otherwise}.
    \end{cases}
  \end{equation}
  At corner nodes, both contributions are added.
\end{remark}

\section{Computing Solitary Waves Numerically}
\label{sect:computing_solitary_waves}

The one-dimensional Serre--Green--Naghdi equations possess an exact solitary wave
solution~\cite{ranocha2025structure} of the form
\begin{equation} \label{eq:solitary_wave_analytical} \begin{split}
    h = h_{\infty} \left( 1 + \epsilon \operatorname{sech}^2\bigl( \kappa(x-Ct) \bigr) \right), \quad
    u = C \left( 1 - \dfrac{h_{\infty}}{h} \right),
  \end{split}\end{equation}
where $\epsilon = A/h_{\infty}$ with $A$ the soliton amplitude, and where
\begin{equation}
\label{eq:solitary_wave_parameters}
  \kappa^2 =\dfrac{3\epsilon}{4h_{\infty}^2(1+\epsilon)}\;,\quad
  C^2 = gh_{\infty}(1+\epsilon).
\end{equation}
Since we are not aware of a closed-form solitary wave solution of the hyperbolic SGN system, we determine a traveling wave profile numerically.
Substituting the usual ansatz with wave speed $c$, e.g., $h(t, x) = \tilde{h}(\xi)$ with $\xi = x - c t$, into the one-dimensional hyperbolic SGN equations in non-conservative form and with flat bathymetry,
\begin{equation}
  \begin{aligned}
  & h_t+(h u)_x=0, \\
  & h u_t+\frac{1}{2} g\left(h^2\right)_x+h u u_x+\left(\frac{\lambda}{3} \eta(1-\eta / h)\right)_x=0, \\
  & h w_t+h u w_x=\lambda(1-\eta / h), \\
  & \eta_t+\eta_x u=w,
  \end{aligned}
\end{equation}
we obtain
\begin{equation} \label{eq:system_solitary_wave}
  \begin{aligned}
  &-c \tilde{h}'+(\tilde{h} \tilde{u})'=0, \\
  & -c \tilde{h} \tilde{u}'+\frac{1}{2} g\left(\tilde{h}^2\right)' + \tilde{h} \tilde{u} \tilde{u}'+\left(\frac{\lambda}{3} \tilde{\eta} (1-\tilde{\eta} / \tilde{h})\right)'=0, \\
  & -c \tilde{h} \tilde{w}' +\tilde{h} \tilde{u} \tilde{w}'=\lambda(1-\tilde{\eta} / \tilde{h}), \\
  & -c \tilde{\eta}'+\tilde{\eta}' \tilde{u} = \tilde{w}.
  \end{aligned}
\end{equation}
When integrating with respect to $\xi$, we choose the integration constants for the variables $(\tilde{h},\tilde{u},\tilde{w},\tilde{\eta})$ as $(h_\infty,0,0,h_\infty)$.
Integrating the first equation of~\eqref{eq:system_solitary_wave} leads to
\begin{equation}
\label{eq:solitary_wave_h_u_relation}
- c \tilde{h} + \tilde{h} \tilde{u} = - c h_\infty.
\end{equation}
Using this to replace $\tilde{u}$ in the last equation of~\eqref{eq:system_solitary_wave} gives
\begin{equation}
-c \frac{h_\infty \tilde{\eta}'}{\tilde{h}} = \tilde{w}.
\end{equation}
Inserting this equation for $\tilde{w}$ into the third equation of~\eqref{eq:system_solitary_wave} yields
\begin{equation}
\tilde{\eta}'' + \frac{\lambda}{c^2 h_\infty^2} \tilde{\eta} = \frac{\lambda}{c^2 h_\infty^2} \tilde{h} + \frac{\tilde{h}'}{\tilde{h}} \tilde{\eta}'.
\end{equation}
Finally, writing the second equation of~\eqref{eq:system_solitary_wave} in conservative form, integrating, and inserting the expression for $\tilde{u}$ from~\eqref{eq:solitary_wave_h_u_relation} gives
\begin{equation}
-c^2 h_{\infty}\left(1-\frac{h_{\infty}}{\tilde{h}}\right)+\frac{1}{2} g \tilde{h}^2+\frac{\lambda}{3} \tilde{\eta}-\frac{\lambda}{3} \frac{\tilde{\eta}^2}{\tilde{h}}
- \frac{1}{2} g h_{\infty}^2 = 0.
\end{equation}
Denoting $\beta=\lambda / (c^2 h_\infty^2)$, we obtain the system
\begin{equation}
\begin{pmatrix}
 - \beta & \partial_\xi^2 + \beta \\
0 &  \lambda / 3
\end{pmatrix}
\begin{pmatrix}
\tilde{h} \\ \tilde{\eta}
\end{pmatrix}
=
\begin{pmatrix}
   (\partial_\xi \tilde{h}) (\partial_\xi \tilde{\eta}) / \tilde{h}  \\
    c^2 h_{\infty}\left(1-\frac{h_{\infty}}{\tilde{h}}\right) -\frac{1}{2} g \tilde{h}^2  + \frac{\lambda}{3} \frac{\tilde{\eta}^2}{\tilde{h}} + \frac{1}{2} g h_{\infty}^2
\end{pmatrix}
\end{equation}
or equivalently, but numerically better conditioned,
\begin{equation} \label{eq:solitary_wave_system}
\begin{pmatrix}
 -1 & \beta^{-1} \partial_\xi^2 + 1 \\
 0 &  1/3
\end{pmatrix}
\begin{pmatrix}
\tilde{h} \\ \tilde{\eta}
\end{pmatrix}
=
\begin{pmatrix}
   \beta^{-1} (\partial_\xi \tilde{h}) (\partial_\xi \tilde{\eta}) / \tilde{h}  \\
    \frac{c^2 h_{\infty}}{\lambda}\left(1-\frac{h_{\infty}}{\tilde{h}}\right) -\frac{g}{2\lambda} \tilde{h}^2  + \frac{1}{3} \frac{\tilde{\eta}^2}{\tilde{h}} + \frac{1}{2 \lambda} g h_{\infty}^2
\end{pmatrix}.
\end{equation}
We discretize this system in space using a Fourier collocation method that uses FFTW \cite{frigo2005design} wrapped by SummationByPartsOperators.jl~\cite{ranocha2021sbp}.
We initialize $\tilde{h} = \tilde{\eta}$ using the classical solitary wave solution~\eqref{eq:solitary_wave_analytical} (since we expect $\eta \to h$ as $\lambda \to \infty$, see Section~\ref{sect:sgn-review}).
The resulting nonlinear system is solved using the trust-region method provided by NLsolve.jl~\cite{NLsolve}, with a tolerance of $10^{-12}$ and up to $500$ iterations.

We choose $c = C$ from \eqref{eq:solitary_wave_parameters} with parameters $h_\infty = \SI{1}{m}$, $A = \SI{0.5}{m}$, and $g = \SI{9.81}{m/s^2}$, on the domain $[\SI{-40}{m}, \SI{40}{m}]$ with $\mathcal{N} = 2^{10}$ grid points to compute the solitary wave profile for various values of $\lambda$.
Figure~\ref{fig:convergence_solitary_wave} shows the maximum norm of the difference between the solitary wave profile of the hyperbolic SGN system and the classical analytical solitary wave solution of the original SGN equations, plotted as a function of the relaxation parameter $\lambda$.
As $\lambda$ increases, the error decreases as $\mathcal{O}(\lambda^{-1})$ for all components $(h, \eta, u, w)$.

\begin{figure}[htbp]
  \centering
  \includegraphics[width=0.6\textwidth]{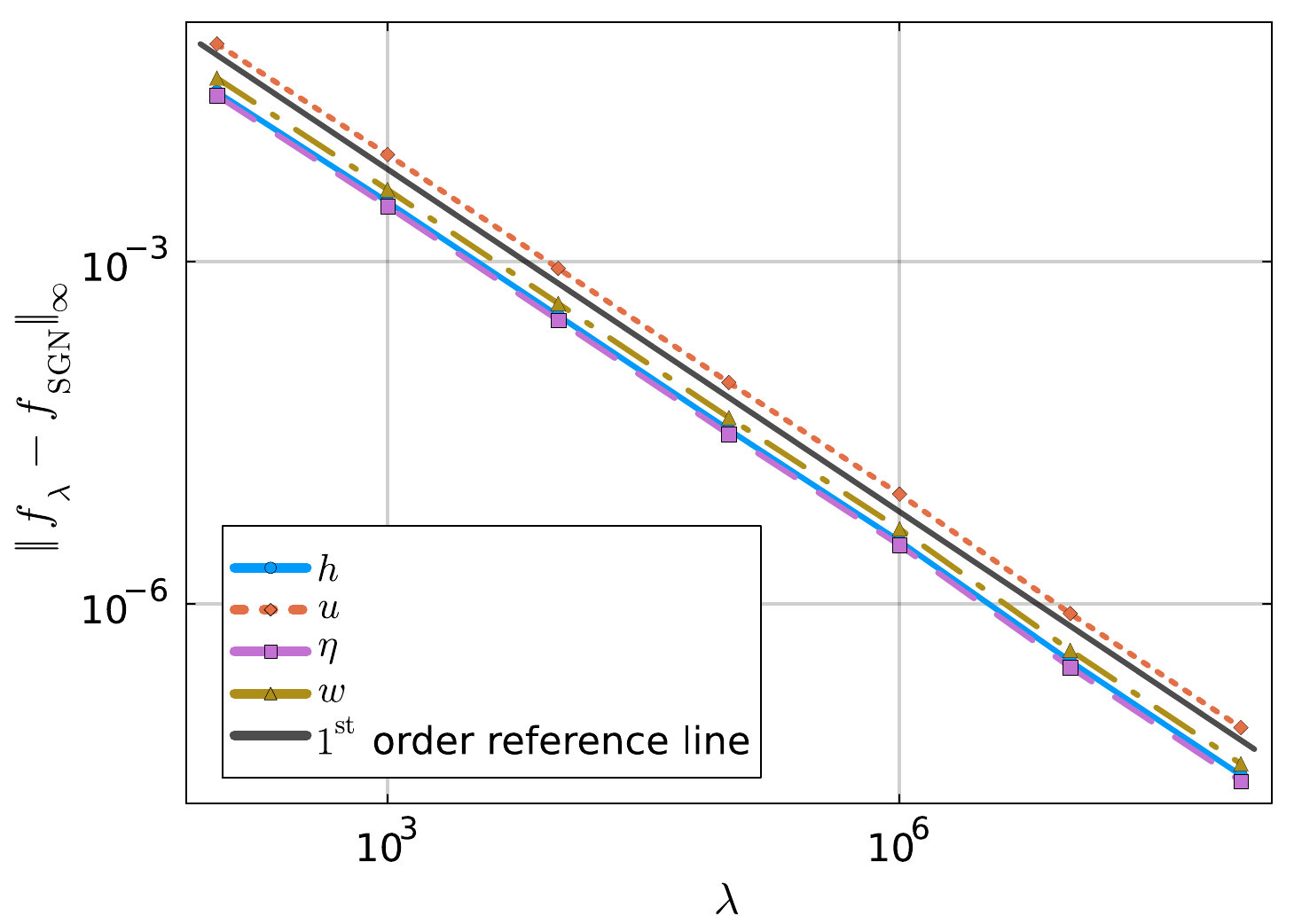}
  \caption{Maximum norm distance of the numerically computed traveling wave profile of the hyperbolization~\eqref{eq:solitary_wave_system} to the classical analytical soliton~\eqref{eq:solitary_wave_analytical} of the original SGN equations as a function of the relaxation parameter $\lambda$.
           The error decreases as $\mathcal{O}(\lambda^{-1})$, consistent with the asymptotic recovery of the original SGN equations as $\lambda \to \infty$.}
  \label{fig:convergence_solitary_wave}
\end{figure}

\section{Numerical Experiments}
\label{sec:numerical_experiments}

We have implemented the methods proposed in this work in Julia~\cite{julia2017}.
The spatial semi-discretization is implemented using vendor-agnostic kernels via
KernelAbstractions.jl~\cite{Churavy_KernelAbstractions_jl}, and time integration
is performed with OrdinaryDiffEq.jl~\cite{DifferentialEquations.jl-2017}; together,
these enable seamless execution on both CPU and GPU architectures with minimal
programming overhead. For visualization of our results, we use Plots.jl~\cite{PlotsJL}. All source code to reproduce our numerical results is available online~\cite{repo}. The code has been tested on two CPUs (Intel Core i7-1185G7, x86; Apple M4, ARM), an AMD GPU (Instinct MI210), and an NVIDIA GPU (H200).

The spatial discretization employs the second-order accurate finite difference SBP operators, as described in Remark~\ref{rem:concrete_sbp}.
Time integration is performed using an explicit third-order, five-stage
Runge--Kutta method with an embedded error estimator for adaptive step size
control. This method utilizes the \textit{First Same As Last} (FSAL) property
and has been specifically optimized for discretizations of hyperbolic
conservation laws when the time step size is constrained by stability rather
than accuracy~\cite{ranocha2021optimized}. The relative and absolute tolerances
of the time integrator are set to $10^{-6}$, and the relaxation parameter
$\lambda$ is set to 500 in all subsequent experiments unless otherwise
specified.

For one-dimensional setups, we uniformly repeat the computational domain in the
other direction to ensure compatibility with our two-dimensional equations.
This effectively reduces the problem to a one-dimensional one, with the same
solution in the other direction, as all cross terms vanish.

Unless otherwise stated, we use SI units for all quantities. The gravitational
acceleration is set to $g = \SI{9.81}{m/s^2}$.
For all experiments, we initialize the auxiliary variables of the hyperbolic
approximation as
\begin{equation}
\boldsymbol{\eta} = \boldsymbol{h}, \quad \boldsymbol{w} = -\boldsymbol{h} \cdot (D_x \boldsymbol{u} + D_y \boldsymbol{v}) + \frac{3}{2}(\boldsymbol{u} D_x \boldsymbol{b} + \boldsymbol{v} D_y \boldsymbol{b}).  \label{eq:initialization}
\end{equation}

\subsection{Convergence Studies}

To assess the spatial accuracy of our methods, we conduct convergence studies
using analytical and manufactured solutions.

\subsubsection{One-Dimensional Solitary Wave Solutions}
\label{sec:convergence_1d_soliton}

Solitary wave solutions provide a useful test case, even for our two-dimensional
implementation when the wave propagates in one spatial direction while
remaining constant in the other. We consider two complementary validation
scenarios based on the two solitary wave solutions available to us. In the
first, we set $\lambda = 30\,000$ to render the influence of the hyperbolic
approximation negligible, so that the two-dimensional hyperbolic system closely
approaches the two-dimensional Serre--Green--Naghdi equations and the classical
analytical solution~\eqref{eq:solitary_wave_analytical} of the original SGN
equations can be used as a reference, as shown in Figure~\ref{fig:convergence_solitary_wave}. In the second, we choose $\lambda = 50$
and validate against the numerically computed solitary wave profile of the
hyperbolic SGN system obtained from equation~\eqref{eq:solitary_wave_system}, which is an exact
traveling-wave solution of the hyperbolic system for any value of $\lambda$.

For both cases, the computational domain in one direction is $[-30, 30]$ and is
uniformly repeated in the other direction to ensure a two-dimensional setup,
using periodic boundary conditions in both directions. The time interval is
chosen such that the wave traverses the domain exactly once.

The convergence results are shown in Figure~\ref{fig:convergence_1d_x}. As
expected, second-order convergence is achieved for both water height and
velocity in both cases. The convergence results for propagation in the
$y$-direction are identical.

\begin{figure}[htbp]
  \centering
  \includegraphics[width=1.0\textwidth]{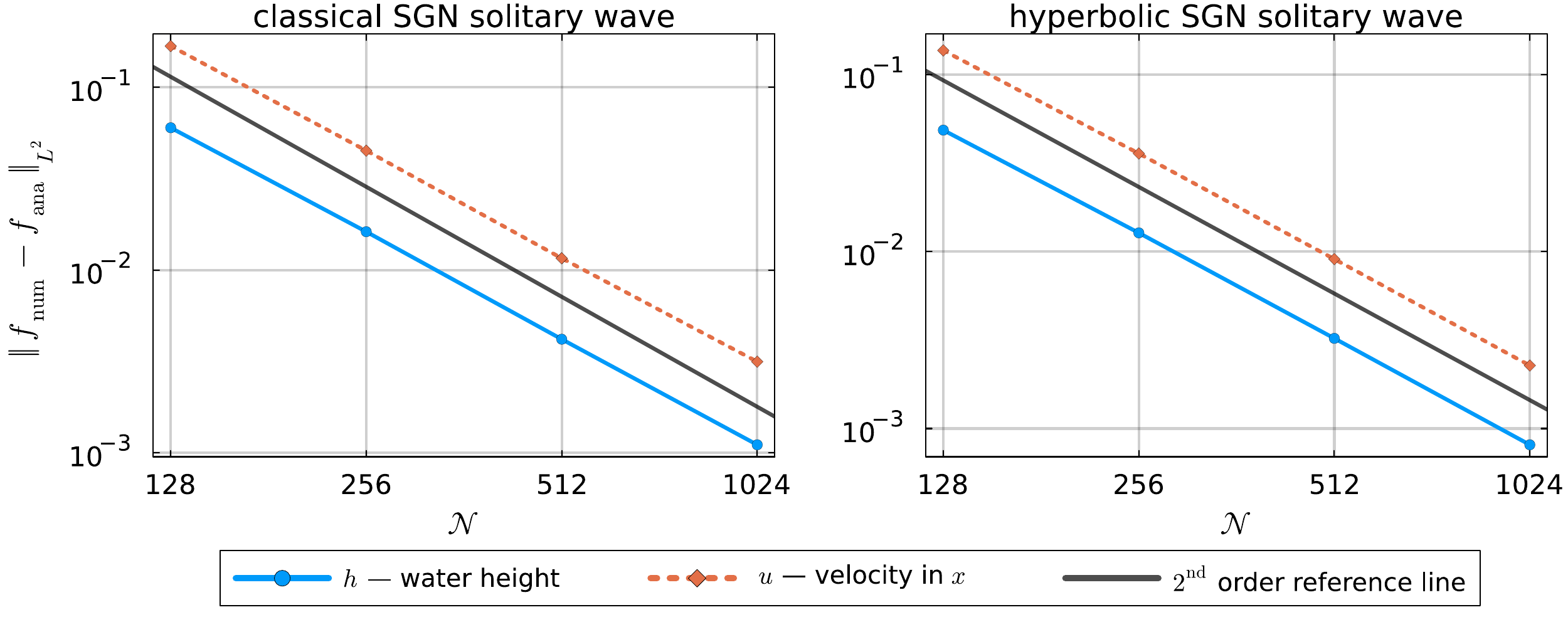}
  \caption{Convergence study for a one-dimensional solitary wave propagating in
    the $x$-direction and repeated in $y$. The computational domain in $x$ is
    $[-30, 30]$. Left: $\lambda = 30\,000$ using the analytical SGN
    solution~\eqref{eq:solitary_wave_analytical} as reference solution.
    Right: $\lambda = 50$ using the numerically computed solitary wave profile
    of the hyperbolic SGN system
    (cf.~\eqref{eq:solitary_wave_system}) as reference solution.
    Second-order convergence is achieved for both water height $h$ and velocity
    $u$ in both cases.}
  \label{fig:convergence_1d_x}
\end{figure}

While this test validates our implementation against both an analytical
reference and the numerically computed traveling-wave profile, we note that
many cross-derivative terms vanish when the wave propagates in only one
direction. Furthermore, reflecting boundary conditions are not exercised in
this test.

\subsubsection{Manufactured Solution for the Two-Dimensional System}
\label{sec:manufactured_solution}

To thoroughly validate the complete two-dimensional implementation with both
periodic and reflecting boundary conditions, we employ the method of
manufactured solutions. We choose the solution
\begin{align*}
  b    & = \frac{8}{100} \left( \cos(2 \pi x) \cos(2 \pi y) + \frac{1}{2} \cos(4 \pi x) \cos(4 \pi y) \right), \\
  h    & = 2 + \frac{1}{2} \sin(2\pi x) \sin(2\pi y) \cos(2\pi t) - b,                                         \\
  u    & = \frac{3}{10} \sin(2\pi x) \sin(2\pi t),                                                             \\
  v    & = \frac{3}{10} \sin(2\pi y) \sin(2\pi t),                                                             \\
  \eta &= h,                                                                                                   \\
   w   &= -h \cdot (\partial_x u + \partial_y v) + \frac{3}{2} \cdot (u\partial_x b + v\partial_y b)
\end{align*}
and add source terms to the equations so that these functions satisfy the governing
equations exactly. The computational domain is $[-1, 1] \times [-1, 1]$, and we analyze the grid convergence of the error at the final time $t=1$.

The results shown in Figure~\ref{fig:man_solution_convergence} confirm the
expected second-order convergence in all variables for both reflecting and
periodic boundary conditions. This comprehensive test validates our
implementation of all terms in the two-dimensional equations, including the
reflecting boundary treatment.

\begin{figure}[htbp]
  \centering
  \includegraphics[width=1.0\textwidth]{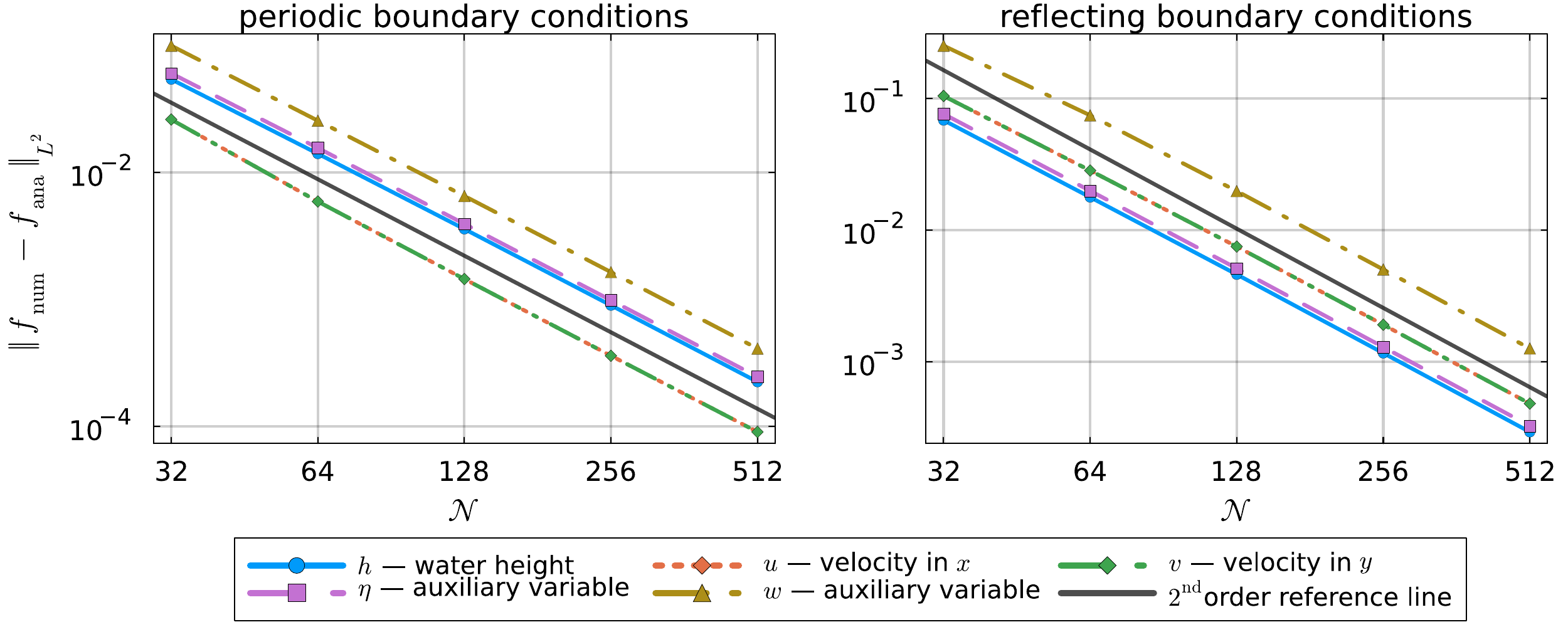}
  \caption{Convergence study using the method of manufactured solutions on the domain
    $[-1, 1] \times [-1, 1]$ with $\lambda = 500$ at final time $t=1$.
    Left: periodic boundary conditions. Right: reflecting boundary conditions.
    Second-order convergence is achieved for water height $h$, velocities
    $u$ and $v$, and auxiliary variables $\eta$ and $w$.}
  \label{fig:man_solution_convergence}
\end{figure}

\subsection{Semi-Discrete Energy Conservation}
\label{sec:energy_conservation}

As demonstrated in Theorem~\ref{theorem:conservation_reflecting}, the spatial
semi-discretization preserves the discrete energy exactly. While standard
explicit Runge--Kutta methods do not conserve nonlinear invariants such as
energy in the fully discrete scheme, a semi-discretization that conserves energy
exactly ensures that energy errors arise solely from time integration. This
typically leads to improved long-term qualitative behavior and numerical
stability. To verify the correctness of our implementation, we numerically
evaluate the time rate of change of energy directly from the semi-discretization
and confirm that it vanishes to machine precision.

We perform this verification using the setup from
Section~\ref{sec:gaussian_obstacle}, in which a solitary wave travels over a
Gaussian-shaped obstacle, providing a nontrivial fully two-dimensional flow
with wave-bathymetry interactions. Figure~\ref{fig:partial_energy_con} shows
the computed values of $\langle \partial_{\boldsymbol{q}} \boldsymbol{E}, \partial_t \boldsymbol{q}
  \rangle_M$ for both periodic and reflecting boundary conditions over an
extended simulation time. The energy time derivative remains at the level of
floating-point errors, confirming that the spatial semi-discretization conserves
energy to within machine precision. This validates both the theoretical
analysis and the correctness of our numerical implementation.

\begin{figure}[!htbp]
  \centering
  \includegraphics[width=1.0\textwidth]{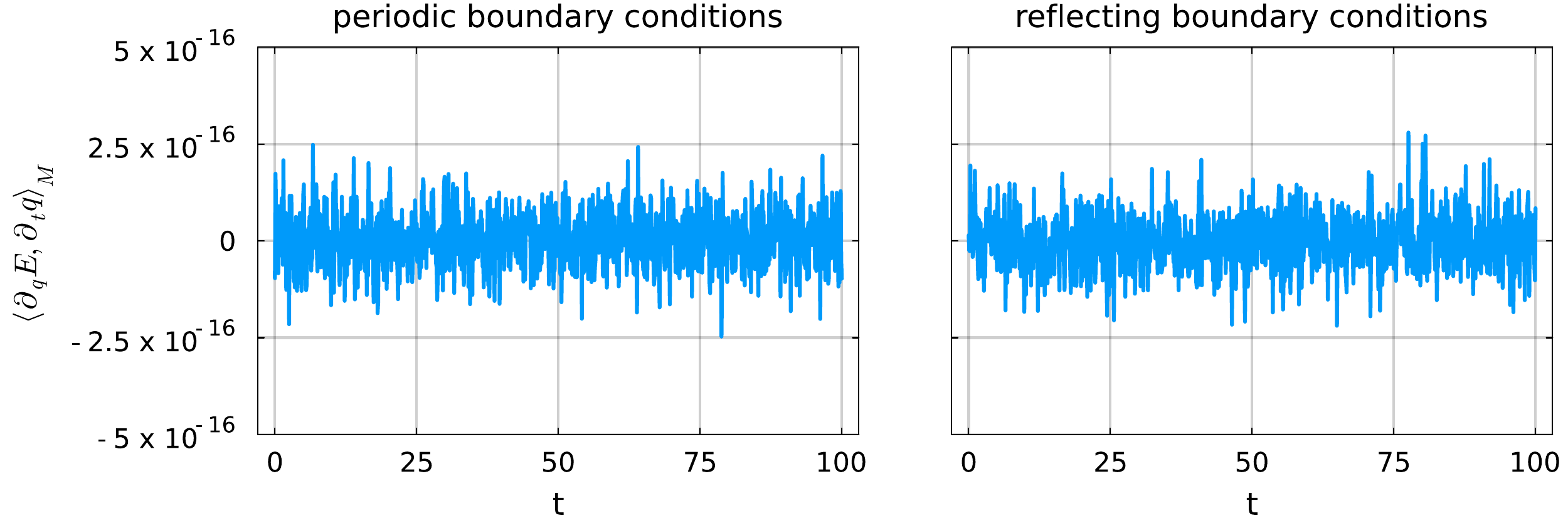}
  \caption{Numerical verification of energy conservation in the spatial semi-discretization using the wave over Gaussian test case from Section~\ref{sec:gaussian_obstacle}. The quantity $\langle \partial_{\boldsymbol{q}} \boldsymbol{E}, \partial_t \boldsymbol{q} \rangle_M$ is computed numerically. The energy time derivative fluctuates at machine precision throughout the simulation, confirming exact energy conservation of the spatial semi-discretization.}
  \label{fig:partial_energy_con}
\end{figure}

\subsection{Dingemans Experiment}
\label{sec:dingemans}

We validate our numerical model by comparing its results with
experimental data from Dingemans~\cite{Dingemans1994, Dingemans1997}. This
experiment involves wave propagation over a submerged trapezoidal bar and has
become a standard benchmark for dispersive wave models.
The bathymetry and initial condition, along with the positions of the six wave
gauges, are shown in Figure~\ref{fig:dingemans_bathymetry_initial_condition}.

\begin{figure}[!htb]
  \centering
  \includegraphics[width=0.6\textwidth]{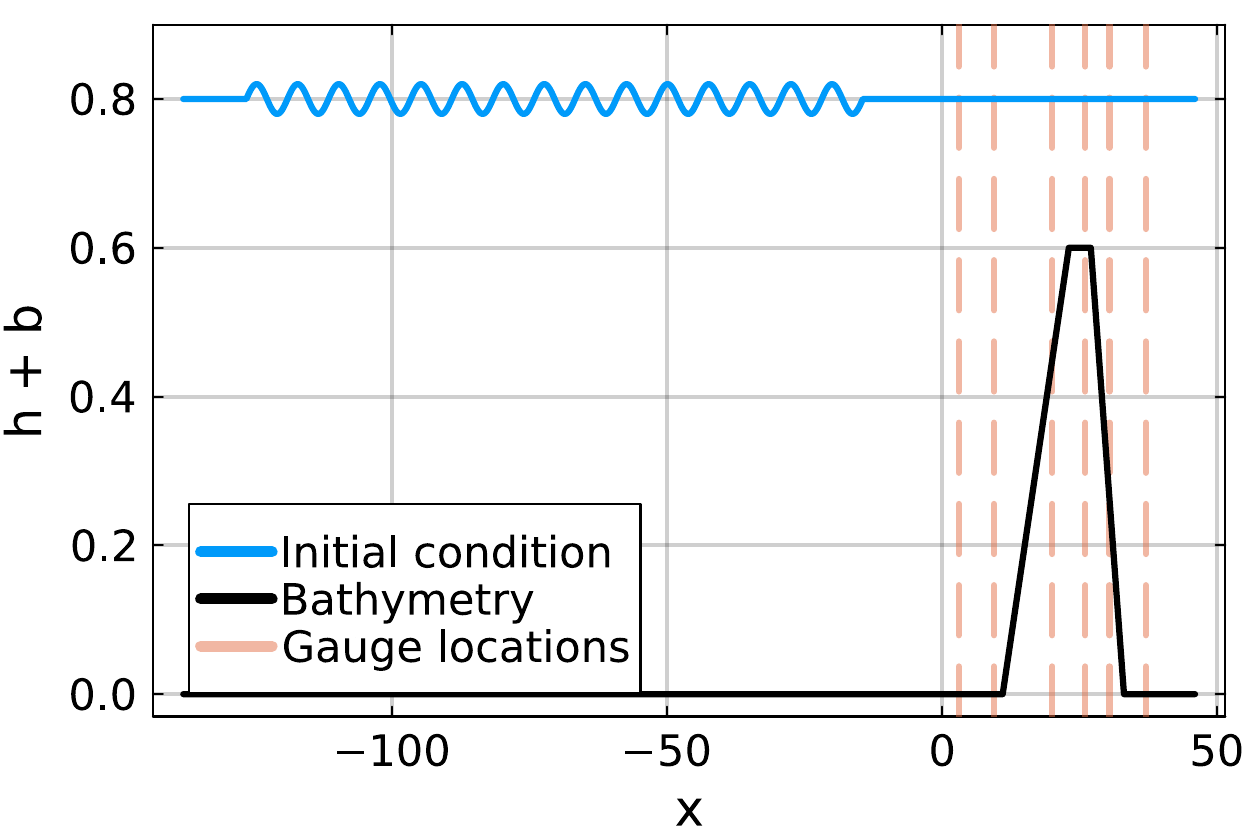}
  \caption{Initial setup for the Dingemans experiment showing the trapezoidal
    bathymetry $b(x)$, initial water surface elevation $h(x,0) + b(x)$,
    and locations of the six wave gauges (vertical dashed lines).
    The domain extends from $x = -138$ to $x = 46$ with periodic
    boundary conditions.}
  \label{fig:dingemans_bathymetry_initial_condition}
\end{figure}

Dingemans' original experiment~\cite{Dingemans1994, Dingemans1997} employed a
wave maker at $x = 0$ to generate waves with amplitude $A = 0.02$ moving to the
right. For our numerical simulation, we use a spatial domain of $[-138, 46]$ in
the $x$-direction with periodic boundary conditions. The setup is uniformly
repeated in the $y$-direction to ensure a two-dimensional implementation,
effectively reducing to a one-dimensional problem. Due to the second-order spatial accuracy of our implementation, we
employ a relatively fine grid resolution of $\Delta x = 0.008$ to ensure
adequate resolution of the wave dynamics.

The initial condition consists of a constant water depth $h + b = 0.8$ with a
sinusoidal perturbation of amplitude $A = 0.02$. The phase of the perturbation
and corresponding velocity field are determined from the linear dispersion
relation of the Euler equations, following the approach of~\cite{svard2025novel,
  lampert2025structure}. The horizontal offset of the perturbation is adjusted manually to
achieve good phase agreement with the experimental data at the first wave
gauge.

The bottom topography is defined by a trapezoidal bar
\begin{equation}
  b(x) = \begin{cases}
    0.6 \cdot \dfrac{x - 11.01}{12.03} & \text{if } 11.01 \leq x < 23.04, \\[0.2em]
    0.6                                & \text{if } 23.04 \leq x < 27.04, \\[0.2em]
    0.6 \cdot \dfrac{33.07 - x}{6.03}  & \text{if } 27.04 \leq x < 33.07, \\[0.2em]
    0                                  & \text{otherwise.}
  \end{cases}
\end{equation}

Figure~\ref{fig:dingemans_results} compares the numerical solution with the
experimental data at the six wave gauges. The agreement is excellent at the first
four gauges, which are located before and on the trapezoidal bar. At the final
two gauges, positioned beyond the bar, the agreement is somewhat reduced,
though the numerical solution still captures the correct amplitude and
qualitative behavior of the waves. This behavior is related
to the poor dispersive properties of the SGN equations, and
matches previous results obtained with non-enhanced Boussinesq-type models
in the literature (see, e.g., \cite{FILIPPINI2015109,ranocha2025structure}).

\begin{figure}[htb]
  \centering
  \includegraphics[width=0.8\textwidth]{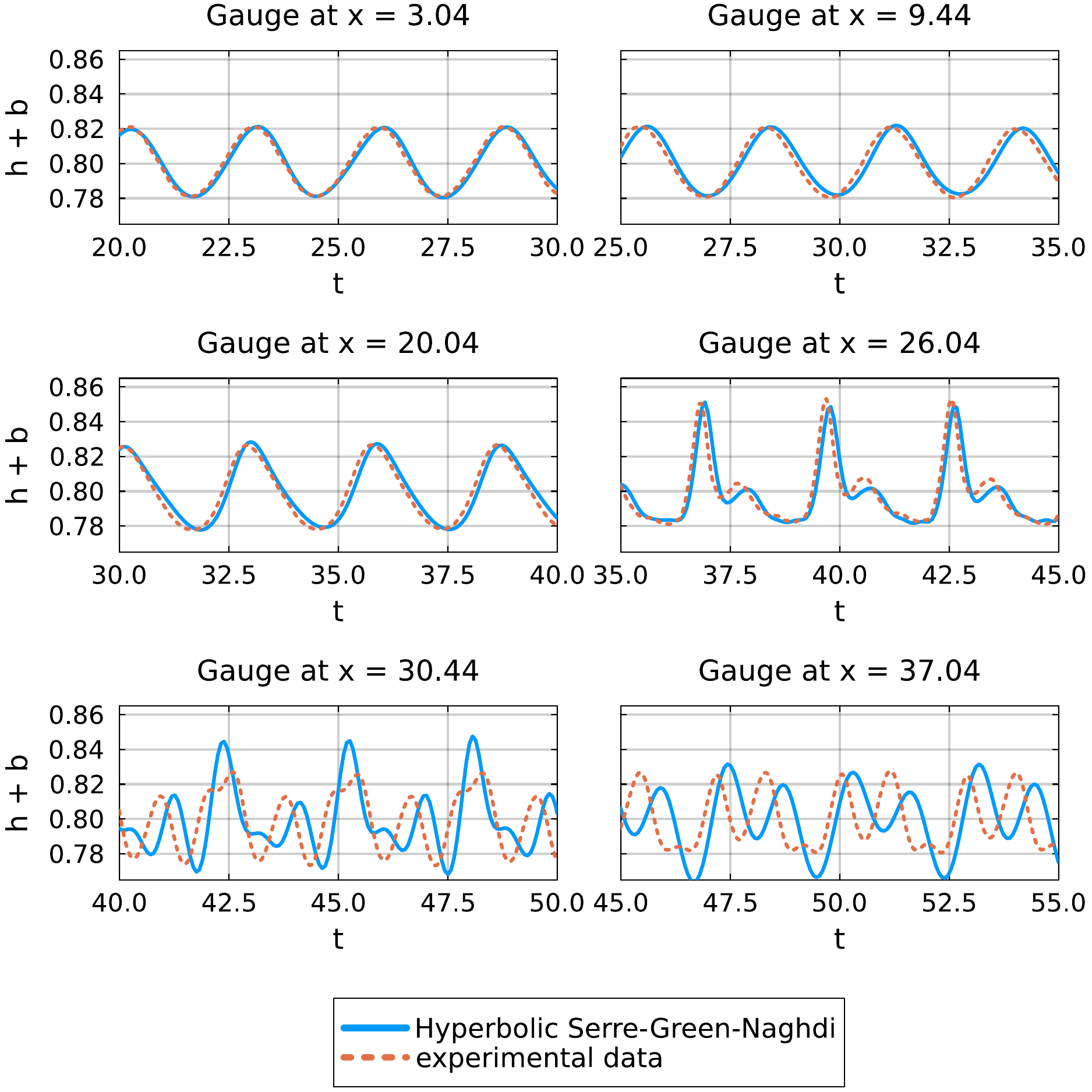}
  \caption{Comparison of numerical solution (solid blue line) with experimental data from Dingemans~\cite{Dingemans1994, Dingemans1997} (orange dotted line) at six wave gauge locations. The numerical solution shows excellent agreement at gauges 1--4 and reasonable agreement at gauges 5--6.}
  \label{fig:dingemans_results}
\end{figure}

\begin{figure}[htbp]
  \centering
  \includegraphics[width=1.0\textwidth]{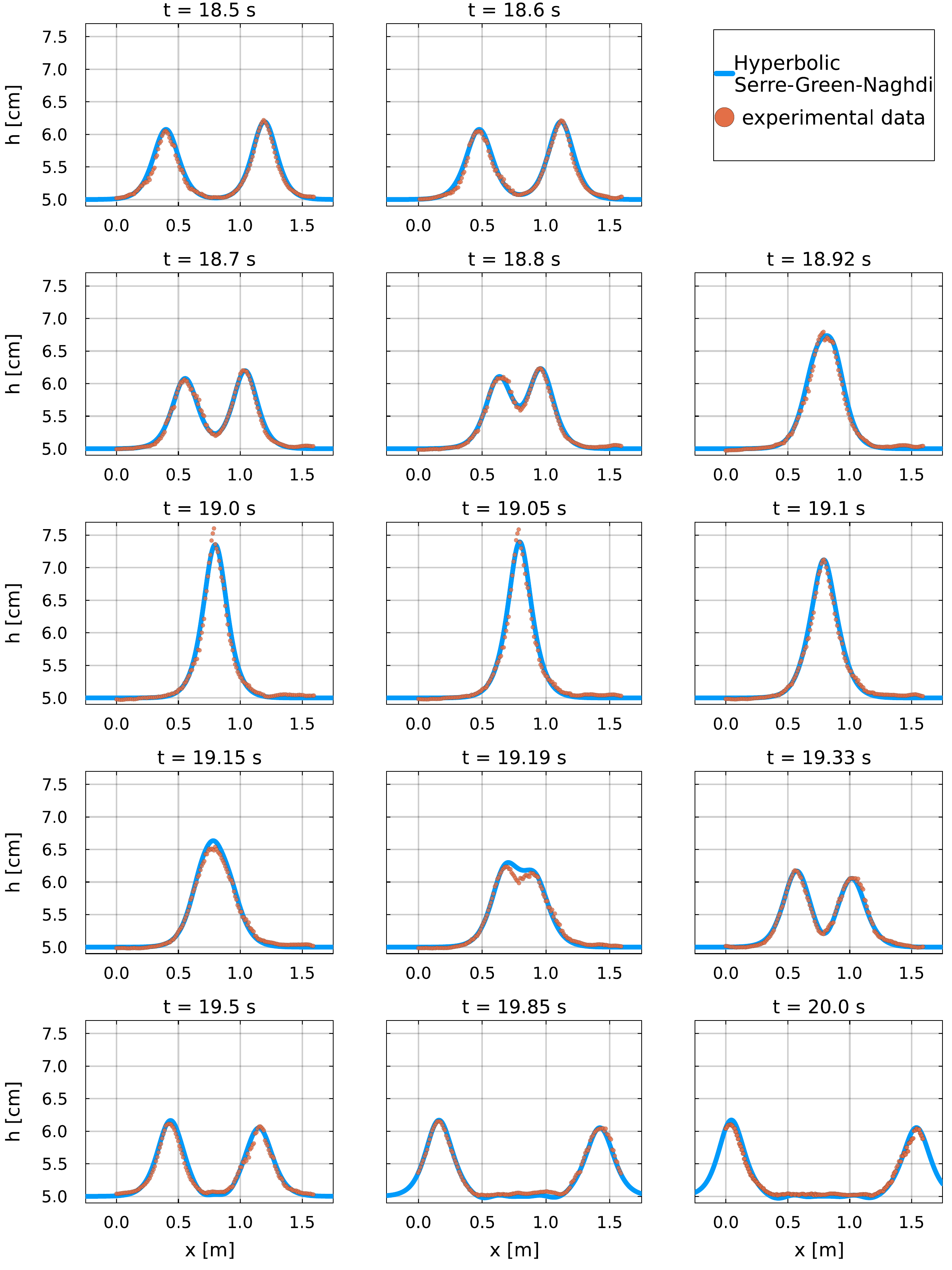}
  \caption{Head-on collision of two solitary waves. Comparison of numerical
    solution (solid blue line) with experimental data from~\cite{wave_interactions} (orange dots) at various times during
    the collision process. The computational domain is $[-10, 10]$
    with initial wave amplitudes $A = 0.01077$ (right-moving) and
    $A = 0.01195$ (left-moving). Very good agreement is observed
    throughout the collision and separation phases.}
  \label{fig:colliding_waves}
\end{figure}

\subsection{Head-on Collision of Solitary Waves}
\label{sec:colliding_waves}

To further validate our numerical model, we simulate the head-on collision of
two solitary waves and compare with experimental data
from~\cite{wave_interactions}. This test has already been used previously to
validate various wave models as in~\cite[Section~4.2]{DUTYKH20113035}
or~\cite[Section~5.2.1]{DUTYKH_CLAMOND_MILEWSKI_MITSOTAKIS_2013}.

The computational domain in the $x$-direction is $[-10, 10]$, chosen
sufficiently large to avoid boundary effects on the region of interest. As in
the previous test, the domain is uniformly extended in the $y$-direction for
the two-dimensional code, with vanishing cross-derivative terms. The initial
condition consists of two solitary waves: one with amplitude $A = 0.01077$
centered at $x = 0.4$ moving rightward, and another with amplitude $A =
  0.01195$ centered at $x = 1.195$ moving leftward. The simulation begins at $t =
  \SI{18.5}{s}$ to match the experimental timing.

Figure~\ref{fig:colliding_waves} shows very good agreement between our
numerical solution and the experimental data throughout the collision process.
The numerical method accurately captures the wave amplification during
collision, the subsequent separation, and the formation of the trailing wave
train. This validates both the hyperbolic approximation of the
Serre--Green--Naghdi equations and our numerical implementation.

\subsection{Riemann Problem}
\label{sec:riemann_problem}

Motivated by the setup in~\cite{tkachenko2023extended,ranocha2025structure}, we consider a
Riemann-type configuration with a smoothed jump given by
\begin{equation}
  h(x,0) = h_R + \frac{h_L - h_R}{2}\bigl(1 - \tanh(x/2)\bigr),
  \qquad
  u(x,0) = 0.
\end{equation}
Using the Riemann invariants of the shallow-water equations together with the
Whitham modulation theory for the SGN system~\cite{ElGrimshawSmyth2006,gavrilyuk2020stationary,tkachenko2023extended}, one can obtain
approximate expressions for the mean state $\left(h^*,u^*\right)$ separating
the rarefaction region from the dispersive shock zone
\begin{equation}
  h^* = \frac{\bigl(\sqrt{h_L} + \sqrt{h_R}\bigr)^2}{4},
  \qquad
  u^* = 2\bigl(\sqrt{g h^*} - \sqrt{g h_R}\bigr),
  \qquad
  a^{+} = \delta_0 - \frac{1}{12}\delta_0^2 + \mathcal{O}(\delta_0^3).
\end{equation}
Here $a^{+}$ denotes a second-order asymptotic approximation of the leading soliton amplitude with $\delta_0 = \vert h_L - h_R \vert$. In the following, we choose $h_L = 1.8$ and $h_R = 1.0$ and solve up to time $t = 47.434$. The spatial domain in the $x$-direction is $[-600,600]$ with $\Delta x = 0.3$.
As before, the setup is uniformly repeated in the $y$-direction.
For these values of $h_L$ and $h_R$, we theoretically expect $h^{*} \approx 1.37$ and $a^{+} \approx 0.747$.
Note that $a^{+}$ is the amplitude above the background height of $1$.
Hence, the maximum water height is $h^m = 1 + a^{+} \approx 1.74$.
Figure~\ref{fig:riemann_problem} displays the solution in the interval $[-300,300]$. Both $h^m$ and $h^*$ are in good agreement with the theoretical predictions.

\begin{figure}[H]
  \centering
  \includegraphics[width=0.6\textwidth]{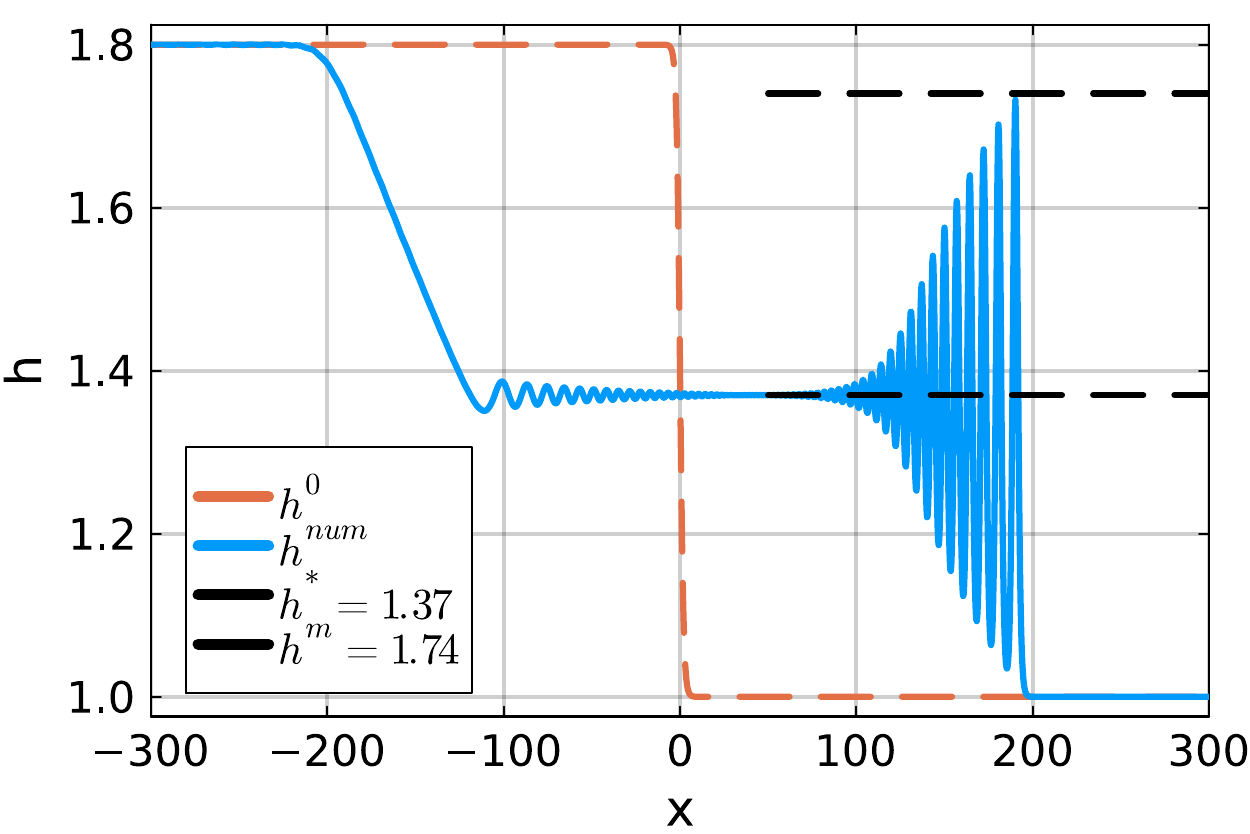}
  \caption{Riemann problem solved on the domain $[-600,600]$ with $\Delta x = 0.3$. The theoretically predicted intermediate water height $h^* \approx 1.37$ and the leading soliton amplitude $h^m \approx 1.74$ are well captured by the numerical solution.}
  \label{fig:riemann_problem}
\end{figure}

\subsection{Favre Waves}
\label{sec:Favre_waves}

We also test our semi-discretization using Favre waves, following the setups in~\cite{Chassagne_Filippini_Ricchiuto_Bonneton_2019,JOUY2024170}. A
bore represents a transition between two streams with different water depths.
The intensity of such a bore can be characterized by the Froude number
\begin{equation}\label{eq:froude_numbers}
  \text{Fr} = \sqrt{(1+\varepsilon)(1+\varepsilon/2)},
\end{equation}
where $\varepsilon$ denotes the
nonlinearity parameter. Larger values of $\varepsilon$ correspond to greater
height differences between the water streams. For small Froude numbers
($\text{Fr} \leq 1.3$), the bore exhibits a smooth profile followed by undular
wave trains, whereas larger Froude numbers lead to increasingly steep bores
that eventually form breaking waves. In this work, we focus on bores with small
Froude numbers, commonly referred to as Favre waves.

The initial condition for this setup is
\begin{equation}
  \begin{aligned}
     & h(x, t=0):=h_0+\frac{\llbracket h \rrbracket}{2}\left\{1-\tanh \left(\frac{x-x_0}{\alpha}\right)\right\}, \\
     & u(x, t=0):=u_0+\frac{\llbracket u \rrbracket}{2}\left\{1-\tanh \left(\frac{x-x_0}{\alpha}\right)\right\},
  \end{aligned}
\end{equation}
where
\begin{equation}
  \llbracket h \rrbracket:=\varepsilon h_0,
\end{equation}
with $\varepsilon$ being the nonlinearity parameter. The velocity jump $\llbracket u
  \rrbracket$ satisfies the shallow-water Rankine--Hugoniot relations, in
particular
\begin{equation}
  \llbracket u \rrbracket=\sqrt{g \frac{h_1+h_0}{2 h_0 h_1}} \llbracket h \rrbracket,
\end{equation}
where $h_1 = h_0 + \llbracket h \rrbracket = h_0(1+\varepsilon)$ is the water depth on the downstream side of the bore.
We use $h_0 = 0.2$, $\alpha = 5 h_0 = 1.0$, $x_0 = 0$, and $u_0 = 0$.

\subsubsection{Evolution of the Free-Surface Elevation}

First, we study how the bore profile evolves over time for different nonlinearity values.
The computational domain in the $x$-direction is chosen as $[-50, 50]$ with
periodic boundary conditions. We use a grid spacing of $\Delta x = 0.05$ to
adequately resolve the wave dynamics. As in the previous tests, the setup is
uniformly repeated in the $y$-direction, effectively reducing to a
one-dimensional problem.

Following~\cite{Chassagne_Filippini_Ricchiuto_Bonneton_2019, JOUY2024170}, we
examine the short-time bore evolution for three nonlinearity values
$\varepsilon \in \{0.1,0.2,0.3\}$ and compare the free-surface elevation at
different dimensionless times $\tilde{t}:=t \sqrt{g / h_0}$ with fully
nonlinear potential flow solutions from~\cite{Wei_Kirby_Grilli_Subramanya_1995}. Figure~\ref{fig:Favre_waves} shows our
results, which agree well with the fully nonlinear potential flow solutions. For
larger values, such as $\varepsilon=0.3$, some limitations become apparent due
to the weakly dispersive character of the model.

\begin{figure}[H]
  \centering
  \includegraphics[width=1.0\textwidth]{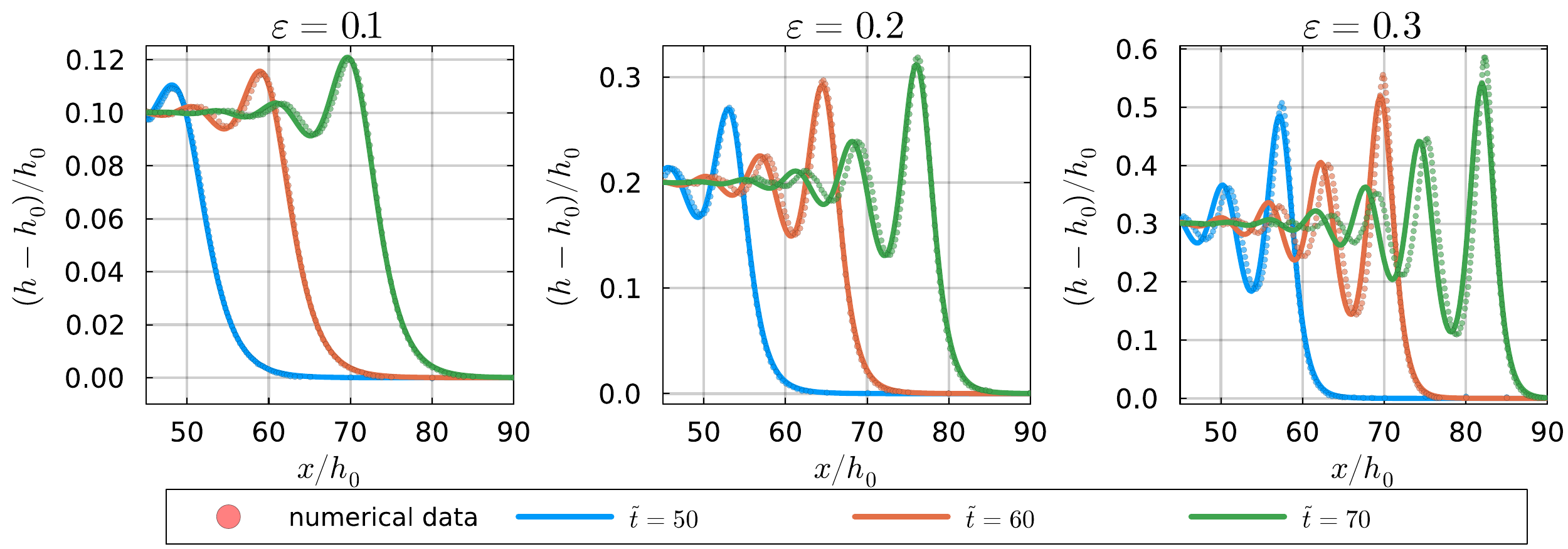}
  \caption{Free-surface elevation at different dimensionless times
    $\tilde{t}:=t \sqrt{g / h_0}$ for three nonlinearity values
    $\varepsilon \in \{0.1,0.2,0.3\}$. The numerical solutions (solid lines)
    are compared with fully nonlinear potential flow solutions from
    \cite{Wei_Kirby_Grilli_Subramanya_1995} (data points). The computational
    domain is $[-50, 50]$ with $\Delta x = 0.05$. Good agreement is observed
    for smaller nonlinearity values, with some deviations for
    $\varepsilon=0.3$ related to the weakly dispersive character of the model.}
  \label{fig:Favre_waves}
\end{figure}

\subsubsection{Comparison with Experimental Data}

To further validate our approach, we compare the maximal wave amplitude as a
function of the Froude number with experimental measurements.

We compare the maximal amplitude of the Favre waves after the first wave
has traveled a distance of \SI{63.5}{m}, matching the experimental setup of
Favre~\cite{Favre1935} and Treske~\cite{Treske1994}. The computational
domain in the $x$-direction is chosen as $[-150, 150]$ with periodic boundary
conditions and a grid spacing of $\Delta x = 0.075$. The results are shown in
Figure~\ref{fig:Froude_numbers}. The simulations are run for
$\text{Fr} \leq 1.25$, giving
good agreement with the experimental data. For larger Froude numbers, wave breaking occurs,
which cannot be modeled by the SGN equations alone~\cite{Kazolea_Ricchiuto_2018}.

\begin{figure}[H]
  \centering
  \includegraphics[width=0.6\textwidth]{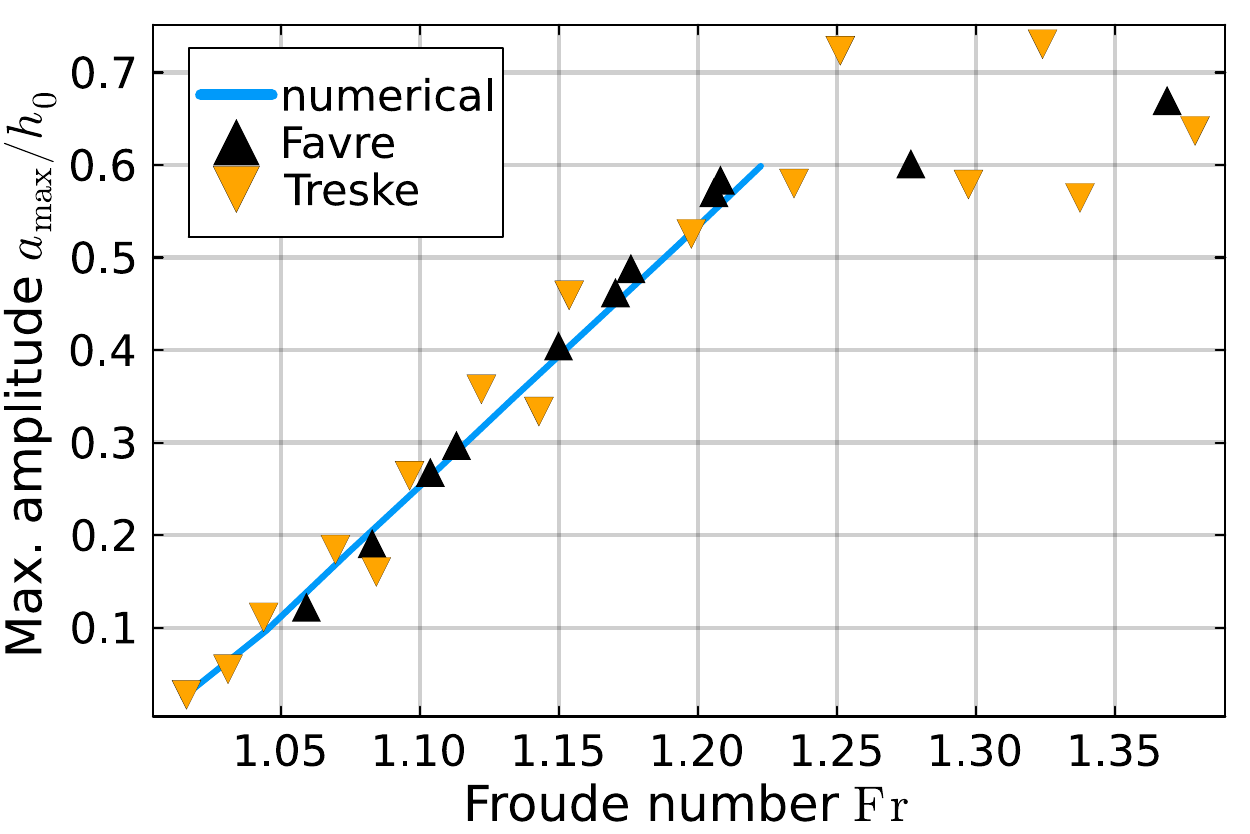}
  \caption{Maximal amplitude of the Favre waves after traveling a fixed distance versus
    Froude number. The numerical results (solid line) are compared with
    experimental data from Favre~\cite{Favre1935} and Treske~\cite{Treske1994}.
    Good agreement is observed for $\text{Fr} \leq 1.25$, with numerical
    simulations breaking down for larger Froude numbers due to wave breaking
    not captured by the SGN equations. The computational domain is $[-150, 150]$ with $\Delta x = 0.075$.}
  \label{fig:Froude_numbers}
\end{figure}

\subsection{Reflection of Solitary Waves from a Vertical Wall}
\label{sec:reflecting_wall}

To validate the implementation of reflecting boundary conditions, we compare
our numerical results with those reported in~\cite[Section~4.2]{Mitsotakis}.

The computational domain in the $x$-direction is $[-100, 0]$ with $\Delta x =
  0.1$. As before, the setup is uniformly repeated in the $y$-direction. The initial
condition is a solitary wave of amplitude $A$ centered at $x = -50$ propagating
rightward toward the vertical wall at $x = 0$. We consider two amplitudes: $A = 0.075$ and $A = 0.65$.

Following~\cite{Mitsotakis}, we present results in dimensionless time units,
where the dimensionless time $t^*$ relates to physical time $t$ (in seconds)
through $t = t^* \sqrt{h_{\infty}/g}$ with $h_{\infty} = \SI{1}{m}$ being the constant
water depth away from the wave and $g = \SI{9.81}{m/s^2}$ still denoting the gravitational acceleration.

Figures~\ref{fig:reflecting_wave_A_0.075} and~\ref{fig:reflecting_wave_A_0.65}
show excellent agreement between our numerical solution and the reference data
from~\cite{Mitsotakis}. Minor deviations only occur in
Figure~\ref{fig:reflecting_wave_A_0.65} at dimensionless times $t^* = 38$ and
$t^* = 42$. We note that the original paper~\cite{Mitsotakis} reported one
snapshot at $t^* = 41$ rather than $t^* = 42$; after consultation with the
authors, this was confirmed to be a typographical error in the original
publication. This validation demonstrates the accuracy of both the hyperbolic
approximation and our implementation of reflecting boundary conditions.

\begin{figure}[htb]
  \centering
  \includegraphics[width=0.9\textwidth]{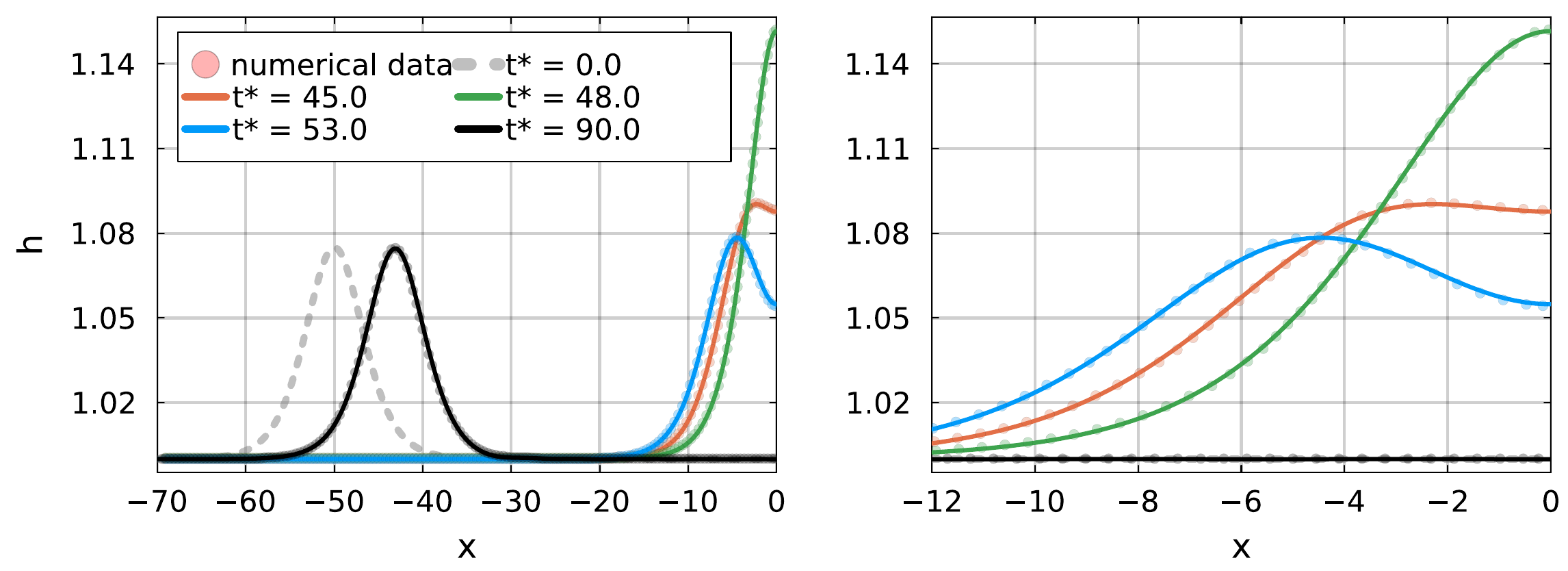}
  \caption{Reflection of a solitary wave with amplitude $A = 0.075$ from a vertical
    wall at $x = 0$. The dots represent numerical data from
    \cite{Mitsotakis}, and solid lines show our numerical
    solution at dimensionless times $t^* = 24, 45, 48, 53, 90$.
    Left: full domain. Right: zoom near the wall. Domain: $[-100, 0]$
    with $\Delta x = 0.1$ and $h_{\infty} = 1$.}
  \label{fig:reflecting_wave_A_0.075}
\end{figure}

\begin{figure}[htb]
  \centering
  \includegraphics[width=0.9\textwidth]{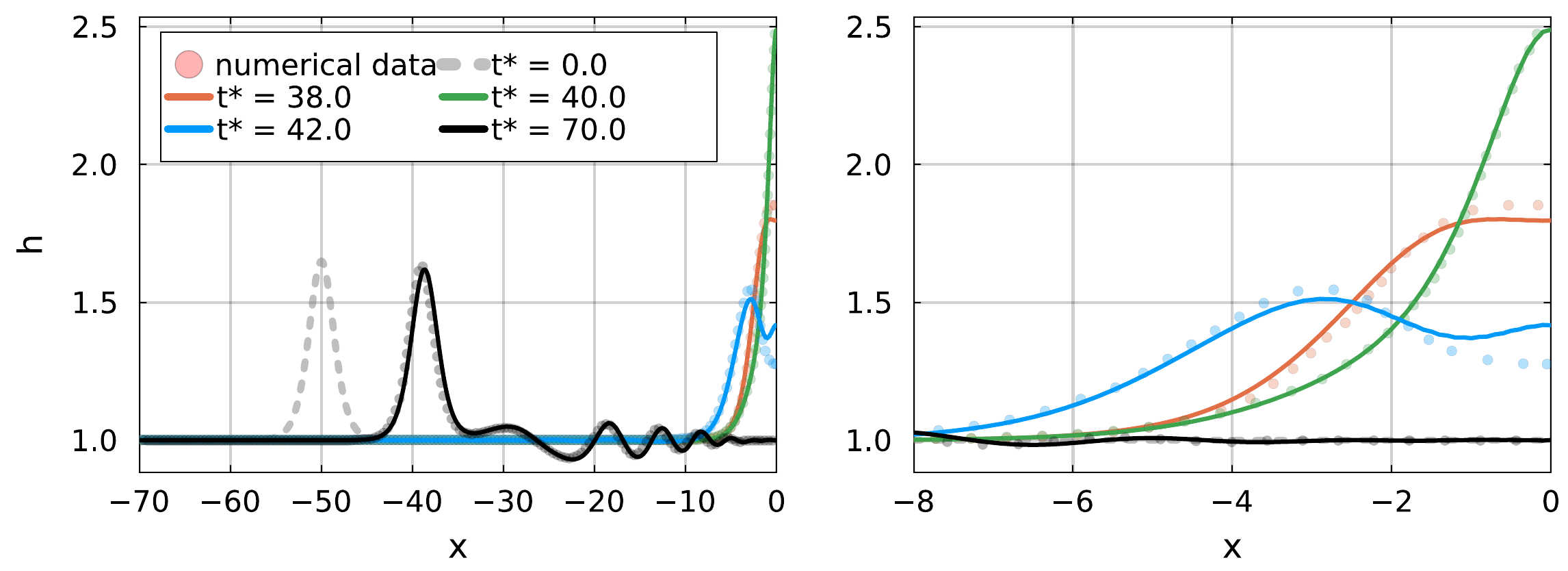}
  \caption{Reflection of a solitary wave with amplitude $A = 0.65$ from a vertical
    wall at $x = 0$, representing strongly nonlinear reflection. The dots
    represent numerical data from~\cite{Mitsotakis}, and solid
    lines show our numerical solution at dimensionless times
    $t^* = 0, 28, 38, 42, 70$. Left: full domain. Right: zoom near the wall.
    Domain: $[-100, 0]$ with $\Delta x = 0.1$ and $h_{\infty} = 1$.}
  \label{fig:reflecting_wave_A_0.65}
\end{figure}

\subsection{Solitary Wave over a Gaussian Obstacle}
\label{sec:gaussian_obstacle}

In this section, we perform a fully two-dimensional simulation and compare our
results with numerical data from~\cite[Section~4.1.7]{Busto2021}. This test
exercises both spatial dimensions and validates the implementation for
wave-bathymetry interactions.

The computational domain is $[-5, 35] \times [-10, 10]$ with periodic boundary
conditions on all boundaries and uniform grid spacing $\Delta x = \Delta y =
  0.025$. A solitary wave front with amplitude $A = 0.0365$ starts at $x = -3$,
propagating over a constant water depth of $h_{\infty} = 0.2$ toward a Gaussian
obstacle. The bathymetry is given by
\begin{equation}
  b(x,y) = 0.1 \exp\left(-\frac{x^2 + y^2}{2}\right).
\end{equation}

Figure~\ref{fig:gaussian_setup} illustrates the time evolution of the water
surface elevation as the solitary wave approaches and interacts with the
bathymetric feature. The wave undergoes complex three-dimensional deformation
as it propagates over the Gaussian bump, generating a dispersive wave train
behind the leading wave front.

\begin{figure}[htb]
  \centering
  \includegraphics[width=1.0\textwidth]{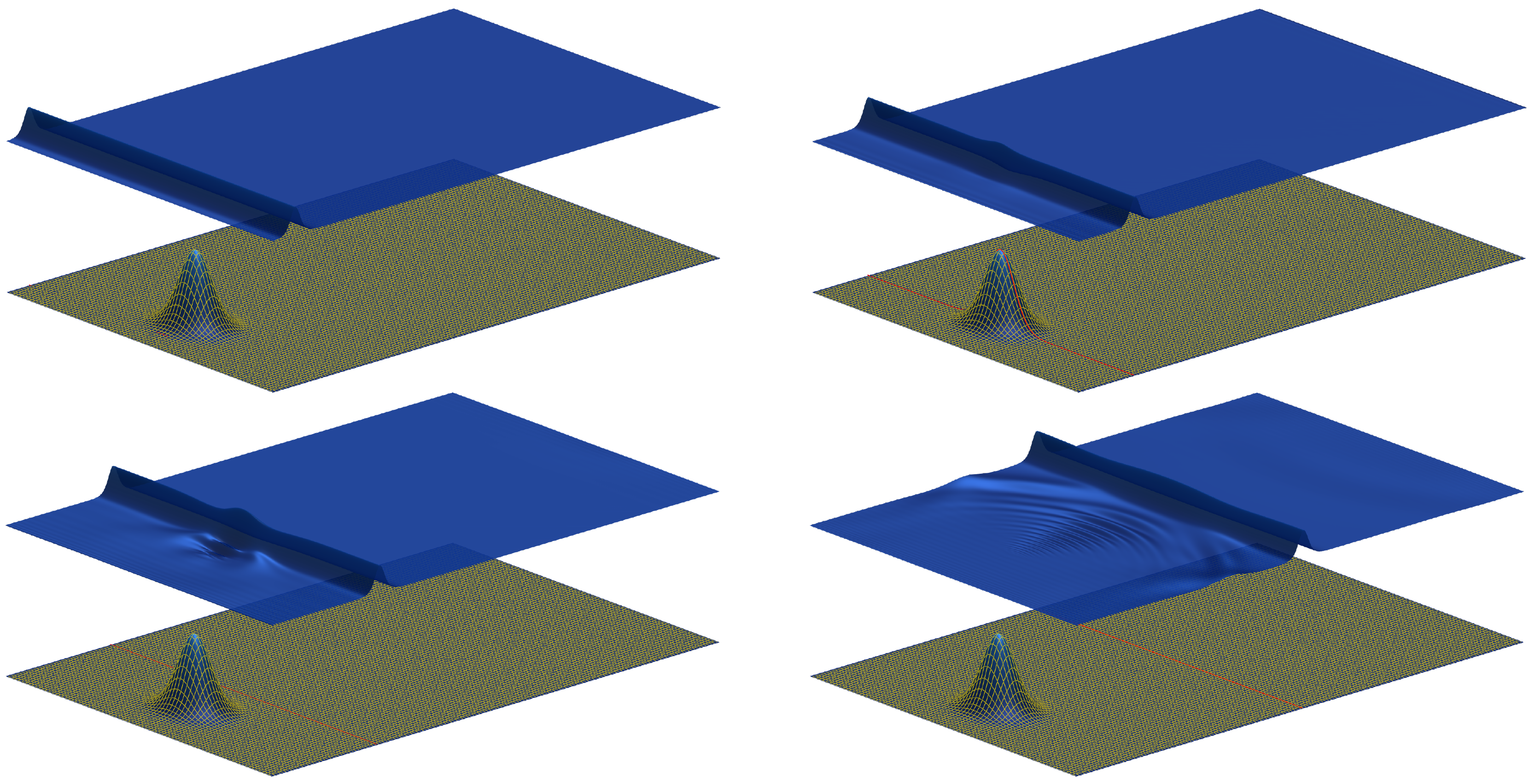}
  \caption{Time evolution of a solitary wave propagating over a Gaussian obstacle
    at $t = 0, 2, 5, 12$ s. The bottom surface shows the bathymetry $b$
    with overlaid computational grid (golden), while the upper surface
    displays the water surface elevation $h + b$ (blue). The red contour
    line on the water surface highlights the wave front location. The
    computational grid is coarsened by a factor of 10 for visualization
    clarity. The wave deforms three-dimensionally as it passes over the
    obstacle, generating a dispersive tail visible in the final snapshot. Created using Makie.jl~\cite{DanischKrumbiegel2021}.}
  \label{fig:gaussian_setup}
\end{figure}

Following~\cite{Busto2021}, we examine the cross-section of the water height
along $y = 0$ at time $t = \SI{12}{s}$. Figure~\ref{fig:busto_gaussian} presents this
comparison. The apparent dip in water height near $x = 0$ is simply due to the
presence of the Gaussian bump (recall that $h$ denotes the water height above the bathymetry $b$).

The numerical solutions show excellent agreement, particularly in the position
of the leading solitary wave and the amplitude and structure of the dispersive
tail trailing behind it. The minor deviation visible near $x \approx 28$ arises
from the periodic boundary conditions; this artifact disappears when using a
larger domain or employing reflecting boundary conditions. This agreement
validates the hyperbolic approximation of the Serre--Green--Naghdi equations and
our numerical implementation for fully two-dimensional wave propagation over
variable bathymetry.

\begin{figure}[htb]
  \centering
  \includegraphics[width=0.6\textwidth]{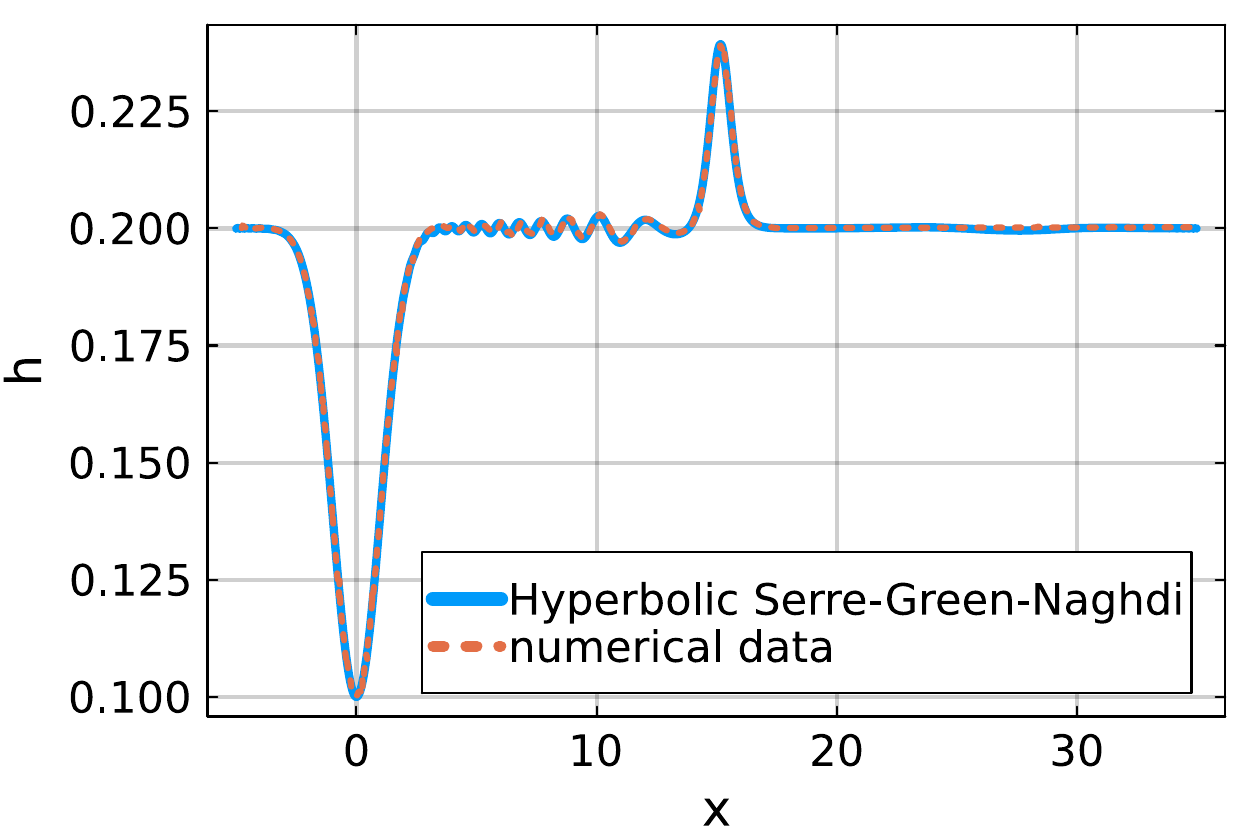}
  \caption{Solitary wave propagating over a Gaussian obstacle. Cross-section of
    water height $h$ along $y = 0$ at time $t = 12$ s. Our numerical
    solution (solid blue line) is compared with data from
    \cite{Busto2021} (orange dashed line). The computational domain
    is $[-5, 35] \times [-10, 10]$ with $\Delta x = \Delta y = 0.025$.
    The dip near $x = 0$ corresponds to the location of the Gaussian
    bathymetry feature. Excellent agreement is observed for both the
    leading wave and the dispersive tail.}
  \label{fig:busto_gaussian}
\end{figure}

\subsection{Propagation of Periodic Waves over a Semi-Circular Shoal}
\label{sec:semi_shoal}

Next, we reproduce the 1971 laboratory experiments of Whalin~\cite{Whalin1971}.
The experiments study the refraction and diffraction
of periodic waves propagating over a semi-circular shoal in a wave tank of
\SI{25.6}{m} length and \SI{6.096}{m} width $W$. The still water depth decreases
from $h_0 = \SI{0.4572}{m}$ at the wave maker to \SI{0.1524}{m} at the end of
the tank. This is a standard benchmark for two-dimensional dispersive wave
models~\cite{RICCHIUTO2014306, guermond2022well, MADSEN1992183,BEJI1996691,WALKLEY2002865,SORENSEN2004181,ESKILSSON2006947,TONELLI2009609,KAZOLEA201242}.

The bottom topography is defined by
\begin{equation}
  b(x,y) = \begin{cases}
    0                                         & \text{if } x < 10.67 - G(y), \\[0.3em]
    \dfrac{x - (10.67 - G(y))}{25}           & \text{if } 10.67 - G(y) \leq x < 18.29 - G(y), \\[0.3em]
    0.30480                                   & \text{if } x \geq 18.29 - G(y),
  \end{cases}
  \label{eq:semi_shoal_bathy}
\end{equation}
where $G(y) = \sqrt{y\,(6.096 - y)}$; the setup is depicted in Figure~\ref{fig:semi_shoal_T2_tend_60s}.

Three test cases with different wave periods and amplitudes are considered:
\begin{enumerate}
  \item[(a)] $T = \SI{1}{s}$, $A = \SI{0.0195}{m}$,
  \item[(b)] $T = \SI{2}{s}$, $A = \SI{0.0075}{m}$,
  \item[(c)] $T = \SI{3}{s}$, $A = \SI{0.0068}{m}$.
\end{enumerate}

The computational domain is $[-10, 36] \times [0, 6.096]$, which is extended beyond the
physical tank to accommodate wave generation and absorption zones. The grid
spacing is uniform with $\Delta x \approx \Delta y \approx 0.023$. The initial
condition is the still-water rest state with zero velocity everywhere.
Reflecting boundary conditions are imposed so that in the $y$-direction they mimic the
lateral walls of the wave tank.

Periodic waves are generated at $x = \SI{-2}{m}$ using source terms following
the approach described in~\cite{RICCHIUTO2014306}. Specifically, a term
$-\partial_t h_{\text{WG}}$ is added to the right-hand side of the equation for $h$, where
\begin{equation}
  h_{\text{WG}}(x,t) = f_{\text{WG}}(x)\, A \sin\!\left(\frac{2\pi}{T}\, t\right).
\end{equation}
Here, $f_{\text{WG}}$ is a Gaussian centered at the generation position that
depends, among other things, on the wave period $T$ and on tuning parameters
that control the spatial extent and amplitude of the source; see Section~5.3
in~\cite{RICCHIUTO2014306} for more details. Using a flat-bathymetry test case, we adjusted the tuning parameters for each period until we achieved the desired wave amplitudes.

Since the source
generates waves traveling in both directions, relaxation zones of \SI{5}{m}
length are placed at both ends of the domain in the $x$-direction to absorb
outgoing waves, so that the choice of boundary conditions in the $x$-direction
is immaterial. Within these relaxation zones, source terms of the form
$-\sigma(\boldsymbol{q} - \boldsymbol{q}_0)$ are added to the
right-hand side of each equation, where $\boldsymbol{q}_0$ denotes
the initial rest state. The damping coefficient $\sigma$ increases
quadratically from zero at the inner edge of the relaxation zone to a maximum
value of 5 at the domain boundary, ensuring a smooth transition.

\begin{figure}[htb]
  \centering
  \includegraphics[width=0.8\textwidth]{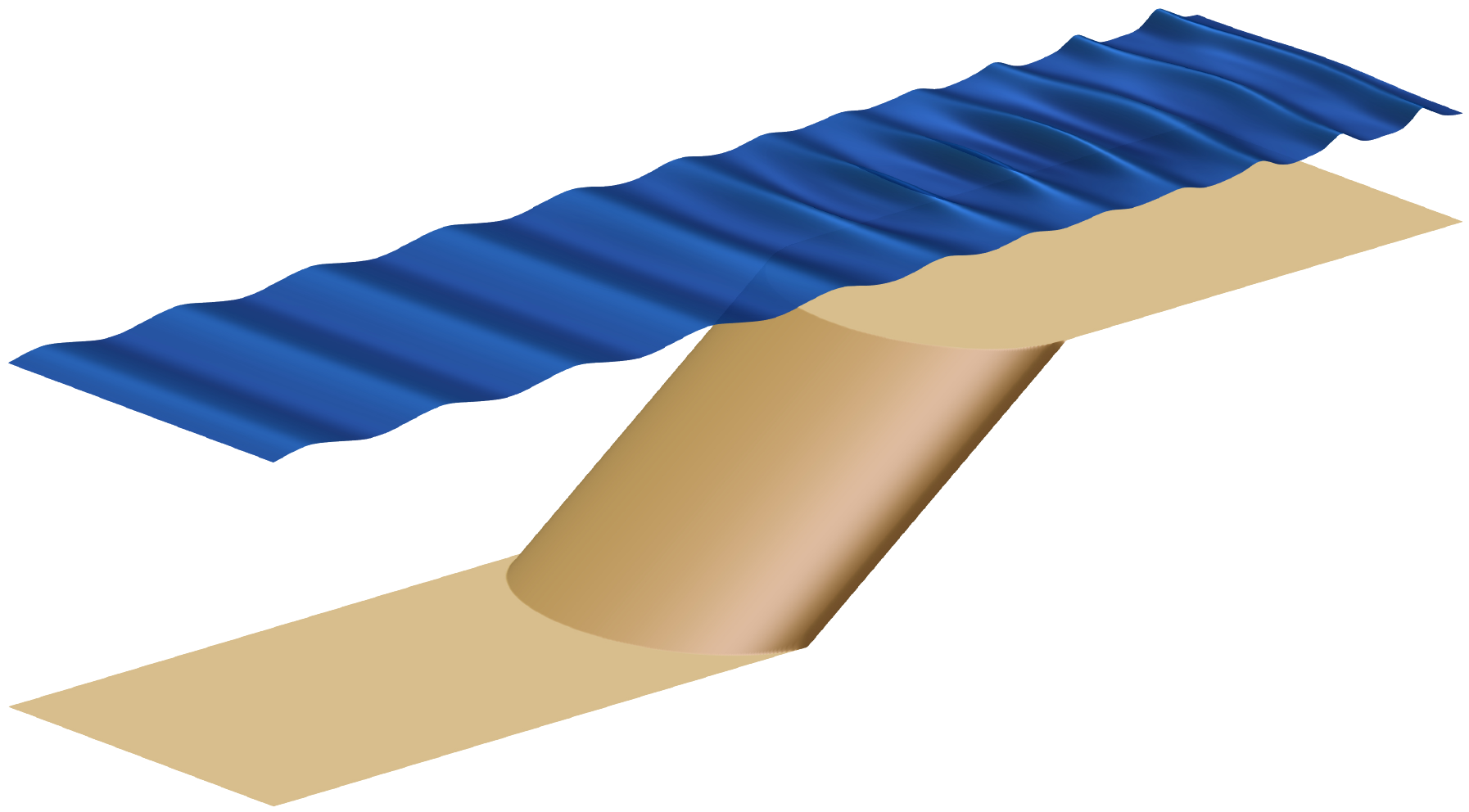}
  \caption{Three-dimensional view of the water surface elevation (blue) and bathymetry (brown) for the semi-circular shoal test case~(b) ($T = \SI{2}{s}$, $A = \SI{0.0075}{m}$) at $t = \SI{60}{s}$, showing the refraction of the initially planar wave fronts as they propagate over the shoal. Only the region $x \in [-5, 30]$ is shown. Created using Makie.jl~\cite{DanischKrumbiegel2021}.}
  \label{fig:semi_shoal_T2_tend_60s}
\end{figure}

For all cases, the simulation is first run for many wave periods to allow
the solution to reach a steady periodic state. After this transient phase, the
free-surface elevation along the tank centerline $y = W/2$ is sampled over
15 periods. A discrete Fourier transform of the time series at each
spatial point along the centerline is then performed to extract the amplitudes
of the first, second, and third harmonics. This harmonic analysis is performed to
allow comparison with the experimental data of Whalin~\cite{Whalin1971}; the results are shown in Figures~\ref{fig:semi_shoal_case_1}--\ref{fig:semi_shoal_case_3}.

For case~(a), the experimental data~\cite{Whalin1971} are only available for the first two harmonics. The numerical solution captures the general shape of the first harmonic well, though its amplitude lags slightly behind the experimental values over the shoal.

\begin{figure}[htb]
  \centering
  \includegraphics[width=0.6\textwidth]{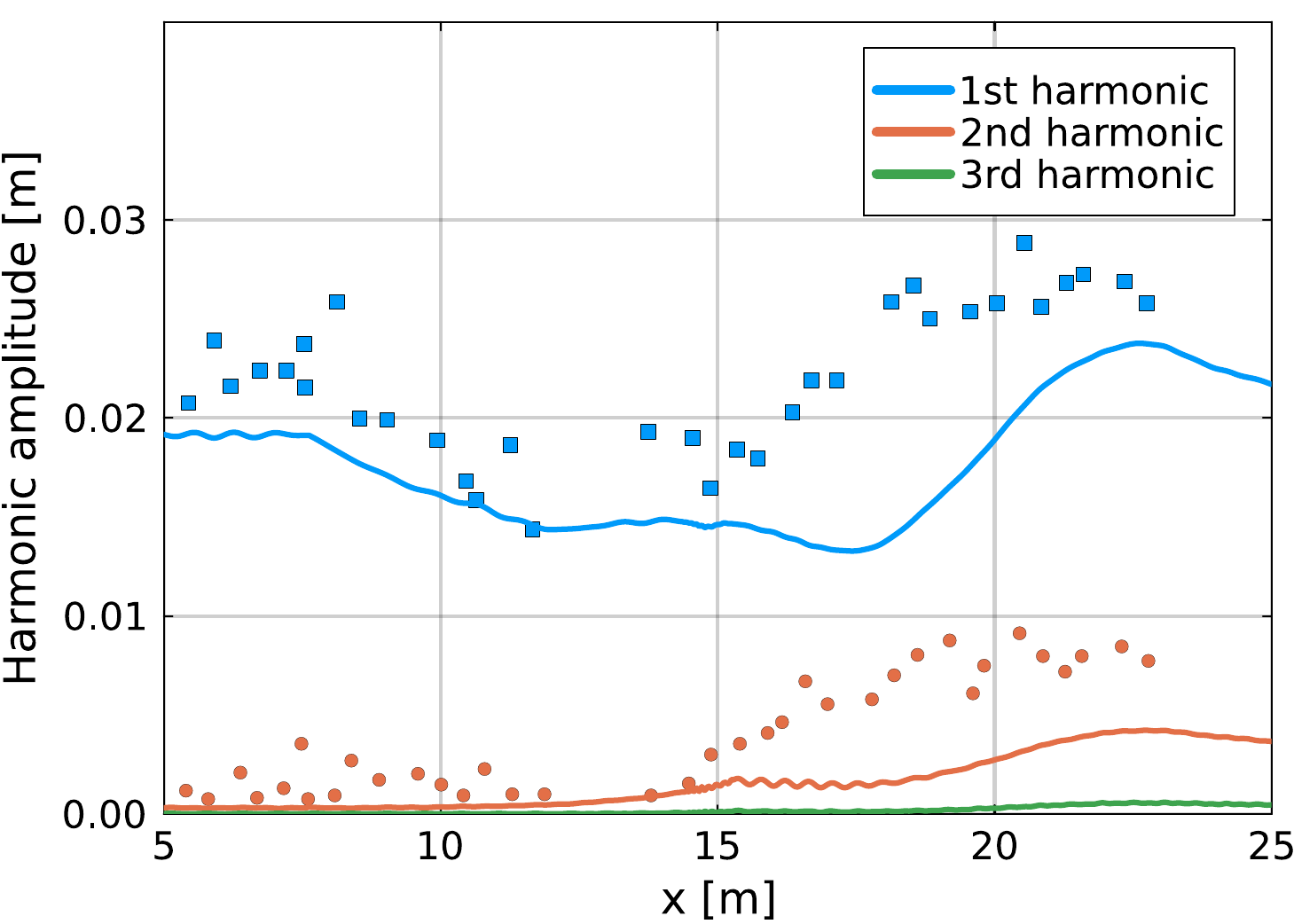}
  \caption{Propagation of periodic waves over a semi-circular shoal, case~(a):
    $T = \SI{1}{s}$, $A = \SI{0.0195}{m}$. Comparison of the first two
    harmonic amplitudes along the tank centerline between the numerical
    solution (solid lines) and experimental data
    from~\cite{Whalin1971} (symbols).}
  \label{fig:semi_shoal_case_1}
\end{figure}

In the intermediate period case~(b), the numerical results are in very good agreement with the experimental data for all three harmonics, in both amplitude and shape.

\begin{figure}[htb]
  \centering
  \includegraphics[width=0.6\textwidth]{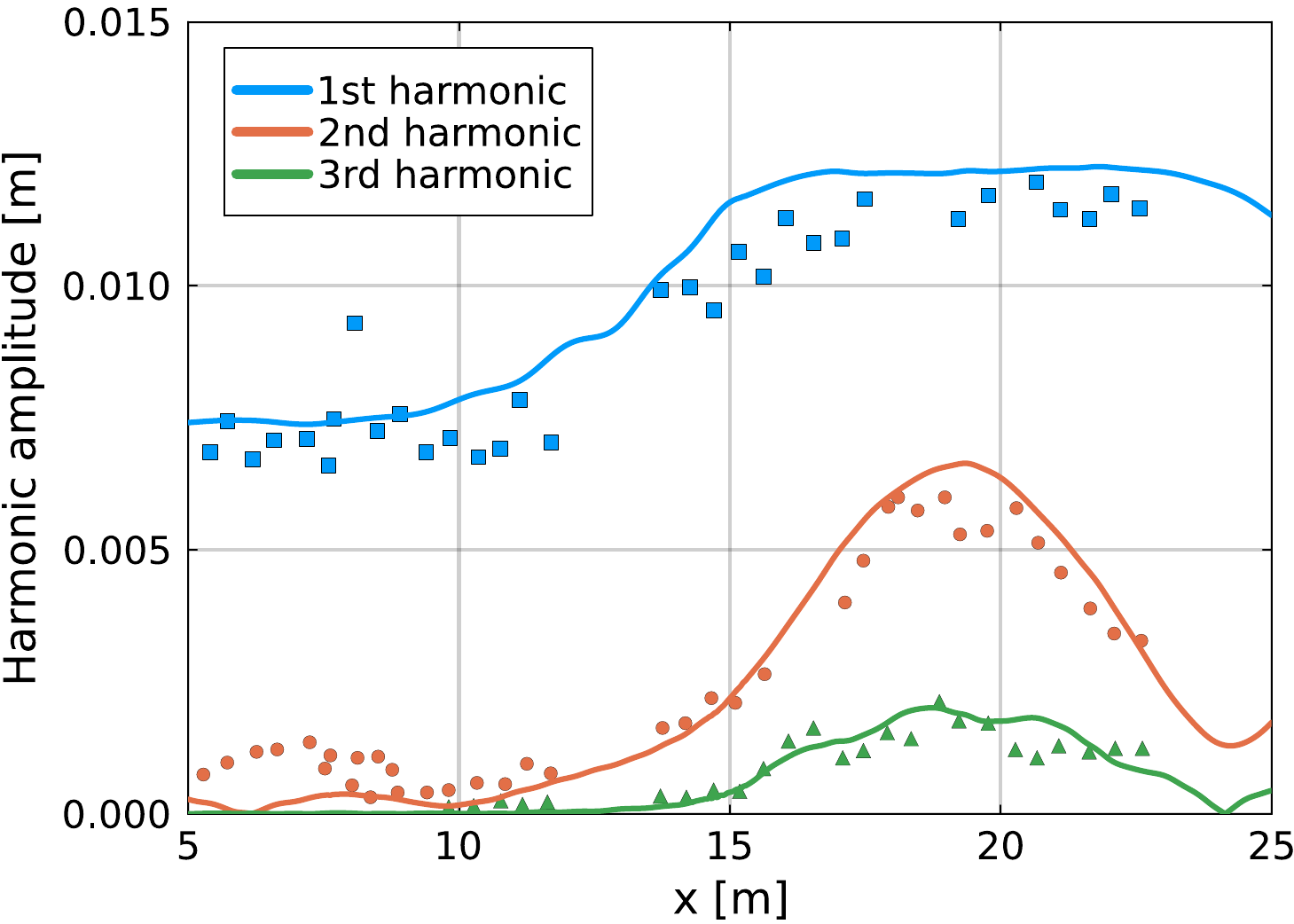}
  \caption{Propagation of periodic waves over a semi-circular shoal, case~(b):
    $T = \SI{2}{s}$, $A = \SI{0.0075}{m}$. Comparison of the first three
    harmonic amplitudes along the tank centerline between the numerical
    solution (solid lines) and experimental data
    from~\cite{Whalin1971} (symbols).}
  \label{fig:semi_shoal_case_2}
\end{figure}

The numerical solution for the longest period case~(c) captures the overall shape of all three harmonics. The first harmonic amplitude is slightly overestimated over the shoal, while the second and third harmonics are somewhat underestimated. This behavior is, however, consistent with what has been reported by other authors for this test case using different dispersive wave models~\cite{RICCHIUTO2014306,MADSEN1992183,BEJI1996691,TONELLI2009609,KAZOLEA201242}.

\begin{figure}[htb]
  \centering
  \includegraphics[width=0.6\textwidth]{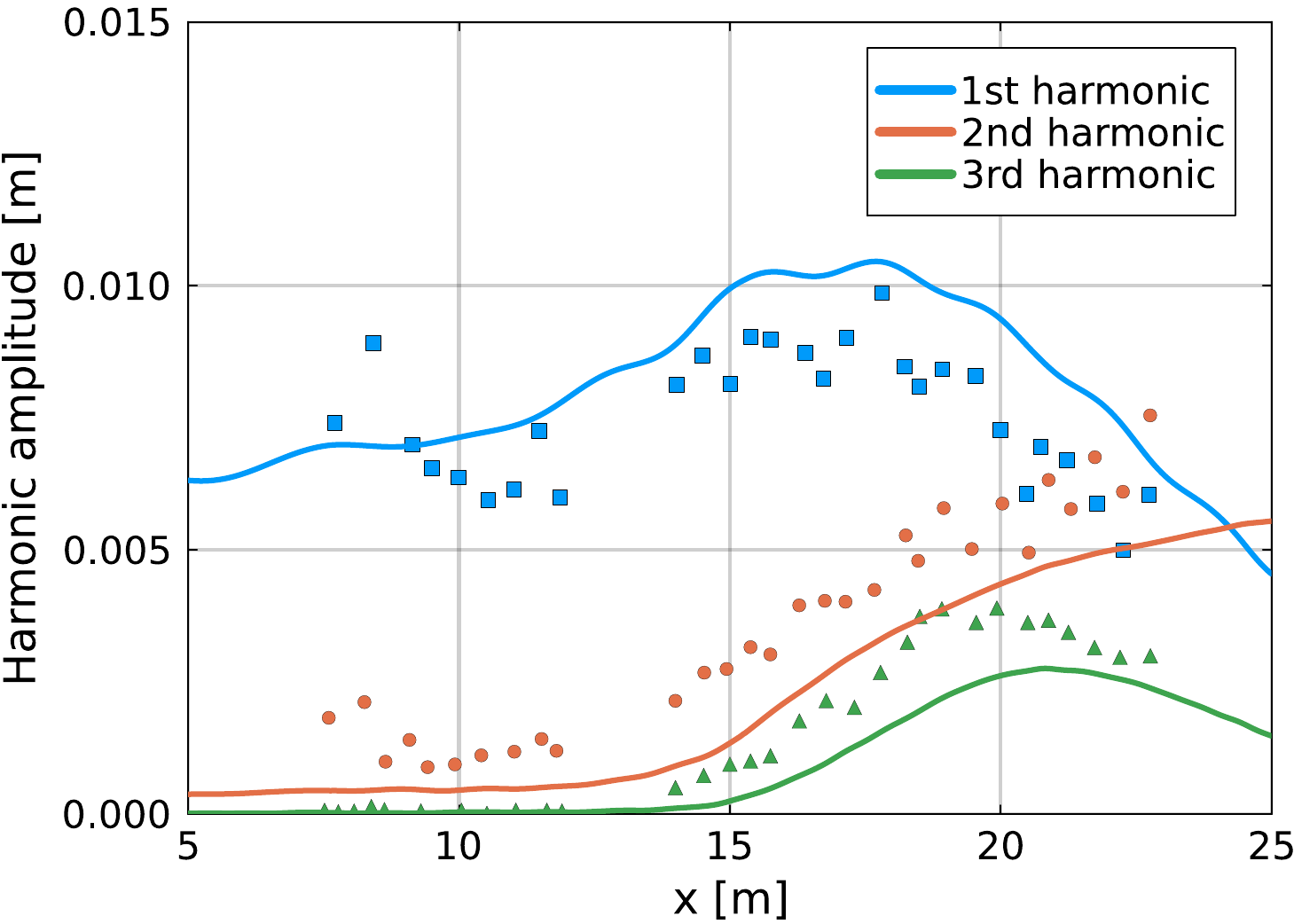}
  \caption{Propagation of periodic waves over a semi-circular shoal, case~(c):
    $T = \SI{3}{s}$, $A = \SI{0.0068}{m}$. Comparison of the first three
    harmonic amplitudes along the tank centerline between the numerical
    solution (solid lines) and experimental data
    from~\cite{Whalin1971} (symbols).}
  \label{fig:semi_shoal_case_3}
\end{figure}

Overall, these results are in general agreement with both previous literature and the experimental data, showing
the ability of the proposed method to reproduce the nonlinear refraction and diffraction effects observed in the experiments.
 
\subsection{Dam Break Problems}
\label{sec:dam_break}

In this section, we consider two-dimensional dam break problems on a flat bottom.
We compare our results against the numerical data of Tkachenko et al.~\cite{tkachenko2023extended}, who solve the same
system in a finite volume framework.

Two configurations are studied: a \emph{cylindrical} dam break and a \emph{square}
dam break. To avoid high-frequency oscillations in our solutions that
would arise from a discontinuous initial condition and to enable a meaningful
comparison with~\cite{tkachenko2023extended}, we replace the discontinuous initial condition of the water height by a smoothed version.

For the cylindrical dam break we use
\begin{equation}
  \begin{aligned}
    h(x,y,0) &= 1 + 0.4 \left(1 - \tanh\left(\frac{r - 20}{\alpha}\right)\right),
    \quad r = \sqrt{x^2 + y^2}, \\
    u(x,y,0) &= v(x,y,0) = 0,
  \end{aligned}
\end{equation}
and for the square dam break
\begin{equation}
  \begin{aligned}
    h(x, y, 0) &= 1 + 0.8\, f(x)\, f(y), \quad f(x) = \tfrac{1}{2}\left(1 - \tanh\left(\frac{|x| - 40}{\alpha}\right)\right) \\
    u(x, y, 0) &= v(x, y, 0) = 0,
  \end{aligned}
\end{equation}
with smoothing parameter $\alpha = 6.5$. In both cases, the initial water height is approximately $1.8$ in the central region and $1.0$ in the outer region, with a smooth transition between these values.
The computational domain is $[-300, 300] \times [-300, 300]$ with grid spacing
$\Delta x = \Delta y = 0.22$. Reflecting boundary conditions are used, though the choice does not matter as the wave does not reach the outer boundaries of the domain during the
simulation, which is performed until $t = 40$.
The smoothing parameter $\alpha$ governs
a trade-off between two regions. A larger $\alpha$ improves the amplitude
agreement for $x \in [130, 160]$ at the cost of accuracy for
$x \in [60, 100]$, while a smaller $\alpha$ has the
opposite effect. The value $\alpha = 6.5$ was chosen as a compromise that
yields reasonable amplitude agreement in both regimes.

The smoothed initial conditions and the water
surface elevations at $t = 40$ are shown in a heatmap in
Figures~\ref{fig:dam_break_cylinder_heat}
and~\ref{fig:dam_break_square_heat} for the cylindrical and square cases,
respectively. For comparison with the reference data, we extract the
cross-section along $y = 0$.

\begin{figure}[htb]
  \centering
  \includegraphics[width=0.8\textwidth]{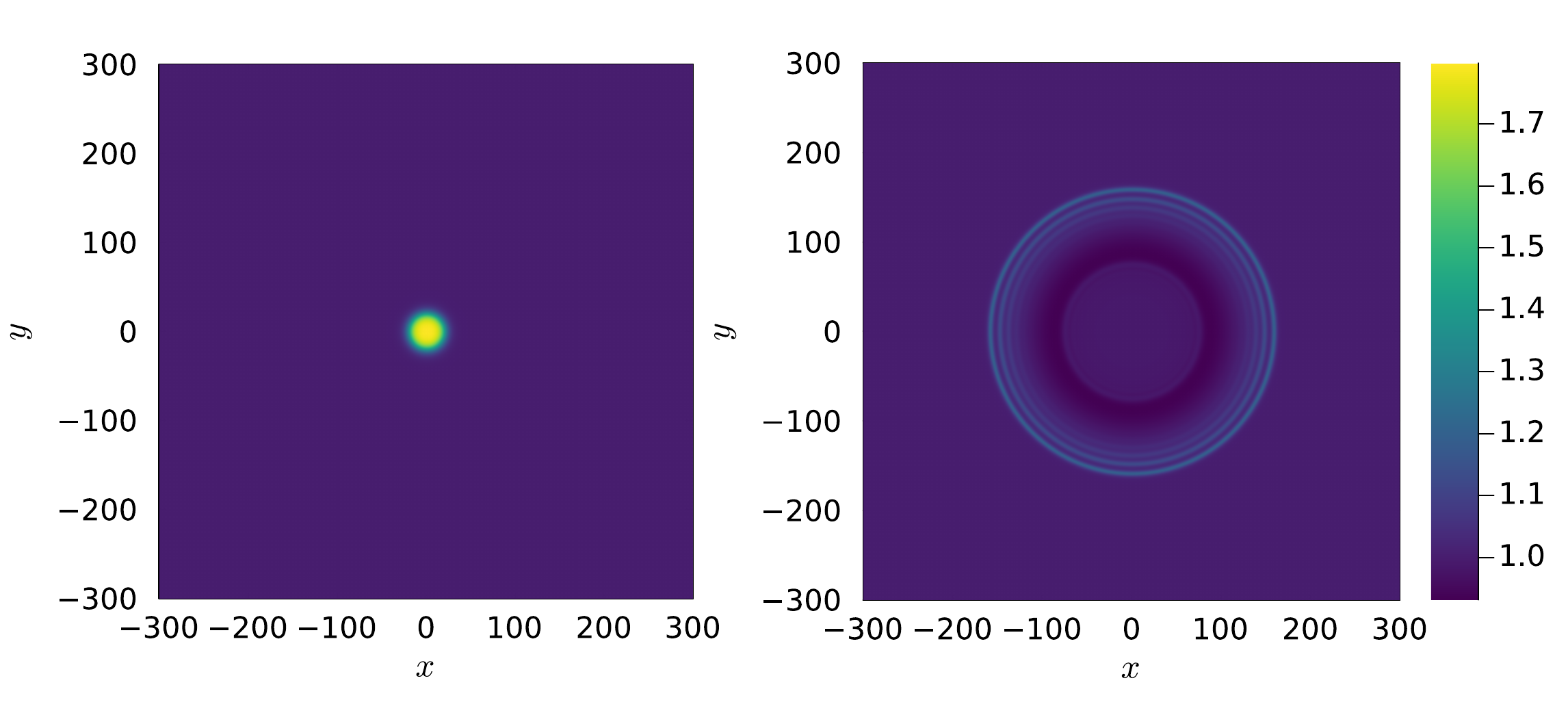}
  \caption{Cylindrical dam break on the domain $[-300, 300] \times [-300, 300]$ with grid
    spacing $\Delta x = \Delta y = 0.22$. Left: initial condition.
    Right: $h$ at $t = 40$, showing the outward-propagating wave.}
  \label{fig:dam_break_cylinder_heat}
\end{figure}

\begin{figure}[htb]
  \centering
  \includegraphics[width=0.6\textwidth]{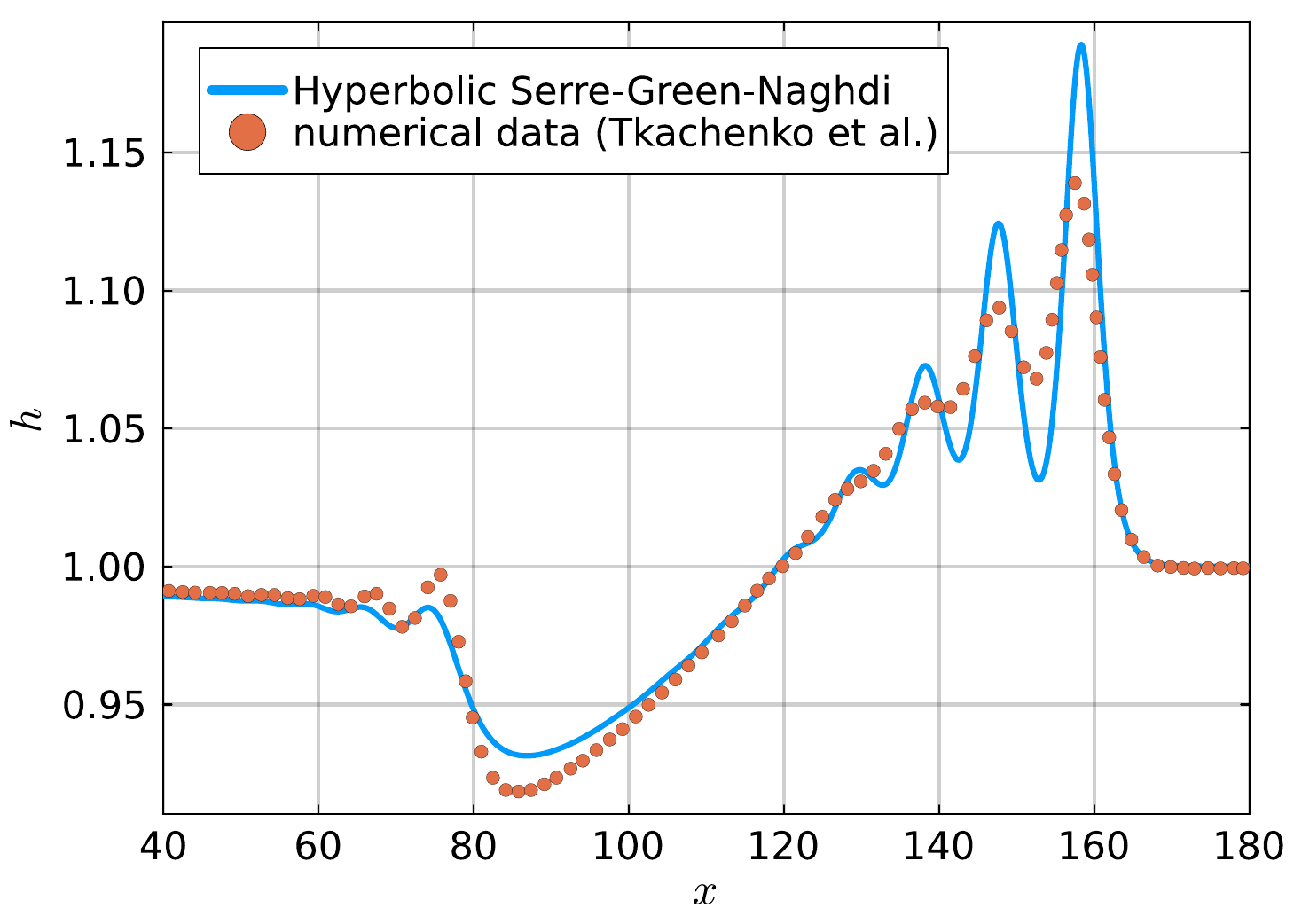}
  \caption{Cross-section of the water surface elevation $h$ along $y = 0$
    at $t = 40$ for the cylindrical dam break. Our numerical solution (solid
    line) is compared with the reference data
    from~\cite{tkachenko2023extended} (dots).}
  \label{fig:dam_break_cylindric}
\end{figure}

Figure~\ref{fig:dam_break_cylindric} shows this cross-section at $t = 40$
for the cylindrical dam break. Our scheme captures the global structure of
the dispersive wave train, both the phase of the leading wave and the overall
elevation of the trailing oscillations, in good agreement
with~\cite{tkachenko2023extended}.

The same observations hold for the square dam break shown in
Figure~\ref{fig:dam_break_square}. The global wave structure and the phase
of both the leading and trailing waves are well reproduced, while there are
amplitude deviations from the reference data of~\cite{tkachenko2023extended}, due to the smoothing of the initial condition.

\begin{figure}[htb]
  \centering
  \includegraphics[width=0.8\textwidth]{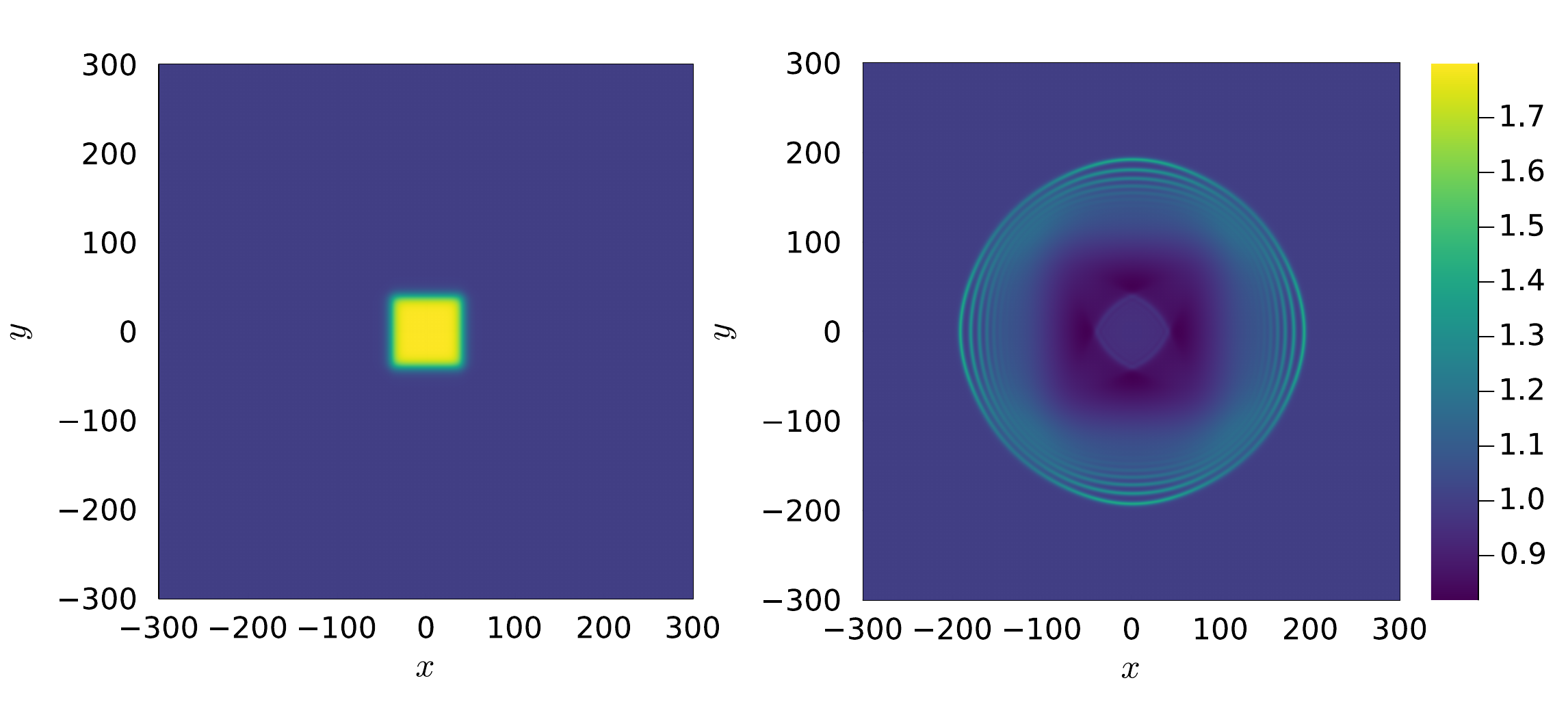}
  \caption{Square dam break on the domain $[-300, 300] \times [-300, 300]$ with grid
    spacing $\Delta x = \Delta y = 0.22$. Left: initial condition.
    Right: $h$ at $t = 40$, showing the outward-propagating wave.}
  \label{fig:dam_break_square_heat}
\end{figure}

\begin{figure}[htb]
  \centering
  \includegraphics[width=0.6\textwidth]{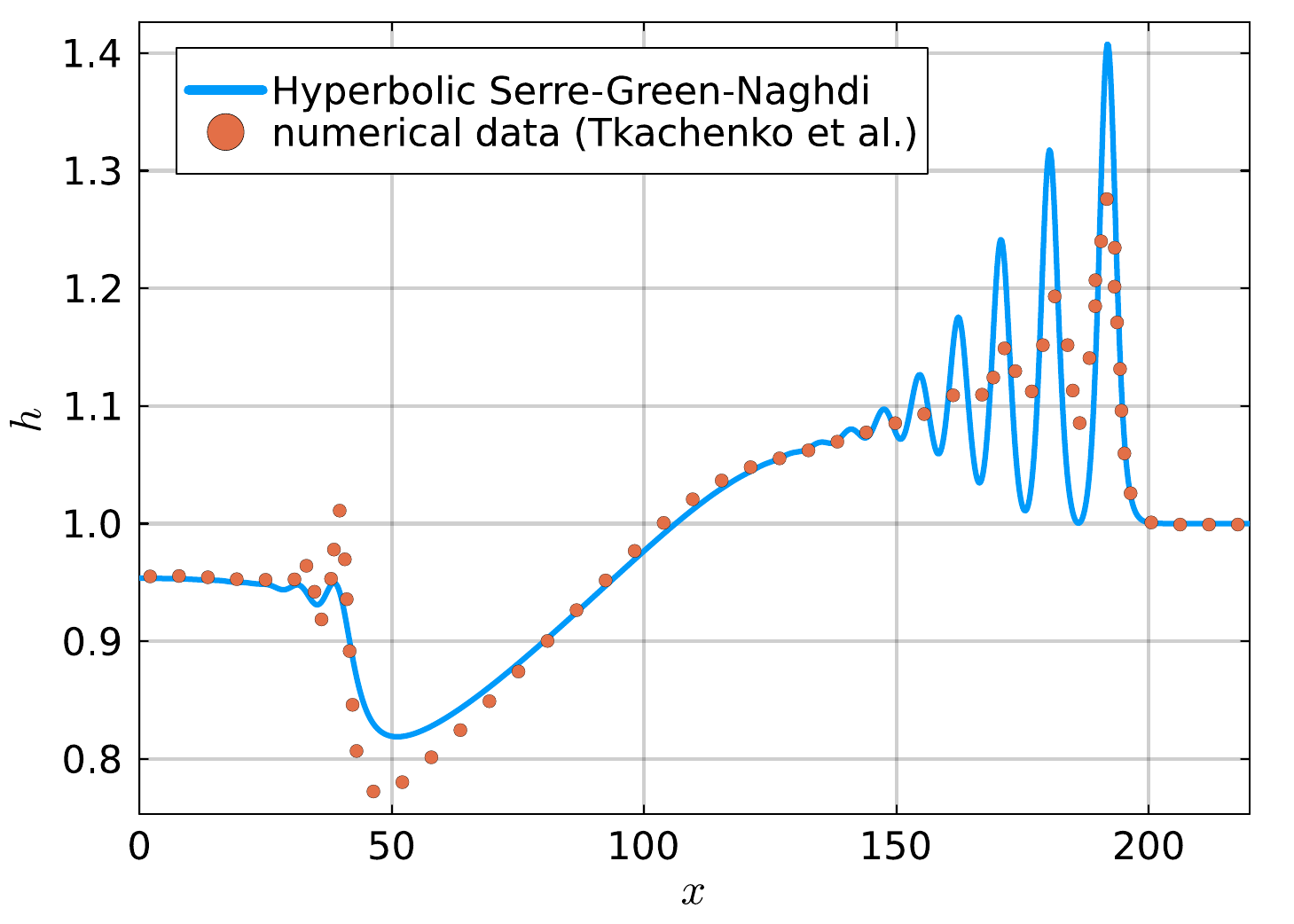}
  \caption{Cross-section of the water surface elevation $h$ along $y = 0$
    at $t = 40$ for the square dam break. Our numerical solution (solid line)
    is compared with the reference data
    from~\cite{tkachenko2023extended} (dots).}
  \label{fig:dam_break_square}
\end{figure}

\subsection{Long-Term Cylindrical Dam Break}
\label{sec:long_term_dam_break}

In this section, we consider the same cylindrical dam break setup as in
Section~\ref{sec:dam_break}, but extend the simulation to $t = 300$ to
investigate the long-term behavior of the dispersive shock wave. To prevent
the outward-propagating waves from reaching the domain boundaries, we enlarge
the computational domain to $[-1100, 1100] \times [-1100, 1100]$. Again, reflecting
boundary conditions are used, though the choice is still immaterial. Using the same grid
spacing of $\Delta x = \Delta y = 0.22$ yields a grid of
$\num{10000} \times \num{10000} = \num{100000000}$ points. A simulation of this scale and
duration takes advantage of both the computational throughput of our GPU-accelerated
implementation and the energy conservation properties of our
semi-discretization: without the former, the computation would be infeasible
on a single workstation, and without the latter, numerical dissipation could
progressively destroy the fine dispersive structures over long
integration times. This extends the one-dimensional soliton fission
experiments of~\cite[Section~11.7]{ranocha2025structure} to the fully two-dimensional
setting with cylindrical geometry.

The radial profile of the water height along $y = 0$ is shown at nine
output times between $t = 30$ and $t = 300$ in
Figure~\ref{fig:dam_break_long_over_time}. The initial undular bore
progressively resolves into a train of solitary waves ordered by decreasing
amplitude, with the tallest and fastest pulse at the leading edge. This
process is known as soliton fission~\cite{ranocha2025structure}. The peak
amplitudes decrease over time due to cylindrical spreading as the wave
energy distributes over an ever-growing circumference.

\begin{figure}[htb]
  \centering
  \includegraphics[width=0.6\textwidth]{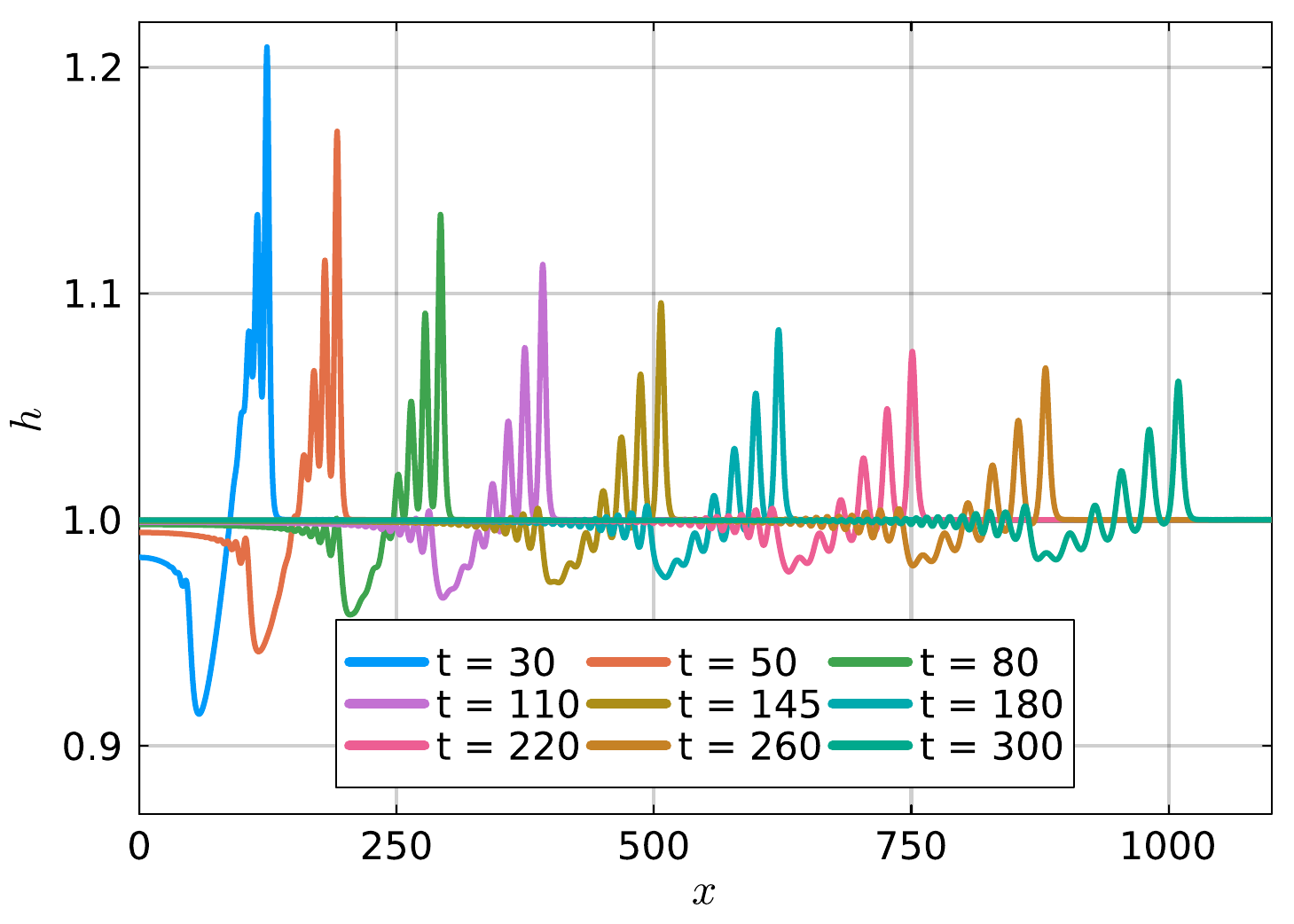}
  \caption{Cross-section of the water surface elevation $h$ along $y = 0$
    at times $t = 30, 50, 80, 110, 145, 180, 220, 260, 300$. The
    dispersive shock wave progressively resolves into a train of solitary
    waves whose amplitudes decay due to cylindrical spreading. The
    computational domain is $[-1100, 1100] \times [-1100, 1100]$ with grid
    spacing $\Delta x = \Delta y = 0.22$.}
  \label{fig:dam_break_long_over_time}
\end{figure}

To verify that the Cartesian grid does not introduce significant asymmetry
over such long integration times,
Figure~\ref{fig:dam_break_long_cross_sections} compares cross-sections at
$t = 300$ taken along $y = 0$ (the \SI{0}{\degree} direction) and along the
diagonal (the \SI{45}{\degree} direction), both plotted as a function of the
radial coordinate $r = \sqrt{x^2 + y^2}$. The two profiles are nearly
indistinguishable, confirming that the solution preserves its radial symmetry
well despite evolving on a square grid.
Note that this agreement deteriorates on coarser grids, where the reduced
resolution amplifies grid-induced directional errors.

\begin{figure}[htb]
  \centering
  \includegraphics[width=0.6\textwidth]{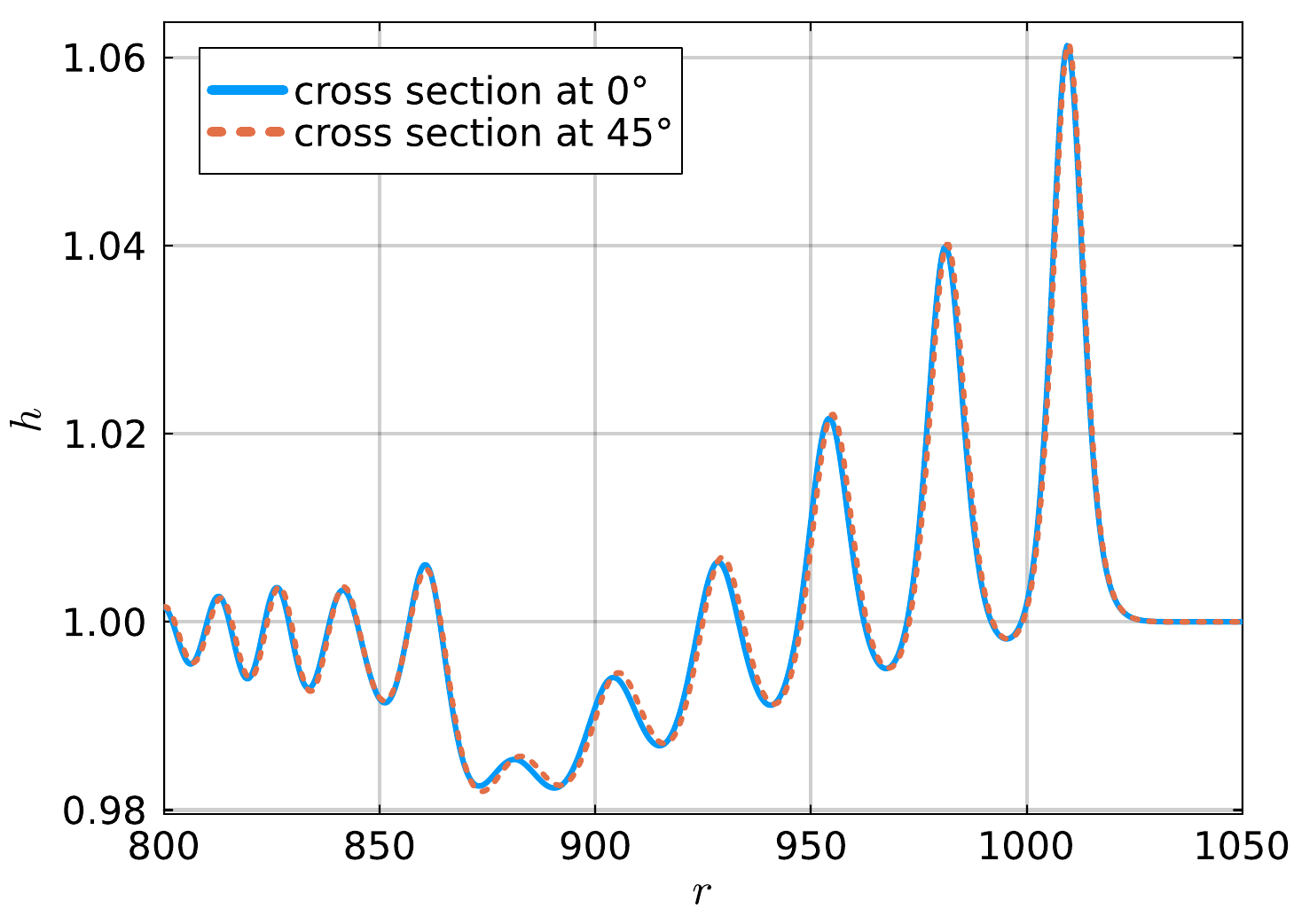}
  \caption{Comparison of cross-sections at $t = 300$ along $y = 0$
    (\SI{0}{\degree}) and along the diagonal (\SI{45}{\degree}), plotted as
    a function of the radial coordinate $r$. The two
    profiles are nearly indistinguishable, confirming that the radial
    symmetry is well preserved by the Cartesian grid.}
  \label{fig:dam_break_long_cross_sections}
\end{figure}

The amplitude of the leading peak decays as the wave propagates outward,
see Figure~\ref{fig:dam_break_long_amplitude_decay}. The overall decay is
consistent with cylindrical spreading, where the wave energy distributes
over a circumference proportional to $r$, suggesting a $1/\sqrt{r}$ scaling.
However, the data points systematically lie below the fitted $1/\sqrt{r}$
curve (anchored at the start time), indicating that the leading peak decays
slightly faster than pure geometric spreading would predict.
This may simply be a consequence of ongoing interactions between different harmonic
components: the leading pulse is still transferring energy to the secondary solitary waves and the
dispersive radiation tail behind it, before the soliton fission has fully taken place.
The existence of different scaling laws before the soliton fission, with faster decay before the fission has
finally taken place, has been investigated for simpler models, such as the
cylindrical KdV equations; see, e.g.,~\cite{Weidman_Zakhem_1988}.

\begin{figure}[htb]
  \centering
  \includegraphics[width=0.6\textwidth]{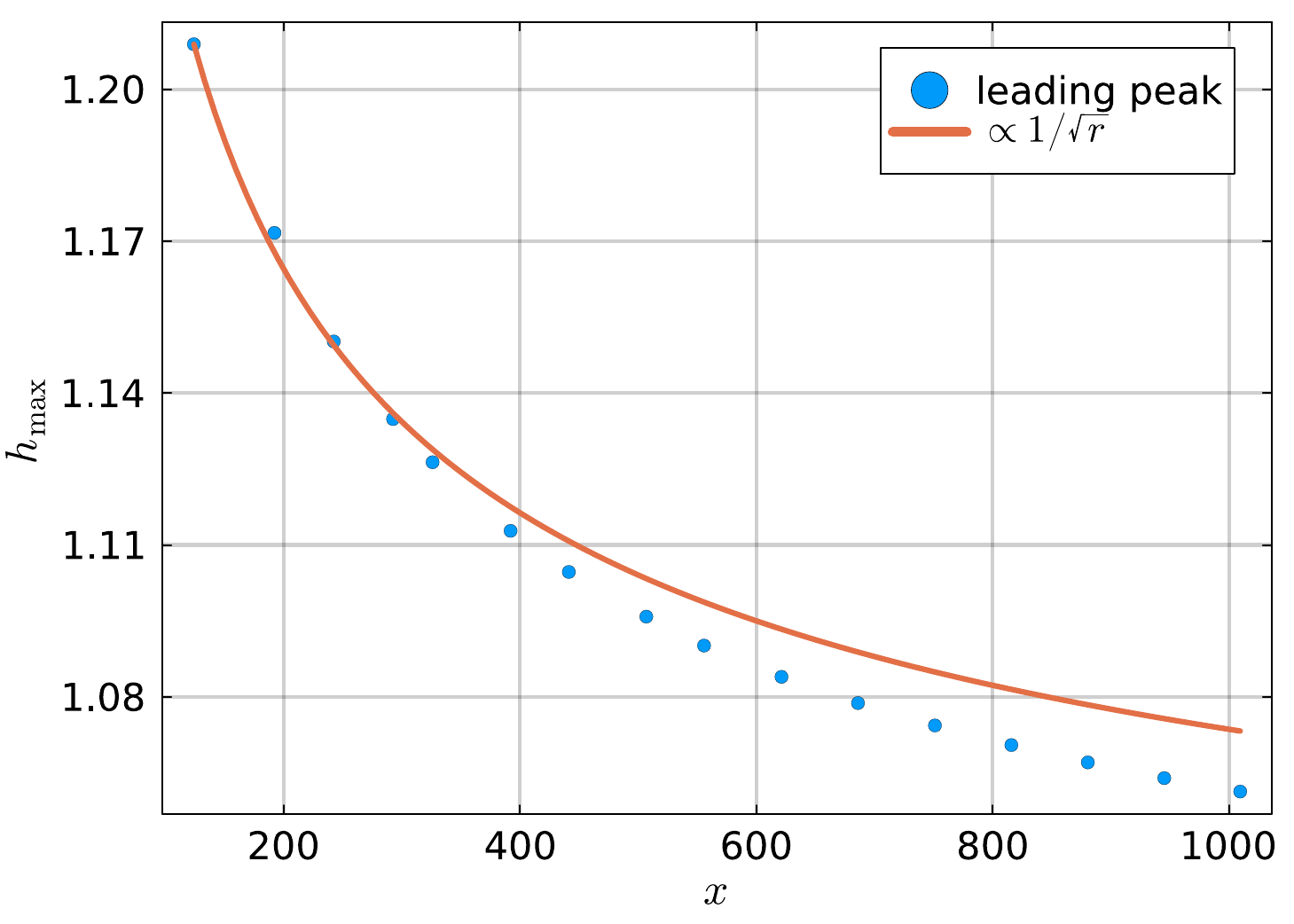}
  \caption{Amplitude of the leading peak $h_{\max}$ as a function of its
  radial position $r$, with data points from $t = 30$ to $t = 300$. The
  solid line shows the $1/\sqrt{r}$ decay law anchored at the starting time.
  The data points lie below this curve, indicating that the leading pulse
  decays faster than geometric spreading alone due to ongoing energy
  redistribution into the trailing wave train.}
  \label{fig:dam_break_long_amplitude_decay}
\end{figure}

Finally, we fit the analytical solitary wave profile from
equation~\eqref{eq:solitary_wave_analytical} to the leading pulse at four
times using LsqFit.jl~\cite{Myles_White_LsqFit_jl}, as shown in Figure~\ref{fig:dam_break_long_soliton_fit}. At $t = 30$,
the leading pulse has not yet separated cleanly from the undular bore, and
the fit residual is correspondingly large, about a tenth of the
peak amplitude. By $t = 80$, the separation has improved significantly, though the
one-sided positive residual indicates that the separation from the trailing waves is still incomplete.
The fit already captures the pulse shape well, with residuals
small relative to the peak amplitude. At later times ($t = 220$ and
$t = 300$), the leading pulse is well separated, with only a slight tail to its left. The residuals are approximately 100 times
smaller than the peak amplitude. Note that due to cylindrical spreading, the
peak amplitude decreases continuously over time, which limits how cleanly
the leading soliton can separate from the trailing wave train: running for
longer times improves the separation but simultaneously reduces the signal
to be fitted. These results are consistent with the one-dimensional soliton
fission experiments reported
in~\cite[Section~11.7]{ranocha2025structure}, where similar fit quality was
observed for the leading solitary waves emerging from a dispersive shock.
Energy-preserving methods are best suited to capture
this behavior on relatively coarse meshes. For such long-time simulations,
dissipative methods would require prohibitive computational effort for the same result.

\begin{figure}[htb]
  \centering
  \includegraphics[width=0.9\textwidth]{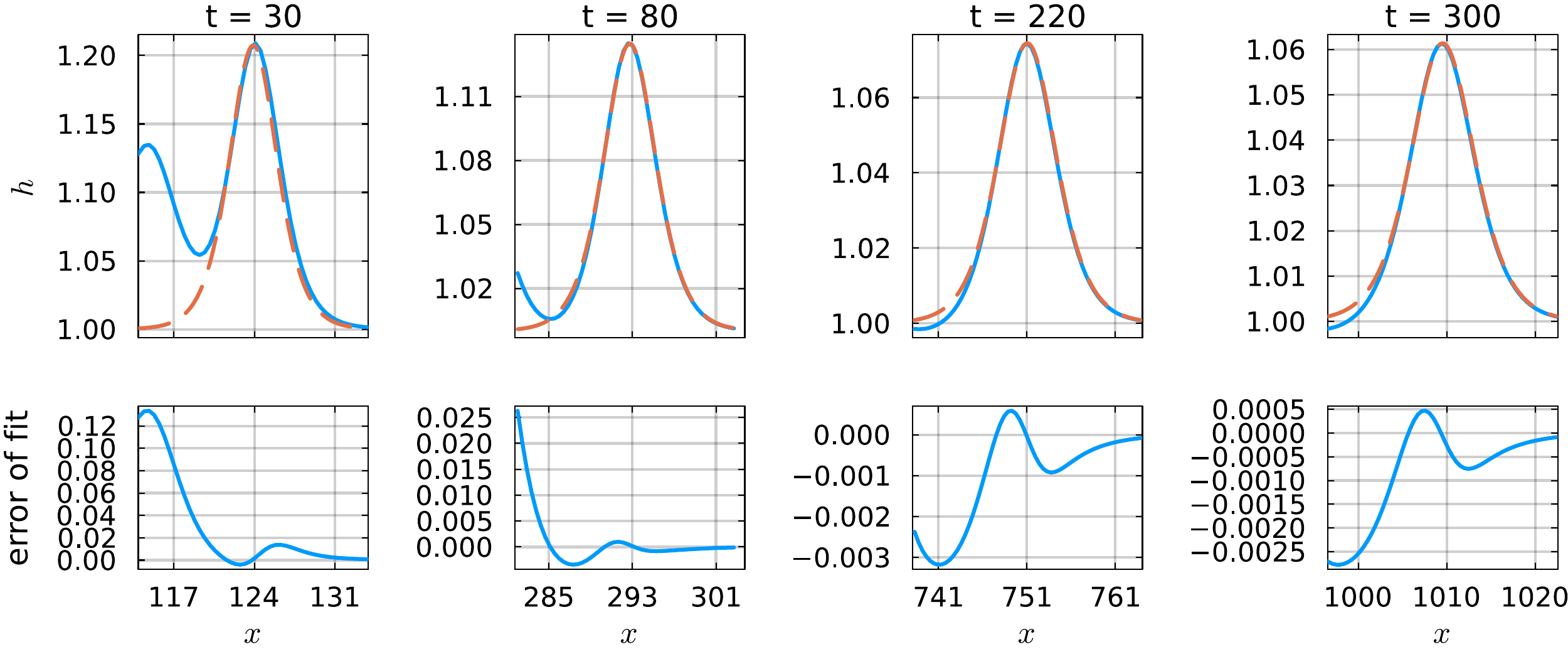}
  \caption{Soliton fits to the leading pulse at times
    $t = 30, 80, 220, 300$. Top: water height $h$ (solid blue line) and
    fitted analytical solitary wave profile (dashed orange line). Bottom:
    fit residual. At early times, the leading pulse has not yet separated
    from the undular bore, resulting in larger fit errors. From $t = 80$
    onward, the residuals are small relative to the peak amplitude,
    indicating convergence toward a soliton profile.}
  \label{fig:dam_break_long_soliton_fit}
\end{figure}

\subsection{Performance Comparison}
\label{sec:performance}

To assess the computational efficiency of our implementation, we present both
a scaling study with varying grid resolution and a summary of execution times
for all numerical experiments presented in this work. The four hardware
platforms tested are:
\begin{itemize}
  \item CPU (x86): Intel Core i7-1185G7 (4 cores)
  \item CPU (ARM): Apple M4 (4 performance cores)
  \item AMD GPU: Instinct MI210
  \item NVIDIA GPU: H200
\end{itemize}
We emphasize that these are single-run benchmarks on each platform, intended
to provide rough estimates of computational efficiency rather than a
comprehensive performance study.

\subsubsection{Scaling Behavior}
\label{sec:scaling}

We benchmark the total runtime of the wave over Gaussian obstacle test case
from Section~\ref{sec:gaussian_obstacle} to final time $t = 12$ with varying
grid resolutions $\mathcal{N}_x = 2^n$ and
$\mathcal{N}_y = \mathcal{N}_x/2$.

Figure~\ref{fig:benchmarking_plot} presents the wall-clock time as a function
of grid resolution. The GPU implementations demonstrate substantial
performance advantages over the CPUs, particularly at higher resolutions when
the overhead becomes negligible. The NVIDIA H200 shows the best performance
across all problem sizes tested, followed by the AMD MI210, with the CPU
implementation being significantly slower for larger grids, as expected.

Extrapolating the Intel CPU time for $\mathcal{N}_x = 8192$ yields an estimated
runtime exceeding one day, whereas the H200 completes the same computation in
less than 15 minutes. These results demonstrate that GPU acceleration is
essential for large problem sizes or long-term simulations.

An interesting observation is that the runtime scales approximately as
$\mathcal{N}_x^{2.7}$ (or equivalently as
$\mathcal{N}_{\text{total}}^{1.35}$ where $\mathcal{N}_{\text{total}}$ is
the total number of grid points). This super-linear scaling is primarily
attributable to the adaptive time stepping: smaller spatial grid spacing
$\Delta x$ imposes more restrictive
Courant--Friedrichs--Lewy (CFL) conditions, requiring additional time steps
to reach the final time.

\begin{figure}[htb]
  \centering
  \includegraphics[width=0.6\textwidth]{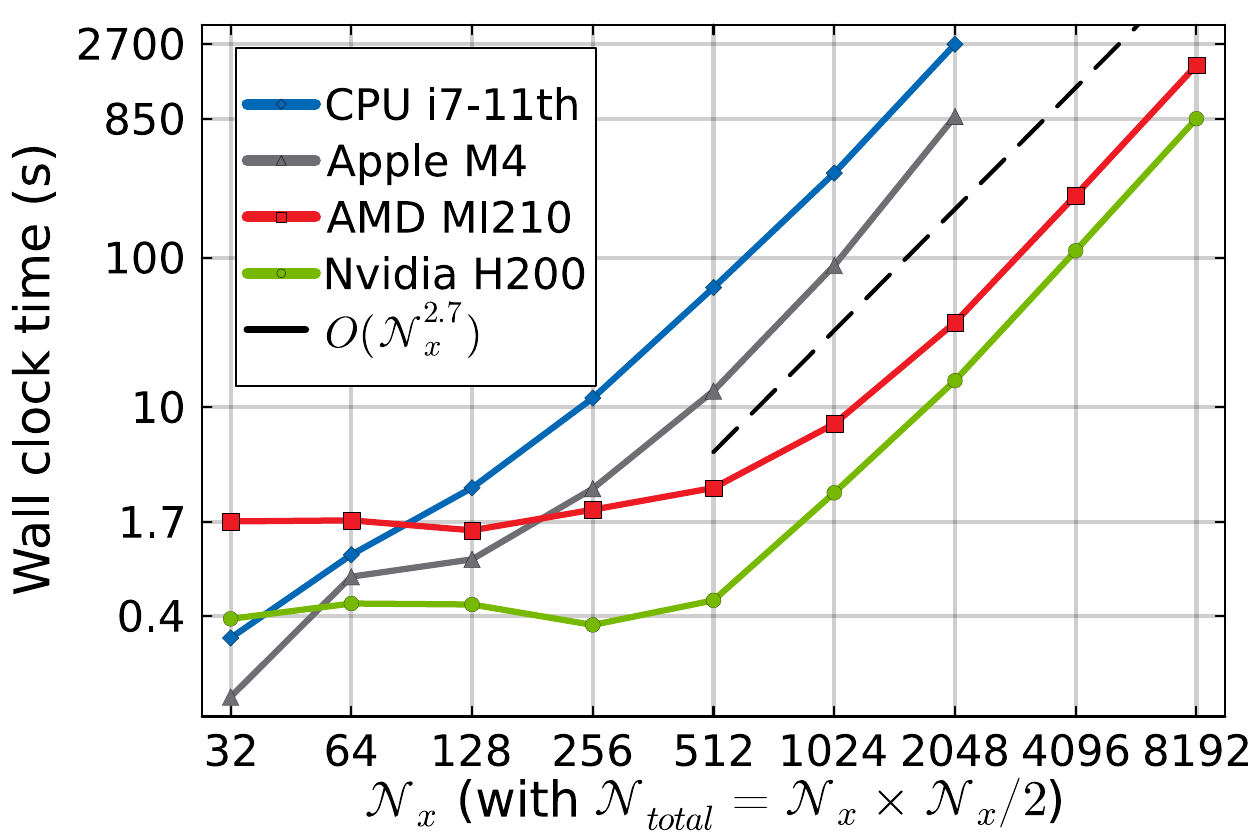}
  \caption{Performance comparison across hardware platforms for the wave over
    Gaussian obstacle test case. Runtime (wall-clock time) versus grid
    resolution $\mathcal{N}_x$ for CPUs (Intel Core i7-1185G7 and Apple M4), AMD GPU
    (Instinct MI210), and NVIDIA GPU (H200). The problem size scales as
    $\mathcal{N}_x^2/2$. Note the log-log scale. These results are from
    single runs on each platform, so no error bars are shown.}
  \label{fig:benchmarking_plot}
\end{figure}

To account for the influence of adaptive time stepping on the observed
scaling, we normalize the runtime by the number of right-hand-side evaluations
performed during the simulation.
Figure~\ref{fig:benchmarking_plot_normalized} shows the average time per
right-hand-side evaluation as a function of the problem size
$\mathcal{N}_{\text{total}} = \mathcal{N}_x \times \mathcal{N}_y$. As
expected, the scaling becomes approximately linear. Note that the time per
right-hand side (RHS) evaluation was calculated by dividing the total
wall-clock time by the number of RHS evaluations reported by the ODE solver,
which ignores any overhead such as initialization or saving of the solution.
This approach is justified since for large problem sizes, this overhead
becomes negligible compared to the total runtime.

\begin{figure}[htb]
  \centering
  \includegraphics[width=0.6\textwidth]{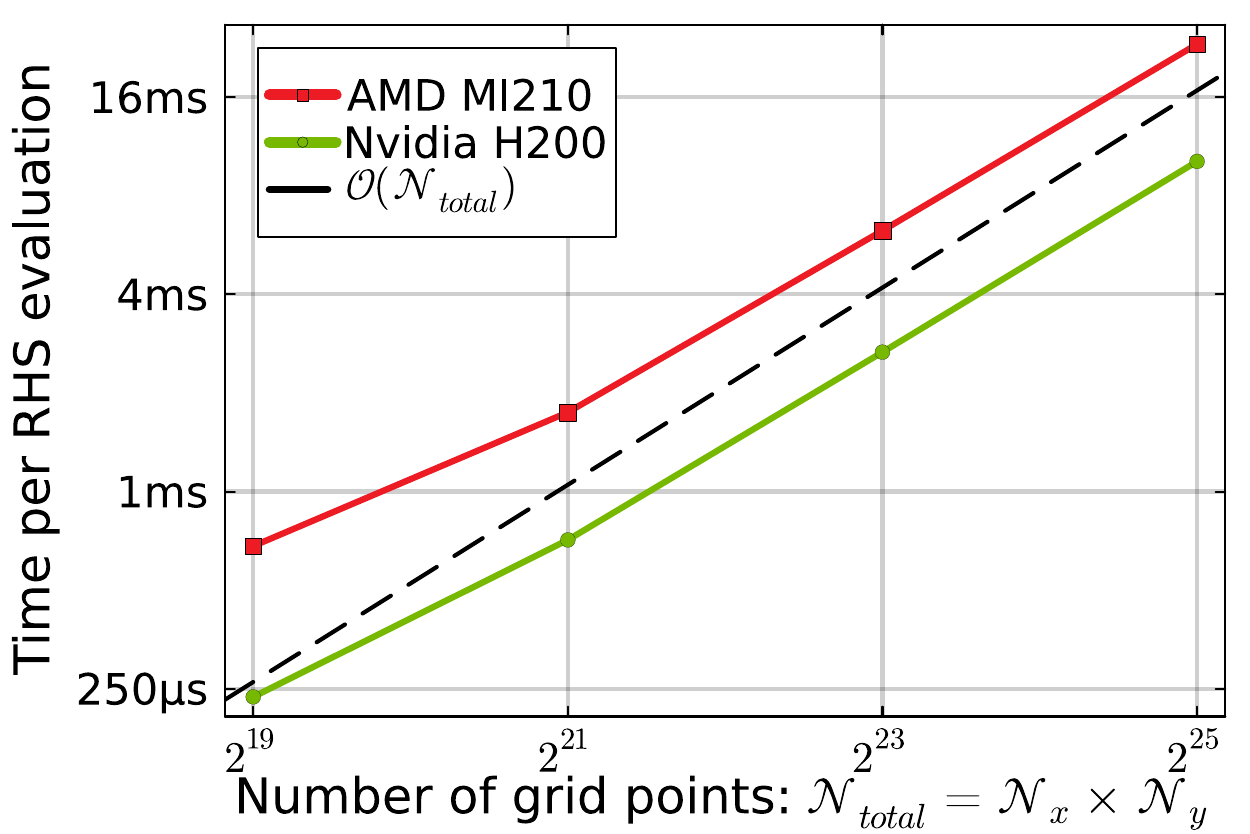}
  \caption{Average time per right-hand-side evaluation versus problem size for
    the wave over Gaussian obstacle test case on GPU platforms. The time per
    RHS evaluation scales approximately linearly with problem size, as
    indicated by the $\mathcal{O}(\mathcal{N}_{\text{total}})$ reference
    line, demonstrating efficient scaling behavior of our implementation on
    both AMD GPU (Instinct MI210) and NVIDIA GPU (H200). These results are
    from single runs on each platform and provide rough order-of-magnitude
    estimates. The problem size varies from $1024 \times 512 = 2^{19}$ to
    $8192 \times 4096 = 2^{25}$.}
  \label{fig:benchmarking_plot_normalized}
\end{figure}

Note that in all benchmarks, as in all other simulations presented in this
work, we used 64-bit floating-point numbers (\texttt{Float64} in Julia). This
choice was made because we encountered unexpected step size issues with
adaptive time stepping when using 32-bit precision in OrdinaryDiffEq.jl,
specifically related to our implementation. From experience, switching to
32-bit precision on an H200 GPU typically yields an additional speedup of
about a factor of~2.

\subsubsection{Execution Times for All Experiments}
\label{sec:execution_times}

Table~\ref{tab:execution_times} summarizes the approximate wall-clock times
for all numerical experiments presented in this work. The reported times
include all postprocessing, analysis, and plotting.
In particular, the energy conservation test
(Section~\ref{sec:energy_conservation}) involves substantial CPU-side analysis,
which accounts for a significant portion of its total runtime.

The timings illustrate the practical benefit of GPU acceleration: the majority
of experiments complete in under a minute on the NVIDIA H200, while several of
the larger two-dimensional simulations are infeasible on a single CPU core
within a reasonable time frame. The AMD MI210 consistently outperforms the CPUs,
though it does not match the H200 for most experiments. A notable exception is
the manufactured solution convergence study
(Section~\ref{sec:manufactured_solution}), where the MI210 is faster than the
H200; the cause of this inversion was not further investigated. Similarly, the
Favre Froude analysis is slower on the MI210 than on the CPUs, likely due to
scalar GPU-to-host transfers used to monitor the wave position during the
simulation, whose per-access latency is significantly higher on the ROCm
backend than on CUDA.

The long-term cylindrical dam break
(Section~\ref{sec:long_term_dam_break}), which evolves
$\num{100000000}$ grid points to $t = 300$, completes in approximately
38~minutes on the H200, the only platform with sufficient memory to
accommodate this problem. A simulation of this scale and duration
would not be feasible on a single workstation without a high-end GPU.

Beyond enabling large-scale simulations, the reduced turnaround times are
equally valuable during code development, where rapid iteration and debugging
of experiments such as the manufactured solution convergence study would be
impractical without GPU-accelerated execution.

\begin{table}[htb]
  \centering
  \caption{Approximate execution times (in seconds) for all numerical
    experiments, including analysis and plotting.
    Entries marked ``OOM'' indicate an out-of-memory failure.}
  \label{tab:execution_times}
  \sisetup{round-mode=places, round-precision=0}
  \begin{tabular}{l S[table-format=6.1] S[table-format=6.1] S[table-format=6.1] S[table-format=4.1]}
    \toprule
    {Experiment} & {CPU i7 (x86)} & {CPU M4 (ARM)} & {AMD MI210} & {NVIDIA H200} \\
    \midrule
    Hyperbolic Soliton (Fig.~\ref{fig:convergence_solitary_wave})         & 7.6      & 2.7      & 5.6    & 14.8   \\
    1D Solitary Wave (Fig.~\ref{fig:convergence_1d_x})                    & 63.4     & 28.0     & 114.3  & 44.7   \\
    Manufactured solution (Fig.~\ref{fig:man_solution_convergence})       & 10854.4  & 2962.3   & 170.7  & 237.0  \\
    Energy conservation (Fig.~\ref{fig:partial_energy_con})               & {OOM}    & {OOM}    & 691.9  & 384.1  \\
    Dingemans (Fig.~\ref{fig:dingemans_results})                          & 1474.9   & 288.7    & 221.7  & 19.9   \\
    Head-on collision (Fig.~\ref{fig:colliding_waves})                    & 12.0     & 3.6      & 4.6    & 1.6    \\
    Riemann problem (Fig.~\ref{fig:riemann_problem})                      & 15.3     & 4.8      & 8.6    & 3.2    \\
    Favre waves (Fig.~\ref{fig:Favre_waves})                              & 11.8     & 4.6      & 13.0   & 2.8    \\
    Favre waves (Fig.~\ref{fig:Froude_numbers})                           & 140.2    & 40.8     & 330.4  & 9.4    \\
    Reflecting wave, $A = 0.075$ (Fig.~\ref{fig:reflecting_wave_A_0.075}) & 5.1      & 2.5      & 8.3    & 1.8    \\
    Reflecting wave, $A = 0.65$ (Fig.~\ref{fig:reflecting_wave_A_0.65})   & 1.4      & 1.1      & 7.3    & 0.6    \\
    Gaussian obstacle (Fig.~\ref{fig:busto_gaussian})                     & 1174.9   & 268.7    & 30.0   & 12.7   \\
    Semi-circular shoal, case (a) (Fig.~\ref{fig:semi_shoal_case_1})      & 5444.4   & 1218.1   & 162.1  & 53.9   \\
    Semi-circular shoal, case (b) (Fig.~\ref{fig:semi_shoal_case_2})      & 5007.1   & 1010.3   & 118.5  & 46.3   \\
    Semi-circular shoal, case (c) (Fig.~\ref{fig:semi_shoal_case_3})      & 7012.2   & 1466.8   & 167.7  & 65.6   \\
    Cylindrical dam break (Fig.~\ref{fig:dam_break_cylindric})            & 3641.4   & 849.6    & 89.1   & 34.9   \\
    Square dam break (Fig.~\ref{fig:dam_break_square})                    & 3556.3   & 882.6    & 85.5   & 40.6   \\
    Long-term dam break (Fig.~\ref{fig:dam_break_long_over_time}--\ref{fig:dam_break_long_soliton_fit}) 
                                                                          & {OOM}    & {OOM}    & {OOM}  & 2270.0 \\
    \bottomrule
  \end{tabular}
\end{table}

\section{Summary and Conclusions}
\label{sect:conclusions}

In this work, we developed structure-preserving numerical schemes for the
two-dimensional hyperbolic approximation of the SGN equations with variable
bathymetry. Using the framework of summation-by-parts operators and split
forms, we derived semi-discretizations that conserve both total mass and total
energy at the discrete level for both periodic and reflecting boundary
conditions (Section~\ref{sect:energy-conserving}).
The implementation in Julia enabled efficient simulations on both CPU
and GPU platforms and is available in our reproducibility
repository~\cite{repo}.

We tested our numerical methods using a comprehensive set of cases. Convergence
studies using solitary wave solutions of the SGN equations
(Section~\ref{sec:convergence_1d_soliton}) and manufactured solutions
(Section~\ref{sec:manufactured_solution}) confirmed second-order spatial
accuracy for all variables. We numerically verified that the spatial
semi-discretization conserves energy to machine precision, as predicted by the
theoretical analysis. Comparisons with experimental data and
numerical reference solutions from the literature demonstrated good agreement
across a wide range of benchmarks, including wave propagation over a submerged
bar (Section~\ref{sec:dingemans}), head-on collision of solitary waves
(Section~\ref{sec:colliding_waves}), Riemann problems
(Section~\ref{sec:riemann_problem}), Favre waves
(Section~\ref{sec:Favre_waves}), reflection from a vertical wall
(Section~\ref{sec:reflecting_wall}), solitary wave over a Gaussian obstacle
(Section~\ref{sec:gaussian_obstacle}), periodic waves over a semi-circular
shoal (Section~\ref{sec:semi_shoal}), and two-dimensional dam break problems
(Sections~\ref{sec:dam_break} and~\ref{sec:long_term_dam_break}). Performance benchmarks
(Section~\ref{sec:performance}) confirmed that GPU acceleration is essential for
large-scale simulations, with substantial speedups on modern accelerators.

The present study also revealed the inherent limitations of the SGN framework.
In the Favre wave simulations (Section~\ref{sec:Favre_waves}) with large Froude
numbers ($\text{Fr} \gtrsim 1.25$), the numerical solution crashed due to wave
steepening that led to breaking, a physical regime not captured by the SGN
equations without wave-breaking modeling.

Several avenues exist for extending this work. Systematic investigations of
long-time soliton propagation errors using relaxation Runge--Kutta (RRK) schemes
would be valuable. Additionally, split-form discretizations that conserve the
total momentum for flat bathymetry and periodic boundary conditions could be investigated.

\appendix

\section*{Acknowledgments}

CW, VM, and HR were supported by the Deutsche Forschungsgemeinschaft
(DFG, German Research Foundation, project numbers 513301895 and 528753982
as well as within the DFG priority program SPP~2410 with project number 526031774)
and the Daimler und Benz Stiftung (Daimler and Benz foundation,
project number 32-10/22).
MR is a member of the Cardamom team, Inria at University of Bordeaux.

\section*{Data Availability}

The data that support the findings of this study are openly available in our reproducibility repository at
\href{https://github.com/cwittens/2025_Hyp_SGN_2D}{github.com/cwittens/2025\_Hyp\_SGN\_2D} and archived at \href{https://doi.org/10.5281/zenodo.18175148}{doi.org/10.5281/zenodo.18175148}.

\printbibliography

\end{document}